\documentclass[a4paper,11pt,leqno]{article}
\usepackage{geometry}
\usepackage{tikz,tikz-cd,amsmath,amsfonts,amssymb,graphicx}
\usetikzlibrary{matrix,arrows,decorations.pathmorphing}
\usepackage[latin1]{inputenc}
\usepackage{amsthm}
\usepackage[autostyle]{csquotes} 
\usepackage{enumerate}
\usepackage[mathscr]{euscript}
\usepackage{mathtools}
\usepackage{hyperref}
\usepackage{url}
\usepackage{scalerel}
\usepackage{dsfont}

\usepackage[noabbrev,nameinlink,capitalize]{cleveref} 
\numberwithin{equation}{section}
\theoremstyle{definition}

\newtheorem{theorem}[equation]{Theorem}
\newtheorem{corollary}[equation]{Corollary}
\newtheorem{definition}[equation]{Definition}
\newtheorem{lemma}[equation]{Lemma}
\newtheorem{example}[equation]{Example}
\newtheorem{proposition}[equation]{Proposition}
\newtheorem{remark}[equation]{Remark}
\newtheorem{construction}[equation]{Construction}
\makeatletter
\newcommand{\colim@}[2]{%
	\vtop{\m@th\ialign{##\cr
			\hfil$#1\Operator@font colim$\hfil\cr
			\noalign{\nointerlineskip\kern1.5\ex@}#2\cr
			\noalign{\nointerlineskip\kern-\ex@}\cr}}%
}
\newcommand{\colim}{%
	\mathop{\mathpalette\colim@{\rightarrowfill@\scriptscriptstyle}}\nmlimits@
}
\renewcommand{\varprojlim}{%
	\mathop{\mathpalette\varlim@{\leftarrowfill@\scriptscriptstyle}}\nmlimits@
}
\renewcommand{\varinjlim}{%
	\mathop{\mathpalette\varlim@{\rightarrowfill@\scriptscriptstyle}}\nmlimits@
}
\makeatother

\let\orignewcommand\newcommand  % store the original \newcommand
\let\newcommand\providecommand  % make \newcommand behave like \providecommand
\RequirePackage{verse}
\let\newcommand\orignewcommand  % use the original `\newcommand` in future

\newcommand{\cate}[1]{\mathscr{#1}}
\newcommand{\s}[1]{\mathbb{S}_{\scalebox{1}{$\scriptscriptstyle #1$}}}

\newcommand{\Sh}[2]{\cate{S}\text{hv}({#1}; {#2})}
\newcommand{\KSh}[2]{\cate{S}\text{hv}_{\cate{K}}({#1}; {#2})}
\newcommand{\Shsp}[1]{\cate{S}\text{hv}({#1})}
\newcommand{\coSh}[2]{\C\text{o}\cate{S}\text{hv}({#1}; {#2})}
\newcommand{\KcoSh}[2]{\C\text{o}\cate{S}\text{hv}_{\cate{K}}({#1}; {#2})}
\newcommand{\sHom}[3]{\text{\underline{Hom}}_{\scalebox{1}{$\scriptscriptstyle #1$}}(#2, #3)}
\newcommand{\Hom}[3]{\text{Hom}_{\scalebox{1}{$\scriptscriptstyle #1$}}(#2, #3)}
\newcommand{\Op}[1]{\cate{U}(#1)}
\newcommand{\cpt}[1]{\cate{K}(#1)}
\newcommand{\Sec}[2]{\Gamma(#1; #2)}
\newcommand{\KSec}[3]{\Gamma_{\scalebox{1}{$\scriptscriptstyle #3$}}(#1; #2)}
\newcommand{\cSec}[2]{\Gamma_{\scalebox{1}{$\scriptscriptstyle c$}}(#1; #2)}

\newcommand{\pf}[1]{#1_{\ast}}
\newcommand{\pb}[1]{#1^{\ast}}
\newcommand{\pfp}[1]{#1_{!}}
\newcommand{\pbp}[1]{#1^{!}}
\newcommand{\pfs}[1]{#1_{\sharp}}

\newcommand{\Cocont}{\C\mathrm{ocont}}
\DeclareMathOperator{\Fun}{Fun}
\DeclareMathOperator{\Pro}{Pro}
\DeclareMathOperator{\opposite}{op}
\DeclareMathOperator{\lex}{lex}
\newcommand{\Funlex}{\Fun^{\lex}}
\newcommand{\op}{^{\opposite}}
\newcommand{\C}{\cate{C}}
\newcommand{\D}{\cate{D}}
\newcommand{\E}{\cate{E}}
\newcommand{\X}{\cate{X}}
\newcommand{\Y}{\cate{Y}}
\newcommand{\Top}{\cate{T}\mathrm{op}}
\newcommand{\Cat}{\cate{C}\mathrm{at}}

\newcommand{\attrib}[1]{%
	\nopagebreak{\raggedleft\footnotesize #1\par}}

\providecommand{\infcat}{\Cat_{\scalebox{1}{$\scriptscriptstyle \infty$}}}
\providecommand{\spaces}{\cate{S}}
\providecommand{\spectra}{\cate{S}\text{p}}
\providecommand{\infsusp}{\Sigma^{\scalebox{1}{$\scriptscriptstyle \infty$}}}
\providecommand{\infloop}{\Omega^{\scalebox{1}{$\scriptscriptstyle \infty$}}}

\newsavebox{\pullback}
\sbox\pullback{%
	\begin{tikzpicture}%
		\draw (0,0) -- (1ex,0ex);%
		\draw (1ex,0ex) -- (1ex,1ex);%
\end{tikzpicture}}
\newsavebox{\pushout}
\sbox\pushout{%
	\begin{tikzpicture}%
		\draw (0ex,0ex) -- (0ex,1ex);%
		\draw (0ex,1ex) -- (1ex,1ex);%
\end{tikzpicture}}
\tikzset{%
	symbol/.style={%
		draw=none,
		every to/.append style={%
			edge node={node [sloped, allow upside down, auto=false]{$#1$}}}
	}
}
{}
\let\o\circ

\begin{document}
	\title{The six operations in topology}
	\author{Marco Volpe\footnote{University of Regensburg, Universit\"atsstrasse 31, Regensburg. email: marco.volpe@ur.de. The author was supported by the SFB 1085 (Higher Invariants) in Regensburg,
			Germany, funded by the DFG (German Science Foundation).
	}}
	\maketitle
	\begin{abstract}
		In this paper we show that the six functor formalism for sheaves on locally compact Hausdorff topological spaces, as developed for example in Kashiwara and Schapira's book \textit{Sheaves on Manifolds}, can be extended to sheaves with values in any closed symmetric monoidal $\infty$-category which is stable and bicomplete. Notice that, since we do not assume that our coefficients are presentable or restrict to hypercomplete sheaves, our arguments are not obvious and are substantially different from the ones explained by Kashiwara and Schapira. Along the way we also study locally contractible geometric morphisms and prove that, if $f:X\rightarrow Y$ is a continuous map which induces a locally contractible geometric morphism, then the exceptional pullback functor $\pbp{f}$ preserves colimits and can be related to the pullback $\pb{f}$. At the end of our paper we also show how one can express Atiyah duality by means of the six functor formalism.
	\end{abstract}
	\tableofcontents
	\section{Introduction}
	\settowidth{\versewidth}{Tell all the truth but tell it slant --}
	\begin{small}
		\begin{verse}[\versewidth]
			\,\,\,\,\,\,\,\,\,\,\,\,\,\,\,\,\,\,\,\,\,\,\,\,\,\,\,\,\,\,\,\,\,\,\,\,\,\,\,\,\,\,\,\,\,\,\,\,\,\,\,\,\,\,\,\,\,\,\,\,\,\,\,\,\,\, Tell all the truth but tell it slant -- \\
			\,\,\,\,\,\,\,\,\,\,\,\,\,\,\,\,\,\,\,\,\,\,\,\,\,\,\,\,\,\,\,\,\,\,\,\,\,\,\,\,\,\,\,\,\,\,\,\,\,\,\,\,\,\,\,\,\,\,\,\,\,\,\,\,\,\, Success in Circuit lies \\
			\,\,\,\,\,\,\,\,\,\,\,\,\,\,\,\,\,\,\,\,\,\,\,\,\,\,\,\,\,\,\,\,\,\,\,\,\,\,\,\,\,\,\,\,\,\,\,\,\,\,\,\,\,\,\,\,\,\,\,\,\,\,\,\,\,\, Too bright for our infirm Delight \\
			\,\,\,\,\,\,\,\,\,\,\,\,\,\,\,\,\,\,\,\,\,\,\,\,\,\,\,\,\,\,\,\,\,\,\,\,\,\,\,\,\,\,\,\,\,\,\,\,\,\,\,\,\,\,\,\,\,\,\,\,\,\,\,\,\,\, The Truth's superb surprise \\
			\,\,\,\,\,\,\,\,\,\,\,\,\,\,\,\,\,\,\,\,\,\,\,\,\,\,\,\,\,\,\,\,\,\,\,\,\,\,\,\,\,\,\,\,\,\,\,\,\,\,\,\,\,\,\,\,\,\,\,\,\,\,\,\,\,\, As Lightning to the Children eased \\
			\,\,\,\,\,\,\,\,\,\,\,\,\,\,\,\,\,\,\,\,\,\,\,\,\,\,\,\,\,\,\,\,\,\,\,\,\,\,\,\,\,\,\,\,\,\,\,\,\,\,\,\,\,\,\,\,\,\,\,\,\,\,\,\,\,\, With explanation kind \\
			\,\,\,\,\,\,\,\,\,\,\,\,\,\,\,\,\,\,\,\,\,\,\,\,\,\,\,\,\,\,\,\,\,\,\,\,\,\,\,\,\,\,\,\,\,\,\,\,\,\,\,\,\,\,\,\,\,\,\,\,\,\,\,\,\,\, The Truth must dazzle gradually \\
			\,\,\,\,\,\,\,\,\,\,\,\,\,\,\,\,\,\,\,\,\,\,\,\,\,\,\,\,\,\,\,\,\,\,\,\,\,\,\,\,\,\,\,\,\,\,\,\,\,\,\,\,\,\,\,\,\,\,\,\,\,\,\,\,\,\, Or every man be blind -- 
		\end{verse}
		\attrib{Emily Dickinson}
	\end{small}
	
	\medskip
	\medskip

	One of the most complete and general reference dealing with the \textit{six functor formalism} for sheaves on locally compact Hausdorff topological spaces is Masaki Kashiwara and Pierre Schapira's seminal book \textit{Sheaves on Manifolds} \cite{kashiwara1990sheaves}. However, for technical reasons related to the construction of derived functors, the authors there restrict themselves to bounded derived categories of sheaves of $R$-modules, where $R$ is assumed to have finite global dimension. From a modern perspective, considering $\infty$-categorical enhancements of derived categories, this can be regarded as the full subcategory of sheaves with values in $D(R)$ spanned by bounded objects. 
		
	%For many applications though, one would like to be able to consider \textit{non-hypercomplete} sheaves with values in \textit{unbounded} derived categories of \textit{any} ring.
	Using Spaltenstein's methods \cite{spaltenstein1988resolutions}, it possible to work with unbounded derived categories of sheaves. These unbounded derived categories agree with (the homotopy category of) hypercomplete sheaves with values in $D(R)$. However, the \textit{proper basechange theorem} is known to fail in that context. To fix this, one has to work with non-hypercomplete sheaves with values in $D(R)$. On the other hand, more recent papers such as \cite{robalo2018lemma}, \cite{jin2024brane}, \cite{jin2024hamiltonian} and \cite{jin2020microlocal}, justify the need of even further generalizations to sheaves of modules over \textit{ring spectra}, in order to apply the power of six functors to generalized cohomology theories. The work of Voevodsky on \textit{stable motivic homotopy theory} has provided an analog of such constructions in the world of algebraic geometry (see \cite{cisinski2019triangulated} for a textbook source on the subject), whereas topologists have succeded only partially in this direction by introducing \textit{parametrized spectra} (see \cite{may2006parametrized}), which correspond to \textit{locally constant} sheaves of spectra. In this paper we exploit the power of the now established theory of $\infty$-categories (as developed for example in \cite{lurie2009higher} or \cite{cisinski2019higher}) to extend the six functor formalism on locally compact Hausdorff spaces to a much broader setting.%, which subsumes all the aforementioned.
	
	Let $f: X\rightarrow Y$ be a continuous map between locally compact Hausdorff topological spaces, and let $\C$ be any stable bicomplete (i.e. complete and cocomplete) $\infty$-category. For any such map, we construct adjunctions
	$$
	\begin{tikzcd}
		\Sh{Y}{\C}\ar[r,bend left,"\pb{f}_{\C}",""{name=A, below}] & \Sh{X}{\C}\ar[l,bend left,"\pf{f}^{\C}",""{name=B,above}] \ar[from=A, to=B, symbol=\dashv]
	\end{tikzcd}
	\,\,\,\,\,\,
	\begin{tikzcd}
		\Sh{X}{\C}\ar[r,bend left,"\pfp{f}^{\C}",""{name=A, below}] & \Sh{Y}{\C}.\ar[l,bend left,"\pbp{f}_{\C}",""{name=B,above}] \ar[from=A, to=B, symbol=\dashv]
	\end{tikzcd}
	$$  
	To do this, we make use of Lurie's covariant Verdier duality equivalence \cite[Theorem 5.5.5.1]{lurie2017higher}
	$$\begin{tikzcd}
		\mathbb{D}_{\C} : \Sh{X}{\C}\ar[r, "\simeq"] & \coSh{X}{\C}
	\end{tikzcd}$$
	where the target is the $\infty$-category of $\C$-valued \textit{cosheaves} on $X$, i.e. $\Sh{X}{\C\op }\op $.  The adjunctions above are then defined to satisfy natural equivalences
	$$\mathbb{D}_{\C}\pfp{f}^{\C}\simeq(\pf{f}^{\C\op })\op \mathbb{D}_{\C} \,\,\,\,\,\, \mathbb{D}_{\C}\pbp{f}_{\C}\simeq(\pb{f}_{\C\op })\op \mathbb{D}_{\C}.$$ 
	Notice that, since we do not require $\C$ to be presentable, it is not a priori clear why a sheafification functor should exist, for presheaves taking values in $\C$ (see the discussion in \Cref{nosheafification}). As a consequence, the existence of $\pb{f}_{\C}$ is not at all obvious, and so we will have to work a bit harder than one might expect. Nevertheless, even though the presentability assumption will be enough for applications, we would like to point out that our efforts to make the results in this paper as general as possible are not vain, and actually lead to many advantages. First of all, with our definition of $\pfp{f}^{\C}$, it is basically immediate that the lower shrieks are functorial with respect to compositions of continuous maps (see \cref{functpfp}). Moreover, by working with a class of coefficients closed under the operation of passing to opposite $\infty$-categories, we do not break the symmetry which comes from Verdier duality. Specifically, whenever we prove some result involving the functors $\pf{f}^{\C}$ and $\pb{f}_{\C}$ that is true for all $\C$ stable and bicomplete, we immediately obtain a dual theorem involving the functors $\pfp{f}^{\C}$ and $\pbp{f}_{\C}$, and viceversa. Another way to put it is that our formalism applies with no distinctions to sheaves or cosheaves: this will be used in a follow-up paper (see \cite{volpe2025verdier}) in which we will prove a duality theorem for constructible sheaves on a conically smooth stratified space, as we will need to extend some results about constructible sheaves, such as homotopy invariance (see \cite{haine2020homotopy}) or the exodromy equivalence (see \cite{porta2022topological}), to constructible cosheaves. It is also worth noticing that, even if one would restrict to presentable coefficients, it would nevertheless be desirable to have formulas such as $\pfp{f}^{\spectra}\otimes\C\simeq\pfp{f}^{\C}$, and to get this one would still need to verify everything we prove in Section 5.
	\par 
	We try to outline the key ingredients in our paper that allow us to work with non-presentable coefficients. As we explained above, the main difficulty lies in showing the existence of the pullback functor $\pb{f}_{\C}$. The main tool we employ to carry out this purpose is Lurie's tensor product of cocomplete $\infty$-categories as defined in \cite[4.8.1]{lurie2017higher}: in particular, a property of this tensor product that we will use over and over is that it preserves adjunctions between cocontinuous functors (see \Cref{tensoringadj}). We show in \Cref{tensstcocont} that it restricts to a symmetric monoidal structure on $\Cocont^{st}_{\infty}$ (i.e the $\infty$-category of stable cocomplete $\infty$-categories with cocontinuous funtors between them) and use it to formulate the following theorem. 
	\begin{theorem}[\cref{shsptensCtoshC}]\label{shsptensCtoshCintr}
		Let $\C$ be a stable bicomplete $\infty$-category, and let $X$ be a locally compact Hausdorff space. Then there is an equivalence
		$$\Sh{X}{\spectra}\otimes \C\simeq\Sh{X}{\C}$$
		where $\otimes$ on the left-hand side denotes the tensor product of stable cocomplete $\infty$-categories. 
	\end{theorem}
	\cref{shsptensCtoshCintr} will play a crucial role in what follows, because it will allow us to reduce a lot of arguments involving cocontinous functors to the case of sheaves of spectra. To prove \cref{shsptensCtoshCintr}, we start by observing in \Cref{coshsttens} that the model of $\cate{K}$-sheaves (see \cite[Theorem 7.3.4.9]{lurie2009higher}) implies that $\Sh{X}{\spectra}$ is a strongly dualizable object in  $\Cocont^{st}_{\infty}$, where $\spectra$ denotes the $\infty$-category of spectra (see also \cite[Proposition 21.1.7.1]{lurie2016spectral}). Hence, for $\C$ any stable and cocomplete $\infty$-category, we get an equivalence
	$$\coSh{X}{\spectra}\otimes \C\simeq\coSh{X}{\C} .$$
	Combined with Verdier duality, this gives \cref{shsptensCtoshCintr}.
	
	Having \cref{shsptensCtoshCintr} at hand, the question of constructing $\pb{f}_{\C}$ can be reduced to the case of sheaves of spectra, where we can directly use the existence of sheafification by the presentability of $\spectra$. More precisely, we first observe that any map $f$ can be factored as the composition of a closed immersion, an open immersion and a proper map (see factorization (\ref{factorization})), and then prove the existence of a left adjoint to $\pf{f}^{\C}$ in these separate cases. When $f$ is an open immersion this is done easily in \cref{openrestrtens}, and the only non-trivial part constists in the following theorem. 
	\begin{theorem}[\Cref{properpsfcocont}, \Cref{proppshtens}]\label{properpfintr}
		Let $f: X\rightarrow Y$ be a proper map between locally compact Hausdorff topological spaces. Then the pushforward 
		$$\pf{f}^{\C} : \Sh{X}{\C}\rightarrow\Sh{Y}{\C}$$
		preserves colimits. Furthermore, there is a commutative square
		$$
		\begin{tikzcd}
			\Sh{X}{\spectra}\otimes\C\ar[r, "\simeq"]\ar[d, "{\pf{f}^{\spectra}\otimes\C}"] & \Sh{X}{\C}\ar[d, "{\pf{f}^{\C}}"] \\
			\Sh{Y}{\spectra}\otimes\C\ar[r, "\simeq"] & \Sh{Y}{\C}.
		\end{tikzcd}
		$$
	\end{theorem}
	The proof of \cref{properpfintr} essentially consists of providing a convenient description of $\pf{f}^{\C}$ through the model of $\cate{K}$-sheaves, which is easily seen to preserve colimits and to be compatible with Verdier duality. We then achieve our final goal observing that, since $\pf{f}^{\spectra}\otimes\C$ admits a left adjoint of the form $\pb{f}_{\spectra}\otimes\C$, the same is true for $\pf{f}^{\C}$. In particular by taking $f$ to be the projection $X\rightarrow\ast$, we see that the global section functor
	$$\Sh{X}{\C}\rightarrow\C$$
	admits a left adjoint. As a consequence, using the results in \cite[6.7]{cisinski2019higher}, we show in \Cref{sheafification} that the inclusion of $\Sh{X}{\C}$ in $\C$-valued presheaves on $X$ admits a left adjoint.
	
	The discussion above involves only the four functors $\pf{f}^{\C}$, $\pb{f}_{\C}$, $\pfp{f}^{\C}$ and $\pbp{f}_{\C}$, but what about the other two? Our first observation is that, a priori, there is no need to require our $\infty$-category of coefficients to have a symmetric monoidal structure to make sense of things like projection formulas or K\"{u}nneth formulas. To be more precise, one can show that there is a functor 
	$$
	\begin{tikzcd}[row sep = tiny]
		\Sh{X}{\C}\times\Sh{Y}{\D}\ar[r] & \Sh{X\times Y}{\spectra}\otimes(\C\otimes\D) \\
		(F, G)\ar[r, mapsto] & F\boxtimes G
	\end{tikzcd}
	$$
	which preserves colimits in both variables and induces an equivalence
	$$\Sh{X}{\C}\otimes\Sh{Y}{\D}\simeq\Sh{X\times Y}{\spectra}\otimes(\C\otimes\D) .$$
	Taking $X = Y$ and composing with $\pb{\Delta}_{\spectra}\otimes(\C\otimes\D)$, where $\Delta : X\hookrightarrow X\times X$ is the diagonal embedding, we get a variablewise colimit preserving functor denoted as
	$$
	\begin{tikzcd}[row sep = tiny]
		\Sh{X}{\C}\times\Sh{Y}{\D}\ar[r] & \Sh{X}{\spectra}\otimes(\C\otimes\D) \\
		(F, G)\ar[r, mapsto] & F\otimes G
	\end{tikzcd}
	$$
	(see \Cref{dfntens,fromprestostbic} for more details). For this kind of tensor product of sheaves, we prove the following formulas.
	\begin{theorem}[\Cref{monpb}, \cref{proj}, \cref{kunneth}]
		Let $f:X\rightarrow X'$, $g:Y\rightarrow Y'$ be continuous maps between locally compact Hausdorff spaces. Let $\C$ and $\D$ be stable and bicomplete $\infty$-categories, and let Then we have the following functorial identifications 
		\begin{align*}
			\pb{f}_{\C\otimes\D}(-\otimes -) &\simeq \pb{f}_{\C}(-)\otimes\pb{f}_{\D}(-) \\ 
			\pfp{f}^{\C\otimes\D}(-\otimes\pb{f}_{\D}-) &\simeq\pfp{f}^{\C}(-)\otimes (-) \\
			\pfp{(f\times g)}^{\C\otimes\D}(-\boxtimes -) &\simeq \pfp{f}^{\C}(-)\boxtimes\pfp{g}^{\D}(-).
		\end{align*}
	\end{theorem}
	In particular, when $\C$ admits a symmetric monoidal structure whose tensor preserves colimits in both variables, one obtains a cocontinuous functor
	$$\Sh{X}{\spectra}\otimes(\C\otimes\C)\rightarrow\Sh{X}{\C}$$
	whose composition with (\ref{tensshCD}) induces a symmetric monoidal structure on $\Sh{X}{\C}$: this way we can deduce all the analogous formulas in the monoidal setting. If $\C$ is also closed, we deduce their dual versions involving the internal homomorphism functor (see \Cref{tenstomon}).
	
	\begin{remark}
		Our results imply that, for any $\C$ stable and bicomplete, the contravariant functor $X\mapsto\Sh{X}{\C}$, $f\mapsto\pb{f}_{\C}$ extends to a six functor formalism in the sense of \cite{mann2022}. See \cref{mannsix} for a more detailed discussion on this matter.
	\end{remark}
	
	We describe one last advantage of our general rendition of the six functor formalism. In Definition \ref{dfnloccontr} we define \textit{locally contractible geometric morphisms} (see also \cite[Definition 3.2.1]{aizenbud2021relative}). Later, in \Cref{subm}, we specify a vast class of continuous maps between topological spaces called \textit{shape submersions} which induce a locally contractible geometric morphism (see \Cref{smoothproj}). Topological submersions are examples of such morphisms, but our definition is much more general in the sense that it does not force the fibers to be topological manifolds. Another illustrating example to keep in mind is that of the unique map $X\rightarrow\ast$, when $X$ is any CW-complex. An easy implementation of our machinery generalizes \cite[Proposition 3.3.2]{kashiwara1990sheaves} and \cite[Section 5]{verdier1965dualite} beyond the case of submersive maps.
	\begin{theorem}[\cref{smoothpb}]
		Let $f : X\rightarrow Y$ be a continuous map between locally compact Hausdorff spaces which induces a locally contractible geometric morphism. Let $\C$ be a stable and bicomplete $\infty$-category. Then $\pbp{f}_{\C}$ admits a right adjoint and we have a formula
		$$\pbp{f}_{\C}(-)\otimes\pb{f}_{\D}(-)\simeq\pbp{f}_{\C\otimes\D}(-\otimes -).$$
	\end{theorem}
	Then, inspired by parallel results in motivic homotopy theory, we conclude our paper by formulating and proving a relative version of Atiyah duality.
	
	\begin{theorem}[\cref{relatiyahdual}]
		Let $f : X\rightarrow Y$ be a proper submersion between smooth manifolds, and let $T_f\rightarrow X$ be the relative tangent bundle of $f$. Denote by $\s{X}\in\Sh{X}{\spectra}$ the constant sheaf at the sphere spectrum, and by $\text{Th}(-T_f)$ the \textit{Thom spectrum sheaf} of the virtual vector bundle $-T_f$. Then $\pfs{f}(\s{X})$ is strongly dualizable with dual  $\text{Th}(-Tf)$.
	\end{theorem}
	
	%-------------------------------------------------------------------%
	%  Linear overview                                                  %
	%-------------------------------------------------------------------%
	
	\subsection{Linear overview}
	
	We now give a linear overview of the contents of our paper.
	\par
	In Section 2 we recall the definition of Lurie's tensor product of cocomplete $\infty$-categories and prove some of its basic properties. In particular, we will interpret the results in \cite[6.7]{cisinski2019higher} in terms of this tensor product in \Cref{preshdual}, show that it preserves the property of being stable in \Cref{tensstcocont}, and show that compactly generated stable $\infty$-category is a strongly dualizable object in the symmetric monoidal $\infty$-category $\Cocont^{st}_{\infty}$ in \Cref{cptlygenstdual}. Afterwards we recall the definition of sheaves and cosheaves with values in a general $\infty$-category and explain how Lurie's tensor product can be used to conveniently describe $\Sh{X}{\C}$ at least when $\C$ is presentable. Most of the results in this section are not original, but we still felt the necessity to spend some time writing them up to make our discussion as self contained and reader-friendly as possible. 
	\par 
	In Section 3 we define for any geometric morphism of $\infty$-topoi $\X\rightarrow\Y$ the \textit{relative shape} $\Pi_{\infty}^{\Y}(\X)$ as a pro-object of $\Y$. We describe explicitly in \Cref{functrelsh} how this construction can be enhanced to a functor $$\Pi_{\infty}^{\Y} : \Top_{/\Y}\rightarrow\Pro(\Y)$$ and, even more, to a lax natural transformation between functors $\Top\op \rightarrow\infcat$ (see \Cref{laxshape}). These coherent structures with which we equip the shape will be used to prove easily that the shape is homotopy invariant in \Cref{hmtpyinv}. In \Cref{natthom}, we use \Cref{laxshape} also to show that the Thom spectrum gives a natural transformation of sheaves of $\mathbb{E}_{\infty}$-spaces. Later we define locally contractible geometric morphisms and give a characterization in \Cref{locconstsheqdef} which mimics the one in \cite[C3.3]{johnstone2002sketches} for the locally connected case. We also show that, when $f$ is a geometric morphism induced by a continuous map of topological spaces, the property of being locally contractible is checked more easily. Then we define shape submersions and prove in \Cref{smoothbc} a base change formula which will imply that they induce locally contractible geometric morphisms (see \Cref{smoothproj}).
	\par 
	In Section 4 we follow the approach of \cite{khan2019morel} to obtain the localization sequences associated to a decomposition of a topological space into an open subset and its closed complement. Also the results here are not so new but, after Section 5, they will imply that there is a recollement of $\Sh{X}{\C}$ associated to any open-closed decomposition of $X$ whenever $\C$ is stable and bicomplete, while this was previously known only for $\C$ presentable. 
	\par 
	Section 5 is devoted to Verdier duality, and how it can be used to show that the pushforward $\pf{f}^{\C}$ admits a left adjoint for any $\C$ stable and bicomplete in the way we have sketched at the beginning of the introduction.
	\par 
	In Section 6 we develop the six functor formalism: as usual, we prove base change (\Cref{basechange}), projection (\Cref{proj}) and K\"unneth (\Cref{kunneth}) formulas for $\pfp{f}^{\C}$, and discuss the properties of $\pbp{f}_{\C}$ when $f$ is a shape submersion in \Cref{smoothpb}.
	\par 
	At last, in Section 7 we show how the six functor formalism can be used to express a relative version of Atiyah duality for any proper submersion between smooth manifolds.
	
	%-------------------------------------------------------------------%
	%  Acknowledgements                                                 %
	%-------------------------------------------------------------------%
	
	\subsection{Acknowledgements}
	
	First of all, I would like to express my gratitude to my supervisor Denis-Charles Cisinski, who has encouraged me to work on this subject and has accompanied me during the whole process of writing the paper, providing deep insights and help whenever I felt lost. Secondly, I want to thank Andrea Gagna and Edoardo Lanari for being patient and kind enough to answer all my (sometimes meaningless) questions over the past two years. I would also like to thank Denis Nardin, Peter Haine, George Raptis, Benedikt Preis, Marc Hoyois, Clark Barwick, Guglielmo Nocera, Ivan Di Liberti and Nicola Di Vittorio for showing interest in my work, for offering their support and for pushing me to keep learning new mathematics. This paper is a part of my PhD thesis.
	
	%-------------------------------------------------------------------%
	%-------------------------------------------------------------------%
	%  Sheaves and tensor products                                      %
	%-------------------------------------------------------------------%
	%-------------------------------------------------------------------%
	
	\section{Sheaves and tensor products}

	The goal of this section is twofold. First, we are going to introduce Lurie's tensor product of cocomplete $\infty$-categories as defined in \cite{lurie2017higher}. Secondly, we will recall the definition of sheaves and cosheaves with values in general $\infty$-categories. We want devote some time discussing Lurie's tensor product, and its relation to $\infty$-categories of sheaves, because in the following sections it will prove to be an extremely convenient tool to describe some $\infty$-categories of sheaves and functors between them. Using Lurie's tensor product we will be able to produce a vast class of \textit{essential geometric morphisms}. We will easily extend some results regarding sheaves of spaces to sheaves with values in any presentable $\infty$-category. Lastly, we will prove the existence of a sheafification functor when the $\infty$-category of coefficients is stable and bicomplete with no presentability assumption, and construct the full six functor formalism in this setting. Most of the results in this section are not at all original and can be found for example in \cite{cisinski2019higher}, in \cite{lurie2017higher}, or are already well known. 
	\subsection{Tensor product of cocomplete $\infty$-categories}
	For the whole section we will fix two universes $\textbf{V}$ and $\textbf{U}$ such that $\textbf{V}$ is $\textbf{U}$-small, and denote by $\Cocont_{\infty}$ the $\infty$-category of \textbf{U}-small $\infty$-categories admitting \textbf{V}-small colimits, with \textbf{V}-cocontinuous functors between them. For short, we will call an object of $\Cocont_{\infty}$ a cocomplete $\infty$-category and, for any two cocomplete $\infty$-categories $\C$ and $\D$ we will denote by $\Fun_!(\C, \D)$ the $\infty$-category of functors preserving \textbf{V}-small colimits. 
	\par
	Let $\C$ and $\D$ be cocomplete. Recall that, by \cite[5.3.6]{lurie2009higher}, there exists a cocomplete $\infty$-category, denoted by $\C\otimes\D$, and a functor
	$$\boxtimes:\C\times\D\rightarrow\C\otimes\D,$$
	which preserves colimits in both variables and such that precomposing with $\boxtimes$ gives an equivalence
	\begin{equation}\label{tensor}
		\Fun_{!\times !}(\C\times\D, \E) \simeq \Fun_!(\C\otimes\D, \E)
	\end{equation}
	functorial on $\C$, $\D$ and $\E$ cocomplete, where $\Fun_{!\times !}$ indicates the $\infty$-category of bifunctors preserving \textbf{V}-small colimits in each variable. More precisely, \cite[Corollary 4.8.1.4]{lurie2017higher} shows that this operation provides $\Cocont_{\infty}$ with the structure of a symmetric monoidal $\infty$-category, and the inclusion of $\Cocont_{\infty}$ in $\infcat$ is lax monoidal, where the latter is equipped with the cartesian monoidal structure. Since we obviously have a functorial equivalence
	$$\Fun_{!\times!}(\C\times\D, \E)\simeq\Fun_!(\C, \Fun_!(\D, \E)),$$ this monoidal structure is closed.  
	\begin{remark}\label{tensoringadj}
		As usual, one may regard $\Cocont_{\infty}$ as a $(\infty, 2)$-category. It follows from (\ref{tensor}) that, for any cocomplete $\infty$-category $\C$, tensoring with $\C$ gives rise to a $2$-functor. An important consequence of this observation is that, since any adjunction is characterized by the classical triangular identities (see for example \cite[Theorem 6.1.23, (v)]{cisinski2019higher}), tensoring with $\C$ preserves adjunctions of cocontinuous functors.
	\end{remark}
	We will now present a list of results about the tensor product of cocomplete $\infty$-categories that will turn out to be very useful later.

	Let $A$ be a small $\infty$-category, $\C$ cocomplete. For any two objects $a\in A$ and $c\in\C$, denote by $a\boxtimes c = \pfp{a}c$ the left Kan extension of $c$ along $a$ (here we are considering $a$ and $c$ as functors $A\op \leftarrow\Delta^0\rightarrow\C)$. Thus we get a functor 
	$$\begin{tikzcd}[row sep = tiny]
		A\times\C \arrow[r, "y_{A/\C}"] & {\Fun(A\op ,\C)} \\
		{(a, c)} \arrow[r, maps to]                 & a\boxtimes c                   
	\end{tikzcd}$$
	which preserves colimits on the $\C$ variable that we will call the \textit{relative Yoneda embedding}. By definition, we have a functorial equivalence
	$$\Hom{}{a\boxtimes c}{F}\simeq\Hom{}{c}{F(a)}$$
	for any $F\in\Fun(A\op , \C)$.
	\begin{remark}
		Recall that in a closed symmetric monoidal $\infty$-category, an object $x$ is \textit{strongly dualizable} if and only if, for any object in $y\in\C$, the canonical map
		\begin{equation}\label{dualizable}
			y\otimes\sHom{}{x}{1}\rightarrow\sHom{}{x}{y}
		\end{equation}
		obtained as adjoint to 
		$$
		\begin{tikzcd}[column sep = small]
			y\otimes\sHom{}{x}{1}\otimes x\ar[r] & y\otimes 1\ar[r, "\simeq"] & y
		\end{tikzcd}
		$$
		is an equivalence. In the case of $\Cocont_{\infty}$, one sees easily that the map (\ref{dualizable}) can be described as induced by
		$$\Fun_!(\C,\spaces)\times\D\simeq\Fun_!(\C, \spaces)\times\Fun_!(\spaces,\D)\rightarrow\Fun_!(\C, \D)$$
		where the last functor is given by composition.
	\end{remark}
	
	\begin{remark}\label{dualretract}
		By the naturality on $x$ of the map (\ref{dualizable}), we see that the full subcategory spanned by strongly dualizable objects is closed under retracts.
	\end{remark}
	
	\begin{remark}\label{comptens}
		Consider the variablewise cocontinuous functor \begin{equation}\label{tenspshbi}
			\Fun(A\op , \spaces)\times\C\rightarrow \Fun(A\op ,\C)
		\end{equation} obtained as the extension by colimits of the relative Yoneda embedding, and denote by $F\boxtimes c$ the image of a pair $(F, c)\in\Fun(A\op , \spaces)\times\C$. This induces a cocontinous functor
		\begin{equation}\label{tenspsh}
			\Fun(A\op , \spaces)\otimes\C\rightarrow \Fun(A\op ,\C).
		\end{equation}
		By definition one has identifications
		$$\Hom{}{F\boxtimes c}{G}\simeq\Hom{}{F}{\Hom{\C}{c}{G(-)}}$$
		functorially on $F\in\Fun(A\op ,\spaces)$, $G\in\Fun(A\op ,\C)$ and $c\in\C$, where the hom-space on the right-hand side is taken on the $\infty$-category of presheaves of $\textbf{U}$-small spaces. 
		\par A very convenient way to model the functor (\ref{tenspshbi}) is as follows. Let $y : \Delta^0\hookrightarrow\spaces$ be the functor picking the terminal object of $\spaces$. Combining the fact that $y$ is fully faithful and \cite[Proposition 6.4.12]{cisinski2019higher}, we have 
		$$\pfp{y}c\o\pfp{a}y \simeq\pfp{a}(\pfp{y}c\o y)\simeq \pfp{a}c $$   
		and hence we get a commutative triangle 
		$$
		\begin{tikzcd}
			\Fun(A\op , \spaces)\times\C\ar[r, "\boxtimes"] \ar[d, "{\text{id}\times\pfp{y}}"] & \Fun(A\op , \C) \\ \Fun(A\op , \spaces)\times\Fun_!(\spaces, \C)\ar[ur, "\o"']
		\end{tikzcd}
		$$
		where the vertical arrow is an equivalence and the diagonal one is given by composition. In particular, one deduces that the functor (\ref{tenspsh}) can be seen as an instance of (\ref{dualizable}). 
	\end{remark}
	For any cocomplete $\infty$-category $\D$, precomposition with $y_{A/\C}$ induces a functor
	$$\Fun(\Fun(A\op ,\C), \D)\rightarrow\Fun(A\times\C, \D)\simeq\Fun(\C, \Fun(A,\D))
	$$
	and since colimits are computed pointwise in functor $\infty$-categories, it restricts to
	\begin{equation}\label{relyon}
		\Fun_!(\Fun(A\op ,\C), \D)\rightarrow\Fun_!(\C, \Fun(A,\D)).
	\end{equation}
	\begin{theorem}\label{preshdual}
		The functor (\ref{relyon}) is an equivalence. In particular, $\Fun(A\op , \spaces)$ is strongly dualizable in the monoidal $\infty$-category $\Cocont_{\infty}$ with dual $\Fun(A, \spaces)$, and thus the functor (\ref{tenspsh}) is an equivalence.
	\end{theorem}
	\begin{proof}
		A complete proof of the first statement can be found in \cite[6.7]{cisinski2019higher}. The main ingredient of the proof is that, by \cite[Lemma 6.7.7]{cisinski2019higher}, any $F\in\Fun(A\op , \C)$ can be written canonically as 
		$$F = \varinjlim_{\scalebox{0.5}{$c\rightarrow F(a$)}}a\boxtimes c$$
		where the colimit is indexed by the Grothendieck construction of the functor $(a,c)\mapsto\Hom{\C}{c}{F(a)}$. Furthermore, even though this indexing $\infty$-category is not small a priori, \cite[Lemma 6.7.5]{cisinski2019higher} proves that it is finally small. From this one may deduce easily the theorem, in a similar spirit to how one proves that $\Fun(A\op , \spaces)$ is the free cocompletion under small colimits of $A$.
		\par
		To prove the last statement, we just observe that we have canonical equivalences
		\begin{align*}
			\Fun_!(\Fun(A\op ,\C), \D)&\simeq\Fun_!(\C, \Fun(A,\D)) \\&\simeq \Fun_!(\C, \Fun_!(\Fun(A\op ,\spaces),\D)) \\
			&\simeq\Fun_!(\Fun(A\op ,\spaces)\otimes\C, \D)
		\end{align*}
		whose composition is given by precomposing with (\ref{tenspsh}), and so we may conclude by \Cref{comptens}.
	\end{proof}
	
	Let $u:A\rightarrow B$ be a functor between small $\infty$-categories, and let $\C$ be any cocomplete $\infty$-category. We will denote by $\pb{u}_{\C}$ ($\pfp{u}^{\C}$) the precomposition with (the left Kan extension along) $u$ for $\C$-valued presheaves.
	
	\begin{corollary}\label{tenskanext}
		Let $u:A\rightarrow B$ be a functor between small $\infty$-categories, and let $\C$ be any cocomplete $\infty$-category. Then we have equivalences $\pfp{u}^{\spaces}\otimes\C\simeq\pfp{u}^{\spaces}$ and $\pb{u}_{\spaces}\otimes\C\simeq\pb{u}_{\C}$ .
	\end{corollary}
	\begin{proof}
		By \Cref{tensoringadj} and \Cref{preshdual}, we have an adjunction $\pfp{u}\otimes\C\dashv\pb{u}\otimes\C$ of cocontinuous functors between $\C$-valued presheaves. By uniqueness of adjoints, it suffices to show that $\pb{u}\otimes\C\simeq\pb{u}$. But this is clear because by \Cref{comptens} we have a commutative square
		$$
		\begin{tikzcd}
			\Fun(B\op , \spaces)\times\C\ar[r, "{\pb{u}\times\text{id}}"]\ar[d, "\boxtimes"] & \Fun(A\op , \spaces)\times\C\ar[d, "\boxtimes"] \\ \Fun(B\op , \C)\ar[r, "{\pb{u}}"] & \Fun(B\op , \C).
		\end{tikzcd}
		$$
	\end{proof}
	Denote by $\Cocont^{pt}_{\infty}$ ($\Cocont^{st}_{\infty}$) the full subcategory of $\Cocont_{\infty}$ spanned by pointed (stable) cocomplete $\infty$-categories.
	
	\begin{lemma}\label{univpropsptsp}
		Let $\C$ be any cocomplete $\infty$-category, and denote by $\emptyset\in\C$ the initial object.  Write $\text{ev}_{S^0}:\Fun_!(\spaces_{\ast},\C)\rightarrow\C$ for the functor given by evaluation at $S^0\in\spaces_{\ast}$. Similarly, write $\text{ev}_{\s{}}:\Fun_!(\spectra,\C)\rightarrow\C$ for the functor given by evaluating at $\s{}\in\spectra$. Then the following statements are true.
		\begin{enumerate}
			\item $\text{ev}_{S^0}$ factors through an equivalence of $\infty$-categories
			$$\Fun_!(\spaces_{\ast},\C)\simeq\C_{/\emptyset}.$$ 
			In particular, if $\C$ is pointed then $\text{ev}_{S^0}$ is an equivalence of $\infty$-categories. 
			\item $\text{ev}_{\s{}}$ factors through an equivalence of $\infty$-categories
			$$\Fun_!(\spectra,\C)\simeq\varprojlim(\cdots\xrightarrow{\Sigma}\C_{/\emptyset}\xrightarrow{\Sigma}\C_{/\emptyset}).$$
			In particular, if $\C$ is stable, $\text{ev}_{\s{}}$ is an equivalence of $\infty$-categories.
		\end{enumerate}
	\end{lemma}
	
	\begin{proof}
		We start by proving statement 1. Notice that, since $\ast\in\spaces$ is terminal, $\spaces_{\ast}$ embeds fully faithfully in $\Fun(\Delta^1,\spaces)$. This embedding admits a left adjoint $L$, explicitly given by $(X\rightarrow Y)\mapsto Y_{+}$. $L$ is therefore a cocontinuous (Bousfield) localization. We therefore obtain a fully faithful functor $\Phi:\Fun_!(\spaces_{\ast},\C)\hookrightarrow\Fun(\Delta^1,\C)$ given by the composition
		$$\Fun_!(\spaces_{\ast},\C)\xhookrightarrow{\pb{L}}\Fun_!(\Fun(\Delta^1,\spaces),\C)\xhookrightarrow{\pb{y}}\Fun(\Delta^1,\C),$$
		where the second functor is given by precomposition with the yoneda embedding, up to identifying $\Delta^1$ with its opposite. Tracing through the definitions, $\Phi$ can be described as sending $F:\spaces_{\ast}\rightarrow\C$ to the arrow $(F(S^0)\rightarrow F(\ast))$. Since $F$ preserves colimits and $\ast\in\spaces_{\ast}$ is an initial object, we see that $\Phi$ factors through $\C_{/\emptyset}$. To conclude our proof, it suffice to show that, for any arrow $f:c\rightarrow\emptyset$, $f$ belongs to the essential image of $\Phi$. One checks that $f\simeq \Phi(F_f)$, where $F_f$ denotes the left kan extension of $\Delta^1\xrightarrow{f}\C$ along $\Delta^1\xrightarrow{S^0\rightarrow\ast}\spaces_{\ast}$. 
		
		We now focus on statement 2. Denote by $\spectra^{fin}$ the $\infty$-category of finite spectra, and by $\spaces_{\ast}^{fin}$ the $\infty$-category of finite pointed spaces. Recall that $\spectra^{fin}$ can be described as the \textit{Spanier-Whitehead category} \cite{adams1974stable, spanier1953first}
		$\spectra^{fin}\simeq\varinjlim(\spaces_{\ast}\xrightarrow{\Sigma}\spaces_{\ast}\xrightarrow{\Sigma}\cdots).$ Here the colimit is taken in the $\infty$-category of small $\infty$-categories with finite colimits and right exact functors between them. Since $\spectra\simeq\text{Ind}(\spectra^{fin})$, the desired result then follows from combining statement 1 with the previous description of $\spectra^{fin}$.
	\end{proof}
	
	\begin{lemma}\label{tensstcocont}
		Let $\D$ be any cocomplete $\infty$-category. Then the following statements are true.
		\begin{enumerate}
			\item If $\C$ is a pointed cocomplete $\infty$-category, then $\C\otimes\D$ is pointed.
			\item If $\C$ is a stable cocomplete $\infty$-category, then $\C\otimes\D$ is stable.
		\end{enumerate}  
		In particular, $\Cocont^{pt}_{\infty}$ (respectively $\Cocont^{st}_{\infty}$) inherits a monoidal structure from $\Cocont_{\infty}$ and the inclusion in $\Cocont_{\infty}$ admits a left adjoint given by tensoring with $\spaces_{\ast}$ (respectively $\spectra$).
	\end{lemma}
	\begin{proof}
		We begin by proving statement 1. First of all, notice that $\Delta^0\otimes\D\simeq\Delta^0$. Recall also that $\C$ is pointed if and only if the zero object $\Delta^0\rightarrow\C$ is simultaneously a right and a left adjoint of the unique functor $\C\rightarrow\Delta^0$. Thus one may tensor these two adjunctions with $\D$ and obtain by \cref{tensoringadj} that $\C\otimes\D$ is pointed.
		\par
		We now turn to proving statement 2, so we assume that $\C$ is stable. Denote by $i_{\epsilon}:\Delta^0\hookrightarrow\Delta^1$ the inclusion corresponding to $\epsilon= 0,1$. Notice that, for any cocomplete pointed $\infty$-category $\E$, one can define a suspension functor as a composition
		$$
		\begin{tikzcd}
			\Sigma_{\E}:\E \arrow[r, "X\mapsto (X\rightarrow0)"] & {\Fun(\Delta^1, \E)} \arrow[r, "\text{cofib}"] & \E.
		\end{tikzcd}
		$$
		Since the first functor preserves colimits and is right adjoint to $\pb{(i_0)}_{\E}$, and the second functor is left adjoint to $\pfp{(i_1)}^{\E}$, by \cref{tensoringadj} and \cref{tenskanext} we see that $\Sigma_{\C\otimes\D}\simeq\Sigma_{\C}\otimes\D$. Since $\Sigma_{\C}$ is an equivalence, one deduces that $\Sigma_{\C\otimes\D}$ must be an equivalence as well, and therefore by \cite[Corollary 1.4.2.27]{lurie2017higher}, one concludes that $\C\otimes\D$ is stable.
		\par
		To prove that last part of the statement, it suffices to show that, for any pointed (stable) and cocomplete $\infty$-category $\C$, the evaluation at $S^0$ (respectively $\s{}$) induces an equivalence $\Fun_!(\spaces_{\ast}, \C)\simeq\C$ (respectively $\Fun_!(\spectra, \C)\simeq\C$), but this follows from \cref{univpropsptsp}. 
		%easily by noticing that $\spaces_{\ast}\simeq\text{Ind}(\spaces^{fin}_{\ast})$ (respectively $\spectra\simeq\text{Ind}(\spectra^{fin})$) and that evaluation at $S^0$ (respectively $\s{}$) induces an equivalence between finitely cocontinuous functors from $\spaces_{\ast}$ (respectively $\spectra^{fin}$) to $\C$ and $\C$.
	\end{proof}
	\begin{remark}\label{relyonstable}
		Let $A$ be any small $\infty$-category and $\C$ any object of $\Cocont^{st}_{\infty}$. By \cref{tensstcocont}, there exist an essentially unique functor $\boxtimes^{st}$ which fits in a commutative triangle
		$$\begin{tikzcd}
			\Fun(A\op , \spaces)\times \C\ar[r, "\boxtimes"] \ar[d, "{\infsusp_+\times\C}"] & \Fun(A\op , \C) \\
			\Fun(A\op , \spectra)\times \C.\ar[ur, dotted, "\boxtimes^{st}"'] &	\end{tikzcd}$$ 
		Moreover, $\boxtimes^{st}$ induces an equivalence
		$$\Fun(A\op ,\C)\simeq\Fun(A\op ,\spectra)\otimes\C.$$ \Cref{comptens} implies that one has identifications
		$$\Hom{}{F\boxtimes^{st} c}{G}\simeq\Hom{}{F}{\sHom{\C}{c}{G(-)}}$$
		functorially on $F\in\Fun(A\op ,\spectra)$, $G\in\Fun(A\op ,\C)$ and $c\in\C$, where $\sHom{\C}{c}{-}$ denotes the canonical enrichment of $\C$ in $\textbf{U}$-small spectra, and we have a commutative triangle
		$$
		\begin{tikzcd}
			\Fun(A\op , \spectra)\times\C\ar[r, "\boxtimes^{st}"] \ar[d, "{\text{id}\times\pfp{y}}"] & \Fun(A\op , \C) \\ \Fun(A\op , \spectra)\times\Fun_!(\spectra, \C)\ar[ur, "\o"']
		\end{tikzcd}
		$$
		
	\end{remark}
	\begin{proposition}\label{cptlygenstdual}
		Let $\C=\text{Ind}(\C^{\omega})$ be a compactly generated stable $\infty$-category. Then $\C$ is a strongly dualizable object in $\Cocont^{st}_{\infty}$, with dual $\text{Ind}((\C^{\omega})\op)$.
	\end{proposition}
	\begin{proof}
		Since $\C$ is stable and compactly generated, it follows that there exists a small stable $\infty$-category $A$ with finite colimits such that $\C\simeq\Fun_{ex}(A\op ,\spectra)$. Thus, since for any $\D$ stable and cocomplete we have $\Fun_!(\C,\D)\simeq\Fun_{ex}(A,\D)$, to prove the proposition we have to show that the canonical functor
		$$\Fun_{ex}(A\op ,\spectra)\otimes\D\rightarrow \Fun_{ex}(A\op ,\D)$$
		is an equivalence.
		We first prove that the inclusion $i:\Fun_{ex}(A\op ,\D)\hookrightarrow\Fun(A\op ,\D)$ admits a left adjoint $L$.
		\par 
		For any $a\in A$, denote by $y^{st}(a)$  the spectrally enriched representable functor associated to $a$, obtained as usual through the equivalence
		\begin{equation}\label{stableyoneda}
			\Fun_{ex}(A\op , \spectra)\simeq\Fun_{lex}(A\op , \spaces).
		\end{equation} 
		We define $L$ as the unique (up to a contractible space of choices) cocontinuous functor extending
		$$
		\begin{tikzcd}[row sep = tiny]
			A\times\D \ar[r] & \Fun_{ex}(A\op , \D) \\
			(a, x) \ar[r, mapsto] & y^{st}(a)\boxtimes^{st} x.
		\end{tikzcd}
		$$	
		Indeed, $y^{st}(a)\boxtimes^{st} x$ is exact as it can be modelled by the composition of two finite colimit preserving functors. When $\D = \spectra$, by (\ref{stableyoneda}) and the Yoneda lemma, one sees that $L$ is left adjoint to $i$. 
		\par 
		Let $\D$ be any stable cocomplete $\infty$-category. To see that $L$ is the desired left adjoint, we observe that for any $F\in\Fun_{ex}(A\op , \D)$, $x\in\D$, we have functorial identifications
		\begin{align*}
			\Hom{}{y^{st}(a)\boxtimes^{st} x}{F}&\simeq\Hom{}{y^{st}(a)}{\sHom{\D}{x}{F(-)}} \\
			&\simeq\Hom{}{y(a)}{\Hom{\D}{x}{F(-)}} \\
			&\simeq\Hom{}{a\boxtimes x}{F}
		\end{align*}
		where the hom-space on the right-hand side is taken on the $\infty$-category of presheaves of $\textbf{U}$-small spectra on $A$, and the second equivalence follows from the fact that $F$, and hence $\sHom{\D}{x}{F(-)}$, is exact.
		\par 
		Now notice that $i : \Fun_{ex}(A\op , \spectra)\hookrightarrow \Fun(A\op , \spectra)$ preserves colimits, and so, by tensoring with $\D$, one obtains an adjunction between cocontinuous functors
		$$\begin{tikzcd}
			\Fun(A\op , \D)\ar[r,bend left,"L\otimes\D",""{name=A, below}] & \Fun_{ex}(A\op , \spectra)\otimes\D\ar[l,bend left,"i\otimes\D",""{name=B,above}] \ar[from=A, to=B, symbol=\dashv]
		\end{tikzcd}$$ 
		where $i\otimes\D$ is fully faithful. Since $\Fun_{ex}(A\op , \spectra)\otimes\D$ and $\Fun_{ex}(A\op , \D)$ can be respectively identified with the essential images of $(Li)\otimes\D$ and $Li$, to conclude the proof it suffices to show that the two functors are naturally equivalent, but this is true because they coincide on objects of the type $a\boxtimes x$.
	\end{proof}
	Recall that an $\infty$-category $\C$ is called \textit{\textbf{V}-presentable} (for short, when there is no possibility of confusion we will only write presentable) if there exists a \textbf{V}-small $\infty$-category $A$ such that $\C$ is a left Bousfield localization of $\Fun(A\op , \spaces)$ by a \textbf{V}-small set of morphisms in $\Fun(A\op , \spaces)$. If we furthermore assume that the localization functor $\Fun(A\op , \spaces)\rightarrow\C$ is left exact, we will say that $\C$ is an \textit{$\infty$-topos}. It follows easily by this definition that any presentable $\infty$-category is complete and cocomplete. Presentable categories are equivalently defined as follows. Recall that, for $\C$ any $\infty$-category and $S$ a class of morphisms in $\C$, we define an object $X\in\C$ to be \textit{$S$-local} if, for every morphim $f:A\rightarrow B$ in $S$, the induced morphism
	$$\Hom{\C}{B}{X}\rightarrow\Hom{\C}{A}{X}$$
	is invertible. Then we say that an $\infty$-category $\C$ is \textit{\textbf{V}-presentable} is there exists a \textbf{V}-small class $S$ of morphisms in $\Fun(A\op , \spaces)$ such that $\C$ is equivalent to the full subcategory of $\Fun(A\op , \spaces)$ spanned by $S$-local objects. We refer the reader to \cite[Sections 5.5.1, 5.5.4]{lurie2009higher} for proofs of the equivalence of the above definitions of presentable $\infty$-categories.
	\par 
	We denote by $\text{Pr}_{L}$ the full subcategory of $\Cocont_{\infty}$ spanned by presentable $\infty$-categories and $\text{Pr}_{R} = \text{Pr}_{L}\op $. Notice that, by the adjoint functor theorem (see for example \cite[Proposition 7.11.8]{cisinski2019higher}), the morphisms in $\text{Pr}_{L}$ are left adjoints and, consequently, morphisms in $\text{Pr}_{R}$ are right adjoints. We also denote by $\Top$ the non full subcategory of $\text{Pr}_{R}$ whose objects are $\infty$-topoi and morphisms are functors which admit a left exact left adjoint (such functors are called \textit{geometric morphisms}).
	\begin{proposition}\label{tenspres}
		Let $\C$ and $\D$ be two presentable $\infty$-categories. Then $\C\otimes\D$ is presentable and there a canonical equivalence $\C\otimes\D\simeq\text{RFun}(\C\op , \D)$. In particular, $\text{Pr}_L$ inherits a symmetric monoidal structure.
	\end{proposition}
	\begin{proof}
		Let $A$ and $B$ be two small $\infty$-categories, $S$ and $S'$ two small sets of morphisms of $\Fun(A\op ,\spaces)$ and $\Fun(B\op , \spaces)$ respectively such that $\C$ and $\D$ are equivalent the full subcategories of $S$ and $S'$-local objects. By \Cref{preshdual} we have
		\begin{align*}
			\Fun(A\op ,\spaces)\otimes\Fun(B\op ,\spaces)&\simeq\Fun(A\op ,\Fun(B\op ,\spaces)) \\
			&\simeq \Fun((A\times B)\op ,\spaces).
		\end{align*} It then follows from the proof of \cite[Proposition 4.8.1.15]{lurie2017higher} that $\C\otimes\D$ can be identified with the full subcategory of $\Fun((A\times B)\op ,\spaces)$ spanned by $S\otimes S'$-local objects, where $S\otimes S'$ is the image of $S\times S'$ under the canonical functor 
		$$\Fun(A\op ,\spaces)\times\Fun(B\op ,\spaces)\rightarrow\Fun((A\times B)\op ,\spaces).$$ The proof of the last assertion follows from \cite[Lemma 4.8.1.16]{lurie2017higher} and \cite[Proposition 4.8.1.17]{lurie2017higher}.
	\end{proof}	
	\begin{remark}\label{presptdst}
		Let $\C$ be a presentable $\infty$-category. One can deduce easily from \Cref{tenspres} identifications $\C\otimes\spaces_{\ast}\simeq\C_{\ast}$ and $\C\otimes\spectra\simeq\text{Sp}(\C)$, where $\C_{\ast}$ denotes the $\infty$-category of \textit{pointed objects} of $\C$, and $\text{Sp}(\C)$ denotes the $\infty$-category of \textit{spectrum objects} of $\C$, i.e. the limit of the tower
		$$\begin{tikzcd}
			\dots\arrow[r, "\Omega"] & \C_{\ast}\arrow[r, "\Omega"] & \C_{\ast}
		\end{tikzcd}$$
		where $\Omega$ is the usual loop functor.
		Both these constructions come with canonical functors $\C_{\ast}\rightarrow\C$ and $\infloop:\text{Sp}(\C)\rightarrow\C$, and since $\C$ is presentable one can show that these admit left adjoints $(-)_+ : \C\rightarrow\C_{\ast}$ and $\infsusp_{+} : \C\rightarrow\text{Sp}(\C)$. By construction, we have a factorization
		$$
		\begin{tikzcd}
			\C\arrow[r, "(-)_+"] & \C_{\ast}\arrow[r, "\infsusp"] & \text{Sp}(\C).
		\end{tikzcd}
		$$
		In particular we see that, if $\C$ is presentable and pointed (stable), by tensoring $(-)_+ : \spaces\rightarrow\spaces_{\ast}$ ($\infsusp_{+} : \spaces\rightarrow\spectra$) with $\C$ we obtain an equivalence $\C\otimes\spaces_{\ast}\simeq\C$ ($\C\otimes\spectra\simeq\text{Sp}(\C)$). Thus, if $\C$, $\D$ and $\E$ are presentable $\infty$-categories where $\D$ is pointed and $\E$ is stable, we get functorial identifications
		$$\Fun_!(\C\otimes\spaces_{\ast}, \D)\simeq\Fun_!(\C,\Fun_!(\spaces_{\ast}, \D))\simeq\Fun_!(\C, \D)$$
		$$\Fun_!(\C\otimes\spectra, \E)\simeq\Fun_!(\C,\Fun_!(\spectra, \E))\simeq\Fun_!(\C, \E)$$
		induced respectively by precomposing with $(-)_+$ and $\infsusp_{+}$. Furthermore, we see that $(-)_+ :\spaces\rightarrow\spaces_{\ast}$ and $\infsusp_{+}: \spaces\rightarrow\spectra$ make $\spaces_{\ast}$ and $\spectra$ into idempotent algebras with respect to Lurie's tensor product: by \cite[Proposition 4.8.2.9]{lurie2017higher} this implies that there are canonical variablewise cocontinuous symmetric monoidal structures on $\spaces_{\ast}$ and $\spectra$ with unit objects given by $S^0\coloneqq(\ast)_+$ and $\s{}\coloneqq  \infsusp_{+}(\ast)$, and one can show that these coincide with the usual smash products of pointed spaces and spectra. In particular, we see that the functors 
		$$
		\begin{tikzcd}
			\spaces\arrow[r, "(-)_+"] & \spaces_{\ast}\arrow[r, "\infsusp"] & \spectra
		\end{tikzcd}
		$$
		are all symmetric monoidal, where $\spaces$ is equipped with the cartesian monoidal structure.
	\end{remark}
	\subsection{Sheaves and cosheaves}
	We now pass to recalling the definition of sheaves with values in an $\infty$-category. Let $X$ be a small $\infty$-category equipped with a Grothendieck topology. Recall that there is a small $\infty$-category $\text{Cov}(X)$, as defined in \cite[Notation 6.2.2.8]{lurie2009higher}, which can be described informally as having for objects pairs $(x, R)$, where $x\in X$ and $R\hookrightarrow y(x)$ is a sieve covering $x$, and morphisms between $(x, R)$ and $(y, R')$ are just maps $f:x\rightarrow y$ in $X$ such that the restriction of $y(f)$ to $R$ factors through $R'$. There is an obvious projection $\rho: \text{Cov}(X)\rightarrow X$ which has a section $s: X\rightarrow\text{Cov}(X)$ defined on objects by sending $x$ to $(x, y(x))$.
	\begin{definition}
		Let $\C$ be a complete $\infty$-category. With the same notations as above, we say that a functor $F\in\Fun(X\op , \C)$ is a \textit{sheaf} if the unit morphism
		$$\pb{\rho}F\rightarrow\pf{s}\pb{s}\pb{\rho}F\simeq\pf{s}F$$
		is an equivalence. Dually, for a cocomplete $\infty$-category $\C$, we say that a functor $F\in\Fun(X, \C)$ is a \textit{cosheaf} if the counit morphism
		$$\pfp{s}F\simeq\pfp{s}\pb{s}\pb{\rho}F\rightarrow\pb{\rho}F$$
		is an equivalence. We denote by $\Sh{X}{\C}$ ($\coSh{X}{\C}$) the full subcategory of $\Fun(X\op , \C)$ ($\Fun(X, \C)$) spanned by (co)sheaves. When $\C$ is the $\infty$-category of spaces $\spaces$, we will simply write $\Shsp{X}$.
	\end{definition}
	\begin{remark}
		More concretely, one can describe a sheaf as a functor $F$ such that for any covering sieve $R\hookrightarrow y(x)$ the canonical morphism
		$$F(x)\rightarrow\varprojlim_{y(x')\rightarrow R}F(x')$$
		is an equivalence. Notice also that we clearly have an equivalence $\coSh{X}{\C}\simeq\Sh{X}{\C\op }\op $.
	\end{remark}
	\begin{remark}
		It is well known that, for any $\infty$-site $X$, the $\infty$-category $\Shsp{X}$ is an $\infty$-topos. Unlike the case of 1-topoi, it's still unclear whether any $\infty$-topos is equivalent to $\Shsp{X}$ for some $\infty$-site $X$ (see \cite{345680}).
	\end{remark}
	
	We now give another description of $\infty$-categories of sheaves and cosheaves.
	
	\begin{lemma}\label{coshdistr}
		Let $X$ be an $\infty$-site, $\C$ be any cocomplete $\infty$-category. Then the restriction along the functor
		$\begin{tikzcd}[column sep= small]
			X\ar[r, "y"] & \Fun(X\op , \spaces)\ar[r, "L"]& \Shsp{X}
		\end{tikzcd}$
		defines an equivalence
		$$\coSh{X}{\C}\simeq \Fun_!(\Shsp{X}, \C),$$
		where $y$ is the Yoneda embedding and $L$ is the sheafification functor. Equivalently, a functor $X\rightarrow\C$ is a cosheaf if and only if its extension by colimits $\Fun(X\op , \spaces)\rightarrow\C$ factors through $L$. Dually, for any complete $\infty$-category $\C$, we have an equivalence
		$$\Sh{X}{\C}\simeq \Fun_{\ast}(\Shsp{X}\op , \C),$$
		where $\Fun_{\ast}$ denotes the full subcategory spanned by limit preserving functors.
	\end{lemma} 
	\begin{proof}
		Since $L$ commutes with colimits, by the universal property of localizations composition with $L$ embeds $\Fun_!(\Shsp{X}, \C)$ in $\Fun_!(\Fun(X\op ,\spaces), \C)$ as the full subcategory of functors sending covering sieves $R\hookrightarrow y(x)$ to equivalences in $\C$. On the other hand, a functor $F: X\rightarrow\C$ is a cosheaf precisely if there is an equivalence $$\varinjlim_{y(x')\rightarrow R}F(x')\simeq F(x)$$ for any sieve $R$ on $x\in X$, thus precisely if its extension by colimits $\Fun(X\op ,\spaces)\rightarrow\C$ lies in $\Fun_!(\Shsp{X}, \C)$.
	\end{proof}
	
	We provide a couple of examples of cosheaves.
	
	\begin{example}\label{pbcosh}
		\begin{enumerate}[(i)]
			\item Let $f:X\rightarrow Y$ be a continuous map between topological spaces. Recall that this induces a geometric morphism $\pf{f}:\Shsp{X}\rightarrow\Shsp{Y}$, which amounts to an adjunction $\pb{f}\dashv\pf{f}$, where $\pf{f}:\Shsp{X}\rightarrow\Shsp{Y}$ is defined by $\Sec{U} {\pf{f}{F} } = \Sec{f^{-1}(U)}{F}$ for any $U\subseteq Y$. By \cref{coshdistr}, one may characterize $\pb{f}:\Shsp{Y}\rightarrow\Shsp{X}$ as the essentially unique $\Shsp{X}$-valued cosheaf on $Y$ with the property that $\pb{f}(y(U)) = y(f^{-1}(U))$. 
			\item Let $\text{Top}$ be the 1-category of topological spaces, and $\text{Kan}\hookrightarrow\text{s}\cate{S}\text{et}$ be the full subcategory of all simplicial sets consisting of Kan complexes. Recall that there is a functor $\text{Top}\rightarrow\text{Kan}$ defined by assigning to each topological space $X$ its \textit{singular complex}, i.e. the simplicial set defined by $n\mapsto \Hom{\text{Top}}{\Delta^n}{X}$, where $\Delta^n$ is the standard n-simplex. Recall also that, by \cite[Theorem 7.8.9]{cisinski2019higher}, there is a functor $\text{s}\cate{S}\text{et}\rightarrow\spaces$ which identifies $\spaces$ as a localization of $\text{s}\cate{S}\text{et}$ at the class of weak homotopy equivalences. We define $\text{Sing} : \text{Top}\rightarrow\spaces$ as the composition of the two functors defined above. It is proven in \cite[A.3]{lurie2017higher} that, for any topological space $X$, the restriction of $\text{Sing}$ to poset of open subsets of $X$ $\Op{X}$ is indeed a cosheaf: this may be regarded as a non-truncated version of the classical Seifert-Van Kampen theorem. Furthermore, one can also show that $\text{Sing}$ is a \textit{hypercomplete cosheaf} (see \cite[Lemma A.3.10]{lurie2017higher}): this means that, as cocontinuous functor $\Shsp{X}\rightarrow\spaces$, $\text{Sing}$ factors through the \textit{hypercompletion} of $\Shsp{X}$.
		\end{enumerate} 
		
	\end{example}
	
	\begin{corollary}
		Let $X$ be an $\infty$-site, $\C$ be any presentable $\infty$-category. Then the inclusion $\coSh{X}{\C}\hookrightarrow\Fun(X, \C)$ admits a right adjoint.
	\end{corollary}
	
	\begin{proof}
		By \cite[Proposition 5.5.3.8]{lurie2009higher} and the previous lemma, the $\infty$-category $\coSh{X}{\C}$ is presentable. Thus, since $\coSh{X}{\C}\hookrightarrow\Fun(X, \C)$ obviously preserves colimits, we may conclude by the adjoint functor theorem.
	\end{proof} 
	
	\begin{corollary}\label{tensshC}
		Let $X$ be an $\infty$-site, $\C$ be any presentable $\infty$-category. Then we have an equivalence $\Shsp{X}\otimes\C\simeq\Sh{X}{\C}$.
	\end{corollary}
	
	\begin{proof}
		It follows from the adjoint functor theorem that $\Fun_{\ast}(\Shsp{X}\op ,\C)\simeq\text{RFun}(\Shsp{X}\op ,\C)$. Thus, by the previous lemma and by \Cref{tenspres}, we get the conclusion.
	\end{proof}
	
	\begin{construction}\label{dfntens}
		Let $X$ and $Y$ be two topological spaces. The functor
		$$
		\begin{tikzcd}[row sep = tiny]
			\Op{X}\times\Op{Y}\ar[r] & \Shsp{X\times Y} \\
			(U,V)\arrow[r, mapsto] & y(U\times V)
		\end{tikzcd}
		$$
		extends by colimits to a functor 
		$$\Fun(\Op{X}\op , \spaces)\times\Fun(\Op{Y}\op , \spaces)\rightarrow\Shsp{X\times Y}.$$
		Since it clearly sends covering sieves to equivalences in both variables, we obtain a functor
		$$
		\begin{tikzcd}[row sep = tiny]
			\Shsp{X}\times\Shsp{Y}\ar[r] & \Shsp{X\times Y} \\
			(F,G)\arrow[r, mapsto] & F\boxtimes G.
		\end{tikzcd}
		$$
		More generally, by \Cref{tensshC}, tensoring with two presentable $\infty$-categories $\C$ and $\D$ gives
		\begin{equation}\label{unstablekunneth}
			\Sh{X}{\C}\times\Sh{Y}{\D}\rightarrow\Sh{X\times Y}{\C\otimes\D}
		\end{equation}
		for which the image of a pair $(F, G)$ in the domain will still be denoted as $F\boxtimes G$.
		Let $\Delta : X\rightarrow X\times X$ be the diagonal. By post composing with $\pb{\Delta}$ we get a functor denoted by
		\begin{equation}\label{tensshCD}
			\begin{tikzcd}[row sep = tiny]
				\Sh{X}{\C}\times\Sh{X}{\D}\ar[r] & \Sh{X}{\C\otimes\D} \\
				(F,G)\arrow[r, mapsto] &  F\otimes G\coloneqq\pb{\Delta}(F\boxtimes G).
			\end{tikzcd}
		\end{equation}
		Suppose now that $\C$ is equipped with a symmetric monoidal structure $\otimes_{\C}$ such that the functor $\otimes_{\C} : \C\times\C\rightarrow\C$ preserves colimits in each variable. This induces a functor
		\begin{equation}\label{monoidalstrcoconttens}
			\C\otimes\C\rightarrow\C.
		\end{equation}
		Combining \cite[Proposition 3.2.4.7]{lurie2017higher} and \cref{tensshC}, one sees that $\Sh{X}{\C}$ admits a symmetric monoidal structure whose tensor preserves colimits in both variables. For the sake of brevity, we which will still denote the tensor product in $\Sh{X}{\C}$ by $\otimes_{\C}$.

		More concretely, the product $\otimes_{\C}$ in $\Sh{X}{\C}$ is given by the composition of the cocontinuous functor $\Sh{X}{\C\otimes\C}\rightarrow\Sh{X}{\C}$ obtained by applying  $\Shsp{X}\otimes-$ to (\ref{monoidalstrcoconttens}) and (\ref{tensshCD}). Alternatively, one can check that the functor
		\begin{equation}\label{monshC}
			\otimes_{\C} : \Sh{X}{\C}\times \Sh{X}{\C}\rightarrow\Sh{X}{\C}
		\end{equation} 
		can be described as
		$$(F, G)\mapsto L^{\C}(U\mapsto\Sec{U}{F}\otimes_{\C}\Sec{U}{G})$$ where $L^{\C}$ is the sheafification for $\C$-valued presheaves. 
		In particular, we have that any $F\in\Sh{X}{\C}$ induces a colimit preserving functor 
		$$-\otimes_{\C} F : \Sh{X}{\C}\rightarrow\Sh{X}{\C}.$$
		Since $\Sh{X}{\C}$ is presentable, this has a right adjoint denoted by $$\sHom{X}{F}{-} : \Sh{X}{\C}\rightarrow\Sh{X}{\C}.$$ This functor supplies $\Sh{X}{\C}$ with a self-enrichment and for this reason will be called \textit{internal Hom sheaf} functor.
	\end{construction}
	
	\begin{remark}\label{tenstomon}
		Let $X$ and $Y$ be two topological spaces, $\C$, $\D$ and $\E$ be presentable $\infty$-categories, and let $\alpha : \Shsp{X}\rightarrow\Shsp{Y}$ and $\Phi : \C\otimes\D\rightarrow\E$ be cocontinuous functors. Notice that the functoriality in each variable of the tensor product of cocomplete $\infty$-categories gives a commutative diagram
		$$
		\begin{tikzcd}
			\Sh{X\times X}{\C\otimes\D}\ar[r, "{\pb{\Delta}}"] \ar[d, "{\pb{\Delta}}"] & \Sh{X}{\C\otimes\D} \ar[r, "{\alpha\otimes(\C\otimes\D)}"] & \Sh{Y}{\cate{\C\otimes\D}} \ar[d, "{\Shsp{Y}\otimes\Phi}"] \\
			\Sh{X}{\C\otimes\D} \ar[r, "{\Shsp{X}\otimes\Phi}"] & \Sh{X}{\E} \ar[r, "{\alpha\otimes\E}"] & \Sh{Y}{\E}.
		\end{tikzcd}
		$$
		In particular, the diagram above shows that whenever we prove a formula involving the functor (\ref{tensshCD}) and operations on sheaves coming from some continuous map, then we may deduce immediately a corresponding formula for the functor $(\ref{monshC})$.
	\end{remark}
	
	We may now formulate the following proposition, which could be interpreted as a sort of K\"unneth formula (we will actually see later in \cref{kunneth} how one can deduce the K\"unneth formula from this).
	
	\begin{proposition}\label{tensprodspaces}
		Let $X$ and $Y$ be topological spaces, $\C$ and $\D$ two presentable $\infty$-categories, and assume that one of the two is locally compact. Then the functor (\ref{unstablekunneth}) induces an equivalence
		$$\Sh{X}{\C}\otimes\Sh{Y}{\D}\simeq\Sh{X\times Y}{\C\otimes\D}.$$ Moreover, let $f:X\rightarrow X'$ and $g: Y\rightarrow Y'$ be two continuous maps and assume that at least one among $X'$ and $Y'$ is locally compact. Then we have an equivalence
		$$\pb{(f\times g)}_{\C\otimes\D}(F\boxtimes G)\simeq\pb{f}_{\C}F\boxtimes\pb{g}_{\D}G$$
		which is functorial on $F\in\Sh{X'}{\C}$ and $G\in\Sh{Y'}{\D}$.
	\end{proposition}
	
	\begin{proof}
		By \cite[Proposition 7.3.3.9]{lurie2009higher}, for any topological space $X$ and any $\infty$-topos $\Y$, $\Sh{X}{\Y}$ is a product of $\Shsp{X}$ and $\Y$ in the $\infty$-category $\Top$. Thus, by the previous corollary combined with \cite[Proposition 7.3.1.11]{lurie2009higher}, if $Y$ is a locally compact topological space, we have an equivalence $$\Shsp{X}\otimes\Shsp{Y}\simeq\Shsp{X\times Y}.$$ For the second part of the statement, we first observe that by \cref{tensshC} and \cref{tenspres} it suffices to prove the case when $\C= \D= \spaces$, which amounts to providing a commutative square 
		$$
		\begin{tikzcd}
			\Shsp{X'}\times\Shsp{Y'} \ar[r] \ar[d, "{\pb{f}\times\pb{g}}"] & \Shsp{X'\times Y'} \ar[d, "{\pb{(f\times g)}}"] \\
			\Shsp{X}\times\Shsp{Y} \ar[r] & \Shsp{X\times Y}.
		\end{tikzcd}
		$$
		Since both the top right and the down left composition commute with colimits in both variables, one then gets this by \cref{coshdistr}, \cref{pbcosh} (i) and by observing that $$(f\times g)^{-1}(U\times V) = f^{-1}(U)\times g^{-1}(V)$$ for any $U$ and $V$ open subsets of $X'$ and $Y'$ respectively.
	\end{proof}
	
	\begin{corollary}[Monoidality]\label{monpb}
		Let $f:X\rightarrow Y$ be a morphism of locally compact topological spaces, $\C$ and $\D$ two presentable $\infty$-categories. Let $F,G, H$ be sheaves on $Y$, and let $K$ be a sheaf on $X$. Then we have a canonical identification
		$$\pb{f}(F\otimes G)\simeq\pb{f}F\otimes\pb{f}G$$
		and in particular, when $\C = \D$ has a symmetric monoidal structure, by transposition
		$$\pf{f}\sHom{X}{\pb{f}H}{K}\simeq\sHom{Y}{H}{\pf{f}K}.$$
	\end{corollary}
	\begin{proof}
		The commutativity of the diagram
		$$
		\begin{tikzcd}[column sep = 60, row sep = 30]
			\Sh{Y}{\C}\times\Sh{Y}{\D}\arrow[r, "{(\pb{f},\pb{f})}"]\arrow[d, "\otimes"'] & \Sh{X}{\C}\times\Sh{X}{\D}\arrow[d, "\otimes"] \\
			\Sh{Y}{\C\otimes\D}\arrow[r, "\pb{f}"] & \Sh{X}{\C\otimes\D}
		\end{tikzcd}
		$$
		follows from the commutativity of
		$$
		\begin{tikzcd}[column sep = 50, row sep = 25]
			X\arrow[r, "f"]\arrow[d, "\Delta"'] & Y\arrow[d, "\Delta"] \\
			X\times X\arrow[r, "{(f, f)}"] & Y\times Y
		\end{tikzcd}
		$$
		that is trivially verified. The last part follows directly from \cref{tenstomon}.
	\end{proof}
	
	Consider now the $\infty$-category $\Sh{X}{\C}$, where $X$ is any $\infty$-site and $\C$ is complete and cocomplete. It is natural to ask oneselves whether at this level of generality one is still able to obtain a result like \Cref{tensshC}, at least when the inclusion $\Sh{X}{\C}\hookrightarrow\Fun(X\op , \C)$ admits a left adjoint. In the rest of the section we will briefly outline the reason why the answer to this question doesn't seem to be affirmative. We start with a general proposition concerning left Bousfield localizations and categories of local objects.
	\begin{proposition}\label{lfunloc}
		Let $\C$ be an $\infty$-category and $S$ a class of morphism in $\C$. Denote by $\C_S$ the full subcategory of $\C$ spanned by $S$-local objects, and assume that the inclusion $i:\C_S\hookrightarrow\C$ admits a left adjoint $L$. Thus, composition with $L$ gives a fully faithful functor
		\begin{equation}\label{locobj}
			\text{LFun}(\C_S, \D)\hookrightarrow\text{LFun}(\C, \D)
		\end{equation}
		whose essential image is given by left adjoints $\C\rightarrow\D$ sending all morphisms in $S$ to equivalences.
	\end{proposition}
	\begin{proof}
		Let $W$ be the class of morphisms in $\C$ which are sent by $L$ to equivalences and denote by $\cate{A}$ and $\cate{A}'$ the full subcategories of $\text{LFun}(\C, \D)$ spanned respectively by left adjoints $\C\rightarrow\D$ sending all morphisms in $W$ to equivalences and left adjoints $\C\rightarrow\D$ sending all morphisms in $S$ to equivalences. By \cite[ Proposition 7.1.18]{cisinski2019higher}, we already know that (\ref{locobj}) is fully faithful and that its essential image is given by $\cate{A}$. It follows immediately by the definition of a local object that $L$ sends all morphisms in $S$ to equivalences, thus we just need to show that $\cate{A}'$ is contained in $\cate{A}$. 
		\par 
		Consider a functor $F:\C\rightarrow\D$ in $\cate{A}'$ with right adjoint $G:\D\rightarrow\C$. By definition of $\cate{A}'$, we have that for every morphisms in $f\in S$ and every $d\in\D$
		$$
		\begin{tikzcd}
			\Hom{\C}{c}{G(d)}\arrow[r, "f"]\arrow[d,"\simeq"] & \Hom{\C}{c'}{G(d)} \arrow[d,"\simeq"] \\
			\Hom{\D}{F(c)}{d}\arrow[r, "F(f)"] & \Hom{\D}{F(c')}{d}.
		\end{tikzcd}
		$$
		Thus $G(d)$ is $S$-local, and hence there exists a functor $G' : \D\rightarrow\C_S$ suche that $G = iG'$. Let now $f$ be a morphism in $W$.  By definition of $W$ we have, functorially on $d\in\D$,
		$$
		\begin{tikzcd}
			\Hom{\D}{F(c)}{d}\arrow[r, "F(f)"]\arrow[d,"\simeq"] & \Hom{\D}{F(c')}{d} \arrow[d,"\simeq"] \\
			\Hom{\C}{c}{iG'(d)}\arrow[r, "f"] \arrow[d,"\simeq"] & \Hom{\C}{c'}{iG'(d)} \arrow[d,"\simeq"] \\
			\Hom{\C_S}{L(c)}{G'(d)}\arrow[r, "L(f)", "\simeq"'] & \Hom{\C_S}{L(c')}{G'(d)},
		\end{tikzcd}
		$$
		and hence $F(f)$ is invertible, and so we may conclude.
	\end{proof}
	We now claim that there exists a class of morphisms $S_X$ of $\Fun(X\op , \C)$ such that $\Sh{X}{\C}$ can be identified with the full subcategory of $S_X$-local objects of $\Fun(X\op , \C)$. We define $S_X$ as the class of morphisms
	$$S_X = \{R\boxtimes M\rightarrow y(x)\boxtimes M\mid R\hookrightarrow y(x) \,\text{is a sieve,}\,M\in\C\}.$$ For any sieve $R\hookrightarrow y(x)$, $M\in\C$ and $F\in\Fun(X\op , \C)$, since $R\simeq\varinjlim\limits_{y(x')\rightarrow R}y(x')$, we have a commutative diagram
	$$\begin{tikzcd}
		\Hom{\C}{-}{F(x)}\arrow[r] \arrow[d, "\simeq"] & \Hom{\C}{-}{\varprojlim\limits_{y(x')\rightarrow R}F(x')}\arrow[d, "\simeq"] \\
		\Hom{\Fun(X\op , \C)}{y(x)\boxtimes -}{F}\arrow[r, "i"] & \Hom{\Fun(X\op , \C)}{R\boxtimes -}{F}
	\end{tikzcd}
	$$
	where the upper horizontal arrow is induced by the canonical map $F(x)\rightarrow\varprojlim\limits_{y(x')\rightarrow R}F(x')$. Thus, we see that the upper horizontal arrow is invertible if and only if the lower horizontal one is, and so $F$ is a sheaf if and only if it is $S_X$-local. In paricular, by \cref{lfunloc}, whenever the inclusion $\Sh{X}{\C}\hookrightarrow\Fun(X\op , \C)$ admits a left adjoint, the $\infty$-category $\Sh{X}{\C}$ is characterized by the universal property 
	$$\text{LFun}(\Sh{X}{\C}, \D)\hookrightarrow\text{LFun}_{S_X}(\Fun(X\op , \C), \D)$$
	where the right-hand side denotes the $\infty$-category of left adjoint functors sending all morphisms in $S_X$ to equivalence. On the other hand, tensoring the usual sheafification $\Fun(X\op , \spaces)\rightarrow\Shsp{X}$ with $\C$ gives a colimit preserving functor $$L':\Fun(X\op , \C)\simeq\Fun(X\op , \spaces)\otimes\C\rightarrow\Shsp{X}\otimes\C.$$ Combining the universal property of the tensor product of cocomplete categories and \Cref{preshdual}, we see that, for any cocomplete $\infty$-category $\D$, precomposition with $L'$ may be factored as
	\begin{align*}
		\Fun_!(\Shsp{X}\otimes\C, \D)&\simeq \Fun_!(\C, \Fun_!(\Shsp{X}, \D)) \\
		&\hookrightarrow \Fun_!(\C, \Fun_!(\Fun(X\op ,\spaces),\D)) \\ 
		&\simeq \Fun_!(\Fun(X\op ,\C), \D)	
	\end{align*}
	and hence indentifies $\Fun_!(\Shsp{X}\otimes\C, \D)$ with the full subcategory of $\Fun_!(\Fun(X\op ,\C), \D)$ spanned by those functors sending maps in $S_X$ to equivalences. Hence we obtain a comparison functor $$\Shsp{X}\otimes\C\rightarrow\Sh{X}{\C}$$
	but unless $\C$ is presentable, there is no evident reason why one should expect this to be an equivalence.
	
	\begin{remark}\label{nosheafification}
		A close inspection of the proof of \cite[Proposition 6.2.2.7]{lurie2009higher} shows that the usual formula for sheafifcation provides the desired left adjoint whenever $\C$ is bicomplete and, for every $x\in X$ and every sieve $R\hookrightarrow y(x)$, the functor
		$$
		\begin{tikzcd}[row sep = tiny]
			\Fun(X\op , \C) \arrow[r] & \C \\
			F \arrow[r, mapsto] & (\pf{s}F)(x, R)\simeq\varprojlim\limits_{y(x')\rightarrow R}F(x')
		\end{tikzcd}
		$$
		is accessible: this will be true automatically for example when $\C$ is presentable, since any functor between presentable $\infty$-categories which is a right adjoint is automatically accessible. A similar observation in the case of sheaves with values in ordinary 1-categories can be found in \cite[17.4]{schapira2006categories}. However, if we drop the presentability assumption for $\C$, it is not clear a priori if the inclusion $\Sh{X}{\C}\hookrightarrow\Fun(X\op , \C)$ should admit a left adjoint. 
	\end{remark}
	
	\begin{remark}\label{natongen}
		Suppose that $\C$ is such that $\Sh{X}{\C}\hookrightarrow\Fun(X\op , \C)$ admits a left adjoint $L$. Hence, for any $x\in X$ and $M\in\C$, by applying $L$ to $x\boxtimes M$ gives an object denoted by $M_x$ with the property that, for any other sheaf $F$, we have a functorial identification
		$$\Hom{}{M_x}{F}\simeq\Hom{}{M}{F(x)}.$$
		It follows from \cite[Proposition 7.1.18]{cisinski2019higher} and \Cref{preshdual} that, for any cocomplete $\infty$-category $\D$, we have a fully faithful functor
		$$\Fun_!(\Sh{X}{\C}, \D)\hookrightarrow\Fun_!(\C, \Fun(A, \D))$$
		and thus any cocontinuous functor with domain $\Sh{X}{\C}$ is uniquely determined by its restriction along the functor 
		$$
		\begin{tikzcd}[row sep = tiny]
			X\times\C\ar[r] & \Sh{X}{\C} \\
			(x, M)\ar[r, maps to] & M_x.
		\end{tikzcd}
		$$
	\end{remark}
	
	\section{Shape theory and shape submersions}
	In this section we will deal with questions related to shape theory from the perspective of higher topos theory: we recommend \cite[Appendix A]{lurie2017higher} and \cite{hoyois2018higher} for some good introductory accounts to this subject. We will start by defining a version of shape which is relative to a geometric morphism, and give a detailed description of its functoriality as well as a proof of its homotopy invariance. After that we will define essential and locally contractible geometric mophisms: the first notion refers to morphisms $\pf{f}:\X\rightarrow\Y$ whose relative shape is constant (as a pro-object on $\Y$) locally on $\X$, while the second to essential geometric morphisms satisfying an additional push-pull formula. After that we will define shape submersions, i.e. continuous maps which are locally given by projections $X\times Y\rightarrow Y$, where $X$ is such that the unique geometric morphism $\Shsp{X}\rightarrow \spaces$ is essential. These are proven to satisfy a base change formula, which will imply that they induce locally contractible geometric morphisms.
	\subsection{Relative shape}
	For any $\infty$-category $\C$, denote by $\Pro(\C)$ the $\infty$-category of pro-objects in $\C$, i.e. the free completion of $\C$ under cofiltered limits. When $\C$ is accessible (e.g. presentable, but more generally see \cite[Definition 5.4.2.1]{lurie2009higher}) and admits finite limits, one shows (\cite[Proposition 3.1.6]{lurie2011dagxiii}) that $\Pro(\C)$ is in fact equivalent to the full subcategory of $\Fun(\C, \spaces)\op $ spanned by the accessible (see \cite[Definition 5.4.2.5]{lurie2009higher}) left exact functors. 
	\par
	Let $F: \C\rightarrow \D$ be an accessible functor between presentable $\infty$-categories, and let $G : \D\rightarrow\Fun(\C, \spaces)\op $ be the composition of the Yoneda embedding $\D\hookrightarrow\Fun(\D, \spaces)\op $ with $F^{\ast} : \Fun(\D, \spaces)\op \rightarrow\Fun(\C, \spaces)\op $. By the adjoint functor theorem, $G$ factors through $\C$ if and only if $F$ commutes with limits, and when this condition is verified $G$ is a left adjoint to $F$. If $F$ is only left exact, by the characterization stated above $G$ factors through $\Pro(\C)$: in this case, we say that $G$ is the \textit{left pro-adjoint} of $F$. Notice that, in this situation, $G$ is a genuine left adjoint of the functor $\Pro(F) : \Pro(\C)\rightarrow\Pro(\D)$.
	\par
	Specializing to the case of a geometric morphism between $\infty$-topoi $\pf{f}:\X\rightarrow\Y$, we see that the pullback $\pb{f}$ admits a pro-left adjoint, that we will denote by $\pfs{f}: \X\rightarrow\Pro(\Y)$. 
	More explicitly, for every object $U\in\Y$, $\pfs{f}(U)$ is the pro-object on $\X$ defined by the assignment
	$$V\mapsto\Hom{\X}{U}{\pb{f}(V)}.$$
	\begin{definition}
		Let $\X$ be an $\infty$-topos, $\pf{f}:\X\rightarrow\Y$ a geometric morphism. We define the \textit{shape of $\X$ relative to $\Y$} as
		$$\Pi_{\infty}^{\Y}(\X)\coloneqq\pfs{f}1_{\X},$$
		where $1_{\Y}$ is a terminal object of $\Y$. We will say that $f$ is of \textit{constant shape} if $\Pi_{\infty}^{\Y}(\X)$ belongs to $\Y$, i.e. it is constant as a pro-object. In the case where $f$ is the unique geometric morphism $\pf{a} : \X\rightarrow\spaces$, $\pfs{a}1_{\X}$ will be denoted just by $\Pi_{\infty}(\X)$ and will be called the \textit{shape} or \textit{fundamental pro-$\infty$-groupoid} of $\X$.
	\end{definition}
	\begin{remark}
		Notice that, as a left exact functor $\X\rightarrow\spaces$, $\pfs{f}1_{\X}$ can be identified with the functor $\pf{a}\pf{f}\pb{f}$, where $a:\X\rightarrow\spaces$ is the unique geometric morphism.
	\end{remark}
	\begin{proposition}\label{functrelsh}
		There exists a functor 
		$$\Pi_{\infty}^{\Y}: \Top_{/\Y}\rightarrow\Pro(\Y)$$
		whose values on objects coincides with the shape relative to $\Y$ and whose values on morphisms $\begin{tikzcd}[column sep = tiny, row sep = tiny]
			\X \arrow[rd, "f"'] \arrow[rr, "g"] &          & \cate{X^{\prime}} \arrow[ld, "f^{\prime}"] \\
			& \Y &                                           
		\end{tikzcd}$ is given by the transformation
		$$\pf{a}\pf{f'}\pb{f'}\rightarrow\pf{a}\pf{f'}\pf{g}\pb{g}\pb{f'}\simeq\pf{a}\pf{f}\pb{f}$$
		induced by the unit of the adjunction $\pb{g}\dashv\pf{g}$. 
	\end{proposition}
	\begin{proof}
		Since we clearly have a functor
		$$\Funlex(\Y, \Y)\rightarrow\Funlex(\Y, \spaces)$$
		given by post composition with $\pf{a}$, and since right adjoints between presentable categories are automatically accessible, it suffices to prove that there is a functor
		$$(\Top_{/\Y})\op \rightarrow\Funlex(\Y, \Y)$$
		that assigns $\pf{f}\pb{f}$ to any $\pf{f}: \X\rightarrow\Y$, which at the level of morphisms $\begin{tikzcd}[column sep = tiny, row sep = tiny]
			\X \arrow[rd, "f"'] \arrow[rr, "g"] &          & \X' \arrow[ld, "f'"] \\
			& \Y &                                           
		\end{tikzcd}$ is given by the transformation
		$$\pf{f'}\pb{f'}\rightarrow\pf{f'}\pf{g}\pb{g}\pb{f'}\simeq\pf{f}\pb{f}$$
		induced by the unit of the adjunction $\pb{g}\dashv\pf{g}$. We will proceed through some reduction steps.
		\par
		First of all, the Yoneda embedding induces a fully faithful functor
		$$\Fun(\Y, \Y)\hookrightarrow\Fun(\Y, \Fun(\Y\op , \spaces))\simeq\Fun(\Y\op \times\Y, \spaces)$$
		and thus it suffices to construct a functor 
		$$(\Top_{/\Y})\op \rightarrow\Fun(\Y\op \times\Y, \spaces)$$
		whose image lies in $\Funlex(\Y, \Y)$. By \cite[Remark 6.1.5]{cisinski2019higher}, standard computations with adjunctions of $1$-categories show that we also have, functorially on $F, G\in\X$, a commutative square
		$$\begin{tikzcd}
			\Hom{\Y'}{\pb{f'}F}{\pb{f'}G} \arrow[r, "\simeq"] \arrow[d, "\pb{g}"'] & \Hom{\X}{F}{\pf{f'}\pb{f'}G} \arrow[d, "\text{unit}"] \\
			\Hom{\Y}{\pb{g}\pb{f'}F}{\pb{g}\pb{f'}G} \arrow[r, "\simeq"]         & \Hom{\X}{F}{\pf{f'}\pf{g}\pb{g}\pb{f'}G},             
		\end{tikzcd}$$
		so it suffices to show that the transformation on the left hand side can be enanched to a functor $(\Top_{/\Y})\op \rightarrow\Fun(\Y\times\Y\op , \spaces)$. Recall also that we have a forgetful functor
		$$\begin{tikzcd}[row sep = small]
			(\Top_{/\Y})\op \simeq(\Top\op )_{\Y/}\arrow[r] & {\infcat}_{\Y/}                 \\
			(\pf{f}:\X\rightarrow\Y) \arrow[r, maps to]                    & (\pb{f}:\Y\rightarrow\X).
		\end{tikzcd}$$ 
		Hence, we will construct a functor $${\infcat}_{\Y/}\rightarrow\Fun(\Y\op \times\Y, \spaces).$$
		\par 
		Recall that, for any $\infty$-category $\C$, we have a left fibration $\begin{tikzcd}
			\text{S}(\C) \arrow[r, "{(s, t)}"] & \C\op \times\C,
		\end{tikzcd}$ called the \textit{twisted diagonal}, classifying the hom-bifunctor $\C\op \times\C\rightarrow\spaces$ (as it is defined in \cite[5.6.1]{cisinski2019higher}). For any functor $f:\Y\rightarrow\X$, consider the left fibration defined by the pullback
		
		\begin{equation}\label{pbtwdiag}
			\begin{tikzcd}[column sep = huge]
				T(f) \arrow[d] \arrow[r] \arrow[dr, phantom, "\usebox\pullback" , very near start, color=black]                        & \text{S}(\X) \arrow[d, "{(s, t)}"] \\
				\Y\op \times\Y \arrow[r, "f\op \times f"] & \X\op \times\X.            
			\end{tikzcd}
		\end{equation}
		
		This classifies the functor $$\Hom{\Y}{f(-)}{f(-)}: \Y\op \times\Y\rightarrow\spaces.$$
		The functoriality on $\X$ of the twisted diagonal allows one to construct a functor 
		$${\infcat}_{\Y/}\rightarrow\Fun(\Delta^1,\infcat)\rightarrow\Fun(\Lambda^2_1, \infcat)$$
		sending an arrow $f:\Y\to\X$ to the cospan in the pullback square (\ref{pbtwdiag}). The functoriality of taking pullbacks of cospans thus implies the existence of a functor 
		$${\infcat}_{\Y/}\rightarrow{\infcat}_{/\Y\op\times \Y}$$
		sending $f$ to the left vertical arrow in the square (\ref{pbtwdiag}). Moreover, since the twisted diagonal is a left fibration, and left fibrations are stable under basechange, we see that the above defined functor factors through $\text{LFib}(\Y\op \times\Y)\subset{\infcat}_{/\Y\op\times \Y}$. Therefore, using the straightening-unstraightening equivalence (see for example \cite[Theorem 4.4.14]{cisinski2019higher}), we obtain the sought functor 
		$$T: {\infcat}_{\Y/}\rightarrow\text{LFib}(\Y\op \times\Y)\simeq\Fun(\Y\op\times\Y, \spaces).\qedhere$$
		%as illustrated by the diagram corresponding to a morphism $\begin{tikzcd}[row sep = tiny, column sep= tiny]
			%	& \X \arrow[ld, "f"'] \arrow[rd, "f'"] &           \\
			%	\Y \arrow[rr, "g"'] &                                            & \Y'
			%\end{tikzcd}$
			%$$\begin{tikzcd}[column sep = huge, row sep = normal]
				%	T(f) \arrow[r] \arrow[d, dotted] \arrow[dd, bend right=49] & \text{S}(\Y) \arrow[rd, "\text{S}(g)"] \arrow[dd, "{\,\,}" description, bend right=49] \arrow[d, dotted] &                                           \\
				%	T(f') \arrow[dr, phantom, "\usebox\pullback" , very near start, color=black] \arrow[r] \arrow[d]                                  & T(g) \arrow[dr, phantom, "\usebox\pullback" , very near start, color=black] \arrow[d] \arrow[d, phantom, "{(s, t)}" scale = 0.7, bend right=100] \arrow[r]                                                                                         & \text{S}(\Y') \arrow[d, "{(s, t)}"] \\
				%	\X\op \times\X \arrow[r, "f\op \times f"]    & \Y\op \times\Y \arrow[r, "g\op \times g"]                                                          & (\Y')\op \times\Y'.            
				%\end{tikzcd}$$
				%Moreover, \cite[Corollary 5.6.6]{cisinski2019higher} implies that $T$ sends any fully faithful functor to a fibrewise equivalence, and hence we can conclude.
			\end{proof}
			
			Recall that, for two maps $\begin{tikzcd}
				X \arrow[r, "f_0", shift left] \arrow[r, "f_1"', shift right] & X'
			\end{tikzcd}$ over a topological space $Y$, we say that \textit{$f_0$ is homotopic to $f_1$ over $Y$} if there exists a map $h:X\times I\rightarrow X'$ over $Y$ such that $f_t = hi_t$ $t=0,1$, where $i_t : X\hookrightarrow X\times I$ is the inclusion corresponding to $t\in I$.
			\begin{corollary}[Homotopy invariance]\label{hmtpyinv}
				The relative shape functor $\Pi_{\infty}^{\Shsp{Y}}$ sends homotopy equivalences over $Y$ to invertible morphisms in $\Pro(\Shsp{Y})$.
			\end{corollary}
			\begin{proof}
				It suffices to prove the analogous statement for the functor $T$ defined at the end of the proof of \cref{functrelsh}. Let $p : X\times I\rightarrow X$ be the canonical projection. By \cite[Lemma A.2.9]{lurie2017higher}, we know that $\pb{p}$ is fully faithful, and hence $T(p)$ is invertible. Since $pi_0 = pi_1 = \text{id}_X$ and $T$ is functorial, we get an equivalence $T(i_0)\simeq T(i_1)$. Thus, since there exists a homotopy $h$ over $Y$ such that $f_t = hi_t$ $t=0,1$, the functoriality of $T$ gives the desired $T(f_0)\simeq T(f_1)$.
			\end{proof}
			\begin{remark}
				Recall that, for any $\infty$-topos $\Y$, there is a fully faithful functor 
				\begin{equation}\label{slicefctrtop}
					\begin{tikzcd}[row sep = tiny]
						\Y \arrow[r, hook] & \Top_{/\Y} \\
						y \arrow[r, maps to]     & \Y_{/y}.               
					\end{tikzcd}
				\end{equation} 
				Since $\Top_{/\Y}$ has small cofiltered limits, the latter can be extended to a functor $$\beta:\Pro(\Y)\rightarrow\Top_{/\Y}.$$ It is possible to construct the functor $$\Pi_{\infty}^{\Y}: \Top_{/\Y}\rightarrow\Pro(\Y)$$ directly by showing that there is an equivalence
				$$\Hom{\Pro(\Y)}{\Pi_{\infty}^{\Y}(\X)}{Z}\simeq\Hom{\Top_{/\Y}}{\X}{\beta(Z)}$$
				
				which is functorial on $Z\in\Pro(\Y)$, as it is done in \cite[Proposition E.2.2.1]{lurie2016spectral}. See also \cite[Section 4.1]{carchedi2018relative}. However, here we preferred to give an alternative and more direct construction of the functorial structure of the relative shape. 
			\end{remark}
			\begin{remark}\label{laxshape}
				As usual, since $\Top$ has pullbacks, the slice $\Top_{/\X}$ can be equipped with a contravariantly functorial structure
				$$\Top\op \rightarrow\infcat$$
				that can be described for any geometric morphism $\pf{g}:\X\rightarrow\Y$ by sending an object $(f:\Y'\rightarrow\Y)\in\Top_{/\Y}$ to the resulting arrow over $\X$ obtained by performing the pullback of $f$ along $g$. Thus, since by \cite[Remark 6.3.5.8]{lurie2009higher} we have for any $y\in\Y$ a canonical pullback square
				$$
				\begin{tikzcd}
					\X_{/\pb{f}y}\arrow[d]\arrow[r] & \Y_{/y}\arrow[d] \\
					\X\arrow[r, "f"] &\Y
				\end{tikzcd}
				$$
				in $\Top$, we see that the functor (\ref{slicefctrtop}) is actually natural in $\X$, where the left hand side is functorial by the usual forgetful $\Top\op \rightarrow\infcat$ sending a geometric morphism $f$ to $\pb{f}$.
				In particular we obtain that 
				$$\beta:\Pro(\X)\rightarrow\Top_{/\X}$$
				is actually natural in $\X$. Notice that, if one regards $\infcat$ as an $(\infty, 2)$-category, the universal property of $\Pro(\X)$ and the definition of the slice imply that $\beta$ can be seen as a natural transformation between 2-functors. Hence, by \cite[Theorem 5.3.5]{haugseng2023lax}, by adjunction we may regard the relative shape as a lax natural transformation, where the 2-cells involved may be described as follows: any geometric morphism induces an adjunction
				$$\begin{tikzcd}
					\Pro(\X)\ar[r,bend left,"\pfs{g}",""{name=A, below}] & \Pro(\Y)\ar[l,bend left,"\pb{g}",""{name=B,above}] \ar[from=A, to=B, symbol=\dashv]
				\end{tikzcd}$$ 
				and for any commutative square of topoi
				$$
				\begin{tikzcd}
					\X'\arrow[r, "g'"]\arrow[d, "f'"] & \Y'\arrow[d, "f"] \\
					\X\arrow[r, "g"] & \Y
				\end{tikzcd}
				$$
				applying $\pb{g'}$ to the unit of the adjunction $\pfs{f}\dashv\pb{f}$ induces a natural transformation 
				$$\pb{g'}\rightarrow\pb{g'}\pb{f}\pfs{f}\simeq\pb{f'}\pb{g}\pfs{f}$$
				and hence by transposition
				$$\pfs{f'}\pb{g'}\rightarrow\pb{g}\pfs{f}$$
				called the \textit{base change transformation}, which when evaluated at $1_{\Y'}$ gives the desired
				$$\pfs{f'}\pb{g'}1_{\Y'}\simeq\pfs{f'}1_{\X'}\rightarrow\pb{g}\pfs{f}1_{\Y'}.$$
				
			\end{remark}
			\subsection{Locally contractible geometric morphisms}
			We start by recalling the definition of a locally cartesian closed $\infty$-category. We refer to \cite{gepner2017univalence} for more background on locally cartesian closed $\infty$-categories.
			\begin{definition}
				An $\infty$-category $\C$ is \textit{cartesian closed} if it admits finite products and, for any object $c\in\C$, the functor $-\times c : \C\rightarrow\C$ admits a right adjoint. 
				\par 
				Let $\C$ be any $\infty$-category, and let $f: c\rightarrow d $ be any arrow in $\C$. We say that $\C$ admits \textit{dependent products indexed by $f$} if the functor $\pb{f}:\C_{/d}\rightarrow\C_{/c}$ given by pulling back along $f$ admits a right adjoint. The right adjoint to $\pb{f}$ is called \textit{dependent product indexed by $f$}, and is usually denoted by $\prod_f : \C_{/c}\rightarrow\C_{/d}$. We say that $\C$ is \textit{locally cartesian closed} if it has pullbacks and, for any object $c\in\C$, the slice $\C_{/c}$ is cartesian closed. Equivalently, $\C$ is locally cartesian closed if, for any arrow $f: c\rightarrow d $,  $\C$ admits dependent products indexed by $f$.
				
				A functor $F:\C\rightarrow\D$ between locally cartesian $\infty$-categories is \textit{locally cartesian closed} if $F$ commutes with pullbacks and dependent products, i.e. for any arrow $f: c\rightarrow d $ in $\C$, we have a commutative square
				$$
				\begin{tikzcd}
					\C_{/c} \arrow[d, "F"'] \arrow[r, "\prod_{f}"] & \C_{/d} \arrow[d, "F"] \\
					\D_{/Fc} \arrow[r, "\prod_{Ff}"']              & \D_{/Fd}.            
				\end{tikzcd}
				$$
				We denote by $\infcat^{lcc}$ the subcategory of $\infcat$ whose objects are locally cartesian closed $\infty$-categories and morphisms are locally cartesian closed functors.
			\end{definition}
			\begin{example}
				By universality of colimits and adjoint functor theorem, any $\infty$-topos is locally cartesian closed.
			\end{example}
			Let $\pf{f}:\X\rightarrow\Y$ be a geometric morphism between $\infty$-topoi. Let $F\rightarrow G$ be a morphism in $\Pro(\Y)$ and let $H\rightarrow\pb{f}G$ be a morphism in $\Pro(\X)$. Then we have a canonical commutative square 
			$$
			\begin{tikzcd}
				\pfs{f}(\pb{f}F\times_{\pb{f}G}H) \arrow[d] \arrow[r] & \pfs{f}H \arrow[d] \\
				\pfs{f}\pb{f}F \arrow[d] \arrow[r]                    & \pfs{f}\pb{f}G \arrow[d] \\
				F \arrow[r]                                           & G                       
			\end{tikzcd}
			$$
			which determines a unique morphism
			\begin{equation}\label{projectionmorphism}
				\pfs{f}(\pb{f}F\times_{\pb{f}G}H)\rightarrow F\times_G \pfs{f}H.
			\end{equation}
			
			We need the following preliminary lemma.
			
			\begin{lemma}\label{lemmaprodvsdepprod}
				Let $\X$ be an $\infty$-topos, $\pf{\pi}: \X\rightarrow\spaces$ the unique geometric morphism. Let $A$ be any $\infty$-groupoid, and denote by $\alpha:A\rightarrow\Delta^0$ the unique map. Then dependent products in $\X$ indexed by $\pb{\pi}(\alpha)$ agree with limits indexed by $A$. More precisely, we have a natural equivalence $$\prod_{\pb{\pi}\alpha}\simeq\varprojlim_{A}$$ of functors $\text{Fun}(A,\X)\simeq\X_{/\pb{\pi}A}\rightarrow \X$.
		\end{lemma}
		
		\begin{proof}
			First of all notice that, by \cite[Lemma 6.1.3.7]{lurie2009higher} we have a cocontinuous functor
			$$
			\begin{tikzcd}[row sep = tiny]
				\spaces \arrow[r]    & \infcat\op \\
				A \arrow[r, maps to] & \X_{/\pb{\pi}A}.
			\end{tikzcd}
			$$ 
			This is naturally equivalent to the functor
			$$
			\begin{tikzcd}[row sep = tiny]
				\spaces \arrow[r]    & \infcat\op     \\
				A \arrow[r, maps to] & {\Fun(A, \X)}
			\end{tikzcd}
			$$
			since have equivalences
			$$\Fun(\Delta^0, \X)\simeq\X\simeq\X_{/\pb{\pi}\Delta^0}.$$ In particular, if $\alpha: A\rightarrow\Delta^0$ is the unique map, we have a corresponding commutative square
			$$
			\begin{tikzcd}
				\X_{/\pb{\pi}\Delta^0} \arrow[d, "\simeq"] \arrow[r, "\pb{\pi}A\times -"] & \X_{/\pb{\pi}A} \arrow[d, "\simeq"] \\
				\Fun(\Delta^0, \X) \arrow[r, "const"] & \Fun(A, \X)
			\end{tikzcd}
			$$
			where the lower horizontal arrow assigns to an object $F$ of $\X$ the constant functor at $F$. Thus we obtain an identification of the respective right adjoints, i.e. a commutative square
			$$
			\begin{tikzcd}
				\X_{/\pb{\pi}A}\arrow[d, "\simeq"] \arrow[r, "{\prod_{\pb{\pi}\alpha}}"] & \X_{/\pb{\pi}\Delta^0} \arrow[d, "\simeq"] \\
				\Fun(A, \X) \arrow[r, "{\varprojlim}"] & \Fun(\Delta^0, \X)
			\end{tikzcd}
			$$
			which is what we wanted.
		\end{proof}
		
		\begin{proposition}\label{locconstsheqdef}
			Let $\pf{f}:\X\rightarrow\Y$ be a geometric morphism between $\infty$-topoi. Consider the conditions
			\begin{enumerate}[(i)]
				\item $\pb{f}$ admits a left adjoint and, for every $F\rightarrow G$ in $\Y$ and $H\rightarrow\pb{f}G$ in $\X$, the morphism (\ref{projectionmorphism}) is invertible;
				\item $\pb{f}$ is locally cartesian closed;
				\item $\pb{f}$ admits a left adjoint and, for every $F$ in $\Y$ and $H$ in $\X$, the morphism
				$$\pfs{f}(\pb{f}F\times H)\rightarrow F\times \pfs{f}H$$ is invertible.
			\end{enumerate}
			Then (i) and (ii) are equivalent. Moreover, if $\Y$ is $0$-localic, then these are also equivalent to (iii).
		\end{proposition}
		\begin{proof}
			We first show that (i) implies (ii). Since $\pb{f}$ commutes with finite limits, it suffices to show that it commutes with dependent products. But, for any $\alpha : F\rightarrow G$ in $\Y$, the square
			$$
			\begin{tikzcd}
				\Y_{/F} \arrow[d, "\pb{f}"'] \arrow[r, "\prod_{\alpha}"] & \Y_{/G} \arrow[d, "\pb{f}"] \\
				\X_{/\pb{f}F} \arrow[r, "\prod_{\pb{f}\alpha}"']              & \X_{/\pb{f}G}            
			\end{tikzcd}
			$$
			commutes if and only if the square given by the corresponding left adjoints commutes. This last assertion is equivalent to requiring the projection morphims to be invertible, and so we are done.
			\par 
			We now show that (ii) implies (i). By the same argument as above, it suffices to prove that $\pb{f}$ admits a left adjoint. Since $\pb{f}$ is cocontinuous and preserves finite limits, we are only left to prove that $\pb{f}$ commutes with infinite products. Using \cref{lemmaprodvsdepprod}, we may exhibite products in any $\infty$-topos as a special case of dependent products. Therefore the implication follows from assumption (ii).
			\par Assume now that $\Y$ is $0$-localic, and that $\pb{f}$ admits a left adjoint. Clearly (iii) is a special case of (i). Assume then that $f$ satisyes the hypothesis (iii). Let $\alpha : F\rightarrow G$ be a morphism in $\Y$ and let $H\rightarrow\pb{f}G$ be a morphism in $\X$. Since $\Y$ is generated by $(-1)$-truncated objects, we have that for any $H\in\X_{/\pb{f}G}$ there is an equivalence $$H\simeq\varinjlim\limits_{U\rightarrow G}(H\times_{\pb{f}G}\pb{f}U)$$, where $U$ runs through the $(-1)$-truncated objects over $G$. Therefore, since the projection morphism is a natural transformation between colimit preserving functors, we may assume that $H\rightarrow\pb{f}G$ factors as $H\rightarrow \pb{f}U\rightarrow\pb{f}G$ for some $U\in\Y$ $(-1)$-truncated. Using the pasting properties of pullbacks, we may also assume that $G = U$. Notice that, for any (-1)-truncated object $V$ in a topos $\cate{Z}$ and for any other two objects $A, B\in\cate{Z}_{/V}$, we have an identification $A\times_V B\simeq A\times B$. This follows because for any other object $C$ mapping both to $A$ and $B$, we have that $\Hom{\cate{Z}}{C}{V}$ is contractible, and thus $\Hom{\cate{Z}}{C}{A\times_V B} \simeq\Hom{\cate{Z}}{C}{A\times B}$. Thus, since both $U$ and $\pb{f}U$ are (-1)-truncated, we are only left to prove that 
			$$\pfs{f}(\pb{f}F\times H)\rightarrow F\times \pfs{f}H$$
			is invertible, which is true by assumption.
		\end{proof}
		We are now ready to introduce the notion of a locally contractible geometric morphism. This was previously considered also in \cite{aizenbud2021relative}.
		\begin{definition}\label{dfnloccontr}
			Let $\pf{f}: \X\rightarrow\Y$ be a geometric morphism of $\infty$-topoi. We say that $f$ is \textit{essential} if $\pfs{f}$ factors through $\Y$, or equivalently if $\pb{f}$ admits a left adjoint. Furthermore, we say that an essential geometric morphism is \textit{locally contractible} if it satisfies the equivalent conditions (i) and (ii) in \Cref{locconstsheqdef}. We say that a geometric morphism is of \textit{trivial shape} if $\pb{f}$ is fully faithful, or equivalently if the unit transformation  $\text{id}_{\Y}\rightarrow\pf{f}\pb{f}$ is an equivalence. When $f$ is the unique geometric morphism $\X\rightarrow\spaces$, we will say that $\X$ is locally contractible.
		\end{definition}
		
		\begin{remark}\label{opensmoothbasechange}
			For a continuous map $f:X\rightarrow Y$ of topological spaces inducing an essential geometric morphism, one may interpret the condition of being a locally contractible geometric morphism as the requirement of a base change formula for $\pfs{f}$ along open immersions. More precisely, let $U$ be an open subset of $Y$, and consider the pullback square
			$$
			\begin{tikzcd}
				f^{-1}(U)\ar[r, "f'"]\ar[d, hook, "j'"] & U\ar[d, hook, "j"] \\
				X\ar[r, "f"] & Y.
			\end{tikzcd}
			$$
			For any $F\in\Shsp{X}$, we have natural equivalences 
			$$
			\begin{tikzcd}
				\pfs{j}\pb{j}\pfs{f}F \arrow[r, "\simeq"] & \pfs{f}F\times y(U) \arrow[r, "\simeq"] & \pfs{f}(F\times y(f^{-1}(U))) \arrow[r, "\simeq"] & \pfs{f}\pfs{j'}\pb{(j')}F\arrow[r, "\simeq"] & \pfs{j}\pfs{f'}\pb{(j')}F
			\end{tikzcd}
			$$
			where the second morphism is (\ref{projectionmorphism}). Therefore, since $\pfs{j}$ is fully faithful, we obtain a natural equivalence $$\pb{j}\pfs{f}\simeq\pfs{f'}\pb{(j')},$$
			and, by transposition, an equivalence
			
			$$\pb{f}\pf{j}\simeq\pf{j'}\pb{(f')}.$$
		\end{remark}
		
		\begin{example}\label{exess}
			\begin{enumerate}[(i)]
				\item Recall that any object $U\in\X$ of an $\infty$-topos determines a geometric morphism $j: \X_{/U}\rightarrow\X$. By \cite[Proposition 6.3.5.1]{lurie2009higher} $j$ is locally contractible, and $\pfs{j} : \X_{/U}\rightarrow\X$ can be described as the usual forgetful functor.
				
				\item By \cite[Proposition A.1.9]{lurie2017higher}, an $\infty$-topos is locally contractible if and only if the unique geometric morphism $\X\rightarrow\spaces$ is essential, since in this case the projection morphism is automatically invertible.
				
				\item Let $\pf{f} : \X\rightarrow\Y$ be an essential geometric morphism. For any $\infty$-topos $\cate{Z}$, by \Cref{tensoringadj} we obtain a geometric morphism $f\otimes\cate{Z}:\X\otimes\cate{Z}\rightarrow\Y\otimes\cate{Z}$ by applying to $f$ the functor $-\otimes\cate{Z}$. Since both $\pb{f}$ and $\pfs{f}$ commute with colimits, the adjunction $\pfs{f}\dashv\pb{f}$ is preserved by $-\otimes\cate{Z}$, and so $f\otimes\text{id}_{\cate{Z}}$ is an essential geometric morphism.
			\end{enumerate}
			
		\end{example}
		\begin{remark}\label{pfess}
			Let $f:X\rightarrow Y$ be a continuous map. Notice that, since the functor
			$$f^{-1} : \Op{Y}\rightarrow\Op{X}$$ preserves open coverings, for any complete $\infty$-category $\C$ we still have a well defined pushforward $\pf{f}:\Sh{X}{\C}\rightarrow\Sh{Y}{\C}$ given as usual by $\Sec{U}{\pf{f}F} = \Sec{f^{-1}(U)}{F}$ for all $U\in\Op{Y}$. Although at this level of generality there is no reason to expect $\pf{f}$ to have a left adjoint, if $f$ induces an essential geometric morphism at the level of sheaves of spaces, then it actually does. Indeed, recall that there is an equivalence $$\Sh{X}{\C}\simeq\Fun_{\ast}(\Shsp{X}\op , \C).$$ Through this equivalence and \Cref{pbcosh} (i), we can identify $\pf{f}:\Fun_{\ast}(\Shsp{X}\op , \C)\rightarrow\Fun_{\ast}(\Shsp{Y}\op , \C)$ with precomposition with the opposite of the pullback $\pb{f} : \Shsp{Y}\rightarrow\Shsp{X}$. Thus, similarly to \Cref{tensoringadj}, by applying the 2-functor $\Fun_{\ast}((-)\op , \C)$ to the adjunction between cocontinuous functors $\pfs{f}\dashv\pb{f}$, we obtain the desired left adjoint.
		\end{remark}
		It is straightforward to check that the composition of two locally contractible geometric morphism is again locally contractible (see [AC21, Corollary 3.2.5]). We observe that the properties of being essential or locally contractible can be checked locally on the source.
		\begin{lemma}\label{loccontrisloc}
			Let $\pf{f} : \X\rightarrow\Y$ be a geometric morphism, and let $\cate{B}\subseteq\X$ which generates $\X$ under colimits. For any object $U\in\cate{B}$, consider the composite geometric morphism
			\begin{equation}\label{restrflocconst}
				\begin{tikzcd}
					\X_{/U}\ar[r] & \X\ar[r, "f"] & \Y. 
				\end{tikzcd}
			\end{equation} 
			We have the following
			\begin{enumerate}[(i)]
				\item $f$ is essential if and only if (\ref{restrflocconst}) is of constant shape for any $U\in\cate{B}$;
				\item $f$ is locally contractible if and only if (\ref{restrflocconst}) is locally contractible for any $U\in\cate{B}$.
			\end{enumerate}
		\end{lemma}
		\begin{proof}
			A proof can be found in \cite[Proposition 3.1.5]{aizenbud2021relative} and \cite[Proposition 3.2.6]{aizenbud2021relative}.
		\end{proof}
		\begin{remark}
			The content of part (i) in \Cref{loccontrisloc} suggests that a valid alternative way to call a geometric morphism whose pullback has a left adjoint could have been \textit{locally of constant shape}. This is actually the approach taken by Lurie in \cite[Appendix A]{lurie2017higher}; however, we have decided to stick with the more concise nomenclature which appears also in \cite{johnstone2002sketches} and \cite{aizenbud2021relative}.
		\end{remark}
		\begin{corollary}\label{loccontrhyp}
			Let $X$ be a locally contractible topological space, $a: X\rightarrow\ast$ the unique map, and assume that $\X = \Shsp{X}$ is hypercomplete. Then $\Shsp{X}$ is locally contractible and  $\pfs{a}$ is equivalent to the extension by colimits of the cosheaf $\text{Sing}$. Consequently, for sheaves of spectra, the functor $\pfs{a} : \Sh{X}{\spectra}\rightarrow\spectra$ obtained by applying $-\otimes\spectra$ is uniquely determined by the formula $\pfs{a}(\s{U}) = \infsusp_{+}U$ for any $U\in\Op{X}$.
		\end{corollary}
		\begin{proof}
			Let $B$ be the poset of all contractible open subsets of $X$. Notice that, even though $B$ in general is not a sieve, by the hypercompleteness assumption of $\Shsp{X}$ we get an equivalence
			$$\varinjlim_{U\in B_{/V}}y(U)\simeq y(V)$$
			for any $V\in\Op{X}$ that can be easily checked on stalks. In particular, we see that the full subcategory $\Shsp{B}\subseteq\Shsp{X}$ generates $\Shsp{X}$ under colimits. Thus by \Cref{loccontrisloc} it suffices to check that $\Shsp{X}_{/U}\simeq\Shsp{U}$ is of constant shape for any $U\in B$, but this is true by homotopy invariance of the shape. The last assertion follows immediately by noticing that through the equivalence $\spaces\otimes\spectra\simeq\spectra$, an object $A\otimes\s{}$ corresponds (functorially on $A$) to $\infsusp_{+}A$.
		\end{proof}
		\begin{remark}
			Beware that the converse of \Cref{loccontrhyp} is not true. For this reason, to avoid confusion, from now on we will say that a topological space $X$ is essential if $\Shsp{X}$ is (or equivalently, if $\Shsp{X}$ is locally contractible by \Cref{exess} part (ii)). 
		\end{remark}

		\subsection{Shape submersions}
		\begin{definition}\label{subm}
			A continuous map $f:X\rightarrow Y$ between topological spaces is a \textit{shape submersion} if for every point $x\in X$ there exist an open neighbourhood $U$ of $x$ and a space $X^{\prime}$ which is essential, such that $f(U)$ is open in $Y$, $U$ is homeomorphic to $f(U)\times X^{\prime}$ and the diagram
			$$\begin{tikzcd}
				f(U)\times X^{\prime} \arrow[r, "\cong"] \arrow[rd, "pr_1"'] & U \arrow[r, hook] \arrow[d] & X \arrow[d, "f"] \\
				& f(U) \arrow[r, hook]        & Y               
			\end{tikzcd}$$
			commutes.
		\end{definition}
		\begin{example}
			\begin{enumerate}[(i)]
				\item By \Cref{loccontrhyp}, if $X$ is locally contractible and hypercomplete, then $X\rightarrow \ast$ is a shape submersion.
				\item Any topological submersion of fiber dimension $n$ is a shape submersion.
			\end{enumerate}
		\end{example}
		\begin{remark}\label{submbas}
			\begin{enumerate}[(i)] 
				\item It follows easily from the definition that shape submersions are stable under pullbacks of topological spaces.
				\item If $f :X\rightarrow Y$ is a shape submersion, then the set of open subsets of $X$ of the type $X^{\prime}\times V$, where $V$ is open in $Y$ and $X^{\prime}$ is essential, forms a basis for the topology of $X$. Although this basis is not closed under finite intersections, the set of representable sheaves corresponding to open subsets homeomorphic to the product of an open in $Y$ and an essential space generates $\Shsp{X}$ under colimits. To see this, consider
				$$\cate{B} = \{U\in\Op{X}\mid U\cong W, \,\text{with}\, W\in\Op{V\times S}  \,\text{for some}\, V\in\Op{Y}, S\, \text{essential}\}.$$ The set $\cate{B}$ clearly forms a basis closed under finite intersections, and then we have $\Shsp{X}\simeq\Shsp{\cate{B}}$ by \cite[Appendix A]{aoki2023tensor}. Moreover, since open immersions induce essential geometric morphisms, and since $\Op{V}\times\Op{S}$ forms a basis of $V\times S$ which is closed under finite intersections, we get our claim.
				\par 
				In the particular case of a topological submersion of fiber dimension $n$, since $\mathbb{R}^n$ is hypercomplete, we have an equivalence $\Shsp{\mathbb{R}^n}\simeq\Shsp{\cate{W}}$ where $\cate{W}\subseteq\Op{\mathbb{R}^n}$ is the poset of open balls inside $\mathbb{R}^n$, and thus, since $\mathbb{R}^n$ is locally compact, we have $\Shsp{V\times\mathbb{R}^n}\simeq\Shsp{V}\otimes\Shsp{\cate{W}}$ by \cite[Proposition 7.3.1.11]{lurie2009higher}. In particular, the set of representable sheaves corresponding to open subsets homeomorphic to the product of an open in $Y$ and an open ball in $\mathbb{R}^n$ generates $\Shsp{X}$ under colimits.
			\end{enumerate}
		\end{remark}
		For technical reasons that will be justified in a moment, from now on whenever we have a shape submersion $f:X\rightarrow Y$ we will assume that either $Y$ is locally compact or all essential spaces appearing in the basis of $X$ are locally compact.
		\begin{lemma}
			Any shape submersion $f:X\rightarrow Y$ induces an essential geometric morphism. Thus, for any presentable $\infty$-category $\C$, we obtain an adjunction $\pfs{f}\dashv\pb{f}$ for $\C$-valued sheaves.
		\end{lemma}
		\begin{proof}
			Since $\pb{f}$ is a left adjoint, it is in particular accessible. Hence, by adjoint functor theorem, it suffices to prove that $\pb{f}$ commutes with limits. For any functor $I\rightarrow\Shsp{Y}$ we have a canonical map
			$$\pb{f}(\varprojlim_{i\in I}F_i)\rightarrow\varprojlim_{i\in I}\pb{f}F_i$$
			and it suffices to check that this is an equivalence after restricting to any open subset in the basis of $X$ associated to $f$. Hence, since the operation of restricting a sheaf to an open subset commutes with limits, we can assume that $f$ is a projection $f: X\times Y\rightarrow Y$ where either $X$ or $Y$ is locally compact and $X$ is essential. Thus, by point (iii) in \Cref{exess} and \Cref{tensprodspaces}, we get that $f$ is essential. The last assertion follows immediately by \Cref{tensshC}.
		\end{proof}
		\begin{lemma}[Smooth base change]\label{smoothbc}
			Let $\C$ be a presentable $\infty$-category. For every given pullback square
			$$\begin{tikzcd}
				X' \arrow[r, "f'"] \arrow[d, "g'"'] 
				\arrow[dr, phantom, "\usebox\pullback" , very near start, color=black]
				& X \arrow[d, "g"] \\
				Y' \arrow[r, "f"] & Y
			\end{tikzcd}$$
			of topological spaces where $g$ and $g'$ are shape submersions, there is a natural equivalence
			$$\pb{f}\pfs{g}\simeq\pfs{g'}\pb{f'}$$
			and, by transposition, also
			$$\pb{g}\pf{f}\simeq\pf{f'}\pb{g'}$$
			for $\C$-valued sheaves. 
		\end{lemma}
		\begin{proof}
			First of all, the base change transformation as defined in \Cref{laxshape} defnes a comparison natural transformation. Since all functors appearing are colimit preserving, it suffices to check that the morphism is an equivalence only on a family of objects generating $\Sh{X}{\C}$ by colimits. Hence, applying \Cref{submbas} (ii), we see that we can assume that the pullback square is of the type
			$$\begin{tikzcd}
				X\times Y' \arrow[r, "\text{id}_X\times f"] \arrow[d, "g'"'] 
				\arrow[dr, phantom, "\usebox\pullback" , very near start, color=black]
				& X\times Y \arrow[d, "g"] \\
				Y' \arrow[r, "f"] & Y
			\end{tikzcd}$$
			where $X$ is essential, $g$ and $g'$ are the canonical projections. By \Cref{exess} point (iii) and \Cref{tensprodspaces}, we have $\pb{(\text{id}_X\times f)}\simeq\pb{(\text{id}_X)}\otimes\pb{f}$, $\pfs{g}\simeq\pfs{a}\otimes\pfs{(\text{id}_Y)}$ and $\pfs{g'}\simeq\pfs{a}\otimes\pfs{(\text{id}_{Y'})}$, where $a:X\rightarrow\ast$, and so, since $\otimes$ is a bifunctor, we have
			\begin{align*}
				\pb{f}\pfs{g}&\simeq\pb{f}(\pfs{a}\otimes\pfs{(\text{id}_{Y})})  \\ 
				&\simeq\pfs{a}\otimes\pb{f} \\
				&\simeq (\pfs{a}\otimes\pfs{(\text{id}_{Y'})})(\pb{\text{id}_X}\otimes\pb{f}) \\
				&\simeq \pfs{g'}\pb{(\text{id}_X\times f)}.\qedhere
			\end{align*}
		\end{proof}
		
		\begin{corollary}[Smooth projection formula]\label{smoothproj}
			Let $f:X\rightarrow Y$ be a shape submersion, $\C$ and $\D$ two presentable $\infty$-categories. Then for any $F\in\Sh{X}{\C}$ and $G, H\in\Sh{Y}{\D}$, we have a canonical equivalence
			$$\pfs{f}(F\otimes\pb{f}G)\simeq\pfs{f}F\otimes G.$$ Hence, by transposition, when $\C = \D$ has a symmetric monoidal structure, equivalences
			$$\pf{f}\sHom{X}{F}{\pb{f}G}\simeq\sHom{Y}{\pfs{f}F}{G}$$
			and
			$$\pb{f}\sHom{Y}{G}{H}\simeq\sHom{X}{\pb{f}G}{\pb{f}H}.$$ 
			In particular, any shape submersion induces a locally contractible geometric morphism.
		\end{corollary}
		\begin{proof}
			Let $\Gamma_f: X\hookrightarrow X\times Y$ be the graph of $f$. We have
			\begin{align*}
				\pfs{f}(F\otimes\pb{f}G)&\simeq\pfs{f}\pb{\Gamma_f}(F\boxtimes G) \\
				&\simeq\pb{\Delta}\pfs{(f\times \text{id}_Y)}(F\boxtimes G) \\
				&\simeq\pfs{f}F\otimes G
			\end{align*}
			where the second equivalence follows from applying \cref{smoothbc} to the pullback square
			$$
			\begin{tikzcd}
				X\arrow[d, "f"]\arrow[r, "\Gamma_f"] & X\times Y \arrow[d, "f\times\text{id}_Y"] \\
				Y\arrow[r, "\Delta"] & Y\times Y,
			\end{tikzcd}
			$$
			and the third follows from \Cref{exess} point (iii) and \Cref{tensprodspaces}.
			
			The last assertion follows from specializing (\ref{monshC}) to the case when $\C$ is $\spaces$ equipped with the cartesian monoidal structure and by the commutativity of the diagram
			$$
			\begin{tikzcd}
				F\times\pb{f}G\arrow[r]\arrow[d, "\simeq"] & \pb{f}(\pfs{f}F\times G)\arrow[d, "\simeq"] \\
				\pb{\Gamma_f}(F\boxtimes G) \arrow[r] & \pb{\Gamma_f}\pb{(f\times\text{id}_Y)}\pfs{(f\times\text{id}_Y)}(F\boxtimes G)
			\end{tikzcd}
			$$
			where the upper horizontal arrow is the one which transposes to the projection morphism and the lower horizontal one transposes to the smooth base change transformation.
		\end{proof}
		
		\section{Localization sequences}
		We now prove a version of the localization theorem for sheaves of spaces on a topological space $X$, and deduce from that a localization theorem for sheaves of spectra. The latter essentially states that, for any closed immersion $i:Z\rightarrow X$ with open complement $j: U\rightarrow X$, there is a cofiber sequence of functors 
		$$\pfs{j}\pb{j}\rightarrow\text{id}\rightarrow\pf{i}\pb{i}.$$
		%the inclusions 
		%$$\begin{tikzcd}
		%	\Shsp{Z} \arrow[r, "\pf{i}", hook] & \Shsp{X} & \Shsp{U} \arrow[l, "\pf{j}", hook]
		%\end{tikzcd}$$
		%form a \textit{recollement} (in the sense of \cite[Definition A.8.1]{lurie2017higher}). Achieving this goal in our context is slightly more complicated than in the case of \cite[Proposition 2.3.6]{kashiwara1990sheaves}, namely because we don't want to assume all our spaces to be hypercomplete. 
		We follow the strategy outlined in \cite{khan2019morel}. The main ingredient is to show that the pushforward $\pf{i} : \Shsp{Z}\rightarrow\Shsp{X}$ commutes with \textit{contractible} colimits, i.e. colimits indexed by contractible simplicial sets. From this we are able to reduce to checking the theorem in the case of representable sheaves, which is almost straightforward. 
		\par
	%	We start by reporting \cite[Definition 3.1.5]{khan2019morel} and \cite[Lemma 3.1.6]{khan2019morel}.
	%	\begin{definition}
	%		Let $X$ and $Y$ be essentially small $\infty$-sites, and assume that $Y$ admits an initial object $\emptyset_{Y}$. A functor
	%		$u : X \rightarrow Y$ is \textit{topologically quasi-cocontinuous} if for every covering sieve $R^{\prime}\hookrightarrow y(u(x))$ in $Y$, the sieve $R \hookrightarrow y(x)$, generated by	morphisms $x^{\prime}\rightarrow x$ such that either $u(x^{\prime})$ is initial or $y(u(x^{\prime}))\rightarrow y(u(x))$ factors through $R^{\prime}\hookrightarrow y(u(x))$, is a covering in $X$.
			
	%	\end{definition}
	%	\begin{lemma}
	%		With notation as in the previous definition, let $u : X \rightarrow Y$ be a topologically quasicocontinuous functor. Assume that the initial object $\emptyset_{Y}$ is strict in the sense that for any object
	%		$y\in Y$, any morphism $d\rightarrow \emptyset_{Y}$ is invertible. Assume also that, for any object $d\in Y$ the sieve $\emptyset_{\Fun(Y, \spaces)}\hookrightarrow y(d)$ is a covering in $Y$ if and only if $d$ is initial (where $\emptyset_{\Fun(Y, \spaces)}$ denotes the initial
	%		object of $\Fun(Y, \spaces)$). Then the functor $\Shsp{Y} \rightarrow \Shsp{X}$, given by the assignment $F \mapsto L_{X} (u^{\ast}(F))$, where $L_{X} : \Fun(X, \spaces)\rightarrow\Shsp{X}$ denotes the sheafification functor, commutes with contractible colimits.
	%	\end{lemma}
		\begin{lemma}
			Let $i: \cate{Z}\hookrightarrow \X$ be a closed immersion of $\infty$-topoi. Then $\pf{i}$ commutes with contractible colimits. In particular, if $i:Z\hookrightarrow X$ is a closed immersion of topological spaces, then $\pf{i}:\Shsp{Z}\rightarrow\Shsp{X}$ preserves contractible colimits.
		\end{lemma}
		\begin{proof}
			Let $U\in\X$ be the $(-1)$-truncated object whose complementary closed subtopos of $\X$ is given by $\cate{Z}$. Recall that, by \cite[Lemma 7.3.2.4]{lurie2009higher}, $\pf{i}$ identifies $\cate{Z}$ with the full subcategory of $\X$ spanned by objects $F$ with the property the map $p_2:F\times U\rightarrow U$ is an equivalence. Since $(-)\times U$ preserves colimits, and $\varinjlim_{i\in I}U\simeq U$ for any contractible $\infty$-category $I$, we therefore deduce that $\pf{i}$ preserves contractible colimits. The part of the statement referring to closed immersions of topological spaces follows from \cite[Corollary 7.3.2.10]{lurie2009higher}.
			%By the lemma above and by unraveling the definition of topologically quasi-cocontinuous functor, this amounts to check that, for any $V\in\Op{X}$ and any open covering $\{W_i\}_{i\in I}\subseteq\Op{Z}$ of $V\cap Z$, the family
			%$$T =\{U\subseteq V\mid U\cap Z\ = \emptyset\,\,\text{or}\,\, U\cap Z\subseteq W_i\,\text{for some}\, i\in I\}\subseteq\Op{X}$$
			%covers $V$. But this is clear, because $V\setminus Z\in T$ and any $W_i$ can be written as $W_i^{\prime}\cap Z$ for some $W_i^{\prime}\in\Op{V}$.
		\end{proof}
		\begin{corollary}\label{expbclosed}
			Let $\C$ be any pointed presentable $\infty$-category. Then the pushforward $\pf{i}^{\C} : \Sh{Z}{\C}\rightarrow\Sh{X}{\C}$ commutes with all colimits, and thus admits a right adjoint $\pbp{i}_{\C} : \Sh{X}{\C}\rightarrow\Sh{Z}{\C}$.
		\end{corollary}
		\begin{proof}
			By \cref{tensshC}, it suffices to prove the corollary for $\C = \spaces_{\ast}$. Note that it suffices to check that $\pf{i}$ preserves the initial object and commutes with contractible colimits: any $F: I\rightarrow\D$ from a simplicial set $I$ to an $\infty$-category $\D$ with an initial object $\emptyset_{\D}$ may be seen as $I\rightarrow\D_{\emptyset_{\D}/}$ and thus corresponds to a functor $\Delta^0\star I\rightarrow\D$ with the same colimit as $F$ but indexed by weakly contractible simplicial set. But $\pf{i}$ preserves the initial object because $\Sh{Z}{\spaces_{\ast}}$ is pointed, and thus we may conclude by the previous lemma.
		\end{proof}
		Let $i: Z\hookrightarrow X$ be a closed immersion with open complement $j: U\hookrightarrow X$. For any $F\in\Shsp{X}$, consider the functorial commutative square
		$$\begin{tikzcd}[column sep = large]
			\pfs{j}\pb{j}(F) \arrow[r] \arrow[d]   & F \arrow[d]     \\
			\pfs{j}\pb{j}\pf{i}\pb{i}(F) \arrow[r] & \pf{i}\pb{i}(F),
		\end{tikzcd}$$
		where all the morphisms are given by the obvious units and counits.
		Notice that for any $G\in\Shsp{Z}$ and $V\in\Op{U}$, we have
		$$\Sec{V}{\pf{i}G} = \Sec{U\cap Z}{G}\simeq\ast,$$
		and so we can identify $\pb{j}\pf{i}:\Shsp{Z}\rightarrow\Shsp{U}$ with a constant functor with value the terminal object $y(U)\in\Shsp{U}$. Hence the previous square may be written as
		\begin{equation}\label{localization}
			\begin{tikzcd}[column sep = large]
				\pfs{j}\pb{j}(F) \arrow[r] \arrow[d]   & F \arrow[d]     \\
				\pfs{j}(y(U)) \arrow[r] & \pf{i}\pb{i}(F).
			\end{tikzcd}
		\end{equation}
		\begin{theorem}\label{localizationtheorem}
			The canonical square (\ref{localization}) is a pushout.
		\end{theorem}
		\begin{proof}
			Since all functors appearing in (\ref{localization}) commute with contractible colimits and any sheaf on $X$ is canonically written as colimit indexed by the contractible category $\Op{X}_{/F} = \Shsp{X}_{/F}\times_{\Shsp{X}}\Op{X}$ (it has an initial object), it suffices to prove the theorem when $F = y(V)$ for some $V\in\Op{X}$, and hence we just need to show that $\pf{i}\pb{i}(y(V))\simeq y(U\cup V)$. For any $W\in\Op{X}$, we have
			\begin{align*}
				\Sec{W}{\pf{i}\pb{i}(y(V))} &\simeq \Sec{W}{\pf{i}(y(V\cap Z))} \\
				& = \Sec{W\cap Z}{y(V\cap Z)} \\
				& = \Hom{\Op{Z}}{W\cap Z}{V\cap Z} \\
				& = \Hom{\Op{X}}{W}{V\cup U} \\
				& = \Sec{W}{y(V\cup U)},
			\end{align*}
			where the second to last identification follows from the usual exponential adjunction in the boolean algebra of all subsets of $X$.	
		\end{proof}
		\begin{corollary}\label{locseq}
			Let $i : Z\hookrightarrow X$ be a closed immersion with open complement $j:U\hookrightarrow X$, and let $\pf{i}^{\C}$, $\pb{i}_{\C}$, $\pbp{i}_{\C}$, $\pfs{j}^{\C}$ and $\pb{j}_{\C}$ be the induced pushforward and pullback functors at the level of $\C$-valued sheaves, where $\C$ is any pointed presentable $\infty$-category. Then we get a canonical cofiber sequence
			\begin{equation}\label{localizationseq1}
				\pfs{j}^{\C}\pb{j}_{\C}F\rightarrow F\rightarrow\pf{i}^{\C}\pb{i}_{\C}F
			\end{equation}
			and dually a fiber sequence
			\begin{equation}\label{localizationseq2}
				\pf{i}^{\C}\pbp{i}_{\C}F\rightarrow F\rightarrow\pf{j}^{\C}\pb{j}_{\C}F.
			\end{equation}denote
		\end{corollary}
		\begin{proof}
			It suffices to treat only the case of the sequence (\ref{localizationseq1}), since (\ref{localizationseq2}) follows from the former by passing to right adjoints. Furthermore, we only need to prove the case of sheaves of pointed spaces, since all functors appearing are colimit preserving and we have a canonical equivalence $\Sh{X}{\C} \simeq\Sh{X}{\spaces_{\ast}}\otimes\C$, where $\C$ is any presentable pointed $\infty$-category and $X$ is any topological space. We define the canonical morphisms in ($\ref{localizationseq1}$) through counit and unit of the appropriate adjunctions, and we see that there is a unique null-homotopy of the composition of those two morphisms coming from the fact that $\pb{i}_{\spaces_{\ast}}\pfs{j}^{\spaces_{\ast}}$ is equivalent to a constant functor with value the zero object in $\Sh{Z}{\spaces_{\ast}}$. 
			\par
			Let $\alpha:\Sh{X}{\spaces_{\ast}}\rightarrow\Shsp{X}$ be the forgetful functor. A close inspection of the appropriate universal properties shows that, for any $F\in\Sh{X}{\spaces_{\ast}}$, there is a canonical pushout square
			$$\begin{tikzcd}
				\pfs{j}(y(U)) \arrow[d] \arrow[r]  & y(X) \arrow[d]                                           \\
				\pfs{j}\pb{j}\alpha(F) \arrow[r] & \alpha(\pfs{j}^{\spaces_{\ast}}\pb{j}_{\spaces_{\ast}}F),\arrow[ul, phantom, "\usebox\pushout", very near start]
			\end{tikzcd}    $$ 
			where the left vertical map is induced by the point of $F$. Thus, since $\alpha$ reflects pushouts and we have an equivalence $\alpha\pf{i}^{\spaces_{\ast}}\pb{i}_{\spaces_{\ast}}\simeq\pf{i}\pb{i}\alpha$, it suffices to prove that the canonical square
			$$\begin{tikzcd}
				\alpha(\pfs{j}^{\spaces_{\ast}}\pb{j}_{\spaces_{\ast}}F) \arrow[d] \arrow[r] & \alpha(F) \arrow[d]   \\
				y(X) \arrow[r]                                                               & \pf{i}\pb{i}\alpha(F).
			\end{tikzcd}$$
			induced by applying $\alpha$ to the sequence (\ref{localizationseq1}) is a pushout. For this purpose, consider the commutative diagram
			$$\begin{tikzcd}
				\pfs{j}(y(U)) \arrow[r] \arrow[d]             & y(X) \arrow[d]                                                               &                       \\
				\pfs{j}\pb{j}\alpha(F) \arrow[d] \arrow[r] & \alpha(\pfs{j}^{\spaces_{\ast}}\pb{j}_{\spaces_{\ast}}F) \arrow[d] \arrow[r] & \alpha(F) \arrow[d]   \\
				\pfs{j}\pb{j}(y(X)) \arrow[r]                     & y(X) \arrow[r]                                                               & \pf{i}\pb{i}\alpha(F).
			\end{tikzcd}$$
			The upper left square and the left vertical rectangle are both pushouts, and so also the lower left square is a pushout. But the lower horizontal rectangle is a pushout, and so we can conclude.
		\end{proof}
		
		\begin{remark}
			\cref{locseq} can alternatively be deduced from \cite[Remark A.8.5, Proposition A.8.15]{lurie2017higher}.
		\end{remark}
		
		\begin{corollary}\label{pfsclosed}
			Consider a pullback square
			$$\begin{tikzcd}
				Z' \arrow[r, "f'"] \arrow[d, hook,  "s'"'] 
				\arrow[dr, phantom, "\usebox\pullback" , very near start, color=black]
				& Z \arrow[d, hook,  "s"] \\
				X \arrow[r, "f"] & Y
			\end{tikzcd}$$ 
			where $f$ and $f'$ are shape submersions and $s$ (and consequently $s'$) is a closed immersion. Then, for sheaves with values in a pointed presentable $\infty$-category, we have a canonical equivalence
			$$\pf{s}\pfs{f'}\simeq\pfs{f}\pf{s'}$$
			or equivalently
			$$\pb{f'}\pbp{s}\simeq\pbp{s'}\pb{f}.$$
		\end{corollary}
		\begin{proof}
			Let $j$ and $j'$ be the complement open immersions associated respectively to $s$ and $s'$. By the localization sequences and by smooth base change, we have a commutative diagram where all the rows are cofiber sequences and all the vertical arrows are invertible
			$$
			\begin{tikzcd}
				\pfs{j}\pb{j}\pfs{f}\arrow[r]\arrow[d, "\simeq"] & \pfs{f}\arrow[r]\arrow[d, "\simeq"] & \pf{s}\pb{s}\pfs{f} \arrow[d, "\simeq"] \\
				\pfs{j}\pfs{f'}\pb{j'}\arrow[r]\arrow[d, "\simeq"] & \pfs{f}\arrow[r]\arrow[d, "\simeq"]& \pf{s}\pfs{f'}\pb{s'}\arrow[d, dashed, "\simeq"] \\
				\pfs{f}\pfs{j'}\pb{j'}\arrow[r] & \pfs{f}\arrow[r] & \pfs{f}\pf{s'}\pb{s'}.
			\end{tikzcd}
			$$ 
			Hence we may conclude by precomposing the dotted equivalence with $\pf{s'}$, since $\pf{s'}$ is fully faithful.
		\end{proof}
		\begin{remark}
			It is not hard to see that one may deduce from \Cref{localizationtheorem} that $\pf{i}$ is fully faithful (this was already proven in \cite[Corollary 7.3.2.10]{lurie2009higher}). From this follows immediately that, in the case of sheaves of pointed spaces or of spectra, one has the identification
			$$\pbp{i}\simeq \text{fib}\begin{tikzcd}
				(\pb{i}\arrow[r, "\pb{i}(unit)"] & \pb{i}\pf{j}\pb{j}).
			\end{tikzcd}$$
		\end{remark}
		\begin{remark}
			From \Cref{localizationtheorem} one deduces immediately that, at least when $\C$ is presentable stable, the functors $\pb{i}$ and $\pb{j}$ are \textit{jointly conservative}, i.e. a morphism $\alpha: F\rightarrow G$ in $\Sh{X}{\C}$ is invertible if and only if both $\pb{i}(\alpha)$ and $\pb{j}(\alpha)$ are invertible. This implies in particular that the fully faithful functors $\pf{i}$ and $\pf{j}$ make $\Sh{X}{\C}$ a \textit{recollement} of $\Sh{Z}{\C}$ and $\Sh{U}{\C}$, in the sense of \cite[Definition A.8.1]{lurie2017higher}. However, this is true in a much greater generality: see \cite{haine2025nonabelian} for a proof in the cases when ${\C}$ is an $\infty$-topos or compactly generated. After \cref{shsptensCtoshC}, for stable coefficients, we will also be able to relax the presentability assumption to the more general requirement for $\C$ to admit both limits and colimits.
		\end{remark}
		
		\section{Pullbacks with stable bicomplete coefficients}
		From now on, unless otherwise specified, all the topological spaces we will deal with will be assumed to be \textit{locally compact and Hausdorff}. This implies that the following are equivalent
		\begin{enumerate}
			\item $f: X\rightarrow Y$ is proper (i.e. the preimage of any compact subset of $Y$ is compact);
			\item $f$ is closed with compact fibers;
			\item $f$ is universally closed.
		\end{enumerate}
		Another important consequence of the previous assumption is that any map $\begin{tikzcd}[column sep= small]
			X\ar[r, "f"] & Y
		\end{tikzcd}$
		can be factored as a composition of a closed immersion (which is in particular proper by the characterization above), an open immersion and a proper map as follows
		\begin{equation}\label{factorization}
			\begin{tikzcd}
				X\times Y\arrow[r,"j\times id_Y"]  & X^+\times Y\arrow[d,"p"] \\
				X \arrow[u,"\Gamma_f"]\arrow[r,"f"] & Y
			\end{tikzcd}
		\end{equation}
		where $\Gamma_f$ is the graph of $f$, $\begin{tikzcd}[column sep= small]
			X\ar[r, "j"] & X^+
		\end{tikzcd}$ is the inclusion of $X$ into its one point compactification and $p$ is the projection to the second coordinate: this factorization will be used very often later. Notice that one may also use the Stone-\v{C}ech compactification $\beta$ to produce a functorial functorization
		\begin{equation}\label{stonecech}
			\begin{tikzcd}
				X \arrow[rdd, hook, bend right] \arrow[rrd, "f", bend left] \arrow[rd, "j", dotted, hook] &                                             &                   \\
				& \overline{X} \arrow[d, hook] \arrow[r, "p"] \arrow[dr, phantom, "\usebox\pullback" , very near start, color=black] & Y \arrow[d, hook] \\
				& \beta X \arrow[r]                           & \beta Y          
			\end{tikzcd}
		\end{equation}
		where $j$ is an open immersion and $p$ is proper. However, in many cases it will turn out to be more convenient to have the proper map in the factorization to be a product projection.
		
		This section contains the central technical ingredients of our paper. We will explain what are the two fundamental facts that allow us to construct pullbacks with non-presentable coefficients: in brief, these are \textit{covariant Verdier duality} (see \cite[Theorem 5.5.5.1]{lurie2017higher}) and the dualizability of $\Sh{X}{\spectra}$ in $\Cocont_{\infty}^{st}$.
		
		We will start by giving an exposition of covariant Verdier Duality. Following \cite[Theorem 5.5.5.1]{lurie2017higher}, this is an equivalence between the categories of sheaves and cosheaves on a locally compact Hausdorff space. We will essentially review the proof of \cite[Theorem 5.5.5.1]{lurie2017higher}, and try to clarify a bit the last step. The reader who is aware of this result may safely skip the first part of this section. Later, we will prove that $\Sh{X}{\spectra}$ is strongly dualizable in the symmetric monoidal $\infty$-category $\Cocont_{\infty}^{st}$, by explicitely exhibiting it as a retract in $\Cocont_{\infty}^{st}$ of a compactly generated $\infty$-category. This result is not new, as it could be deduced from \cite[Proposition 21.1.7.1]{lurie2016spectral} (see also \cite{hoyoisk}): we would like to thank Peter Haine for pointing this out. We then use the above retraction to provide a convenient description of pushforwards along proper maps, that is compatible with tensor products in $\Cocont^{st}_{\infty}$. We conlude the section by showing that $\pb{f}_{\spectra}\otimes\C$ is a left adjoint of $\pf{f}^{\C}$, which is not at all immediate. It will require a use of the factorizations mentioned in the beginning, and a careful combination of all the previously mentioned results.
		
		\subsection{Recollections on covariant Verdier duality}
		We start by recalling the definition of $\cate{K}$-sheaves.
		\begin{definition}\label{KSh}
			Let $\cpt{X}$ be the poset of compact subsets of a locally compact Hausdorff topological space $X$, and let $F : \cpt{X}\op\rightarrow\C$ be any functor. Consider the following conditions:
			\begin{enumerate}[(i)]
				\item $F(\emptyset)$ is a terminal object,
				\item For every $K, K^{\prime}\in\cpt{X}$ the square
				$$\begin{tikzcd}
					\Sec{K\cup K^{\prime}}{F} \arrow[r] \arrow[d] 
					\arrow[dr, phantom, "\usebox\pullback" , very near start, color=black]
					& \Sec{K}{F} \arrow[d] \\
					\Sec{K^{\prime}}{F} \arrow[r] & \Sec{K\cap K^{\prime}}{F}
				\end{tikzcd}$$
				is pullback,
				\item For every $K\in\cpt{X}$, the canonical map
				$$\varinjlim_{K\Subset K^{\prime}}\Sec{K^{\prime}}{F}\rightarrow \Sec{K}{F}$$ is invertible, where $K\Subset K^{\prime}$ means that $K^{\prime}$ contains an open neighbourhood of $K$.
			\end{enumerate}
			Notice that (i) and (ii) together are equivalent to the sheaf condition for the Grothendieck topology on $\cpt{X}$ given by finite coverings. Hence, we will denote by $\Sh{\cpt{X}}{\C}$ the full subcategory spanned by presheaves satifying (i) and (ii). Moreover, we will say that a functor $F : \cpt{X}\op\rightarrow\C$ is a \textit{$\cate{K}$-sheaf} if it satisfies (i), (ii) and (iii). We will denote by $\KSh{X}{\C}\subseteq \Fun(\cpt{X}\op ,\C)$ the full subcategory spanned by $\cate{K}$-sheaves.
		\end{definition}
		
		It is possible to relate $\cate{K}$-sheaves with usual sheaves. Let $M(X)$ be the union $\Op{X}\cup\cpt{X}$ considered as a poset contained in the power set of $X$, and let $i:\Op{X}\hookrightarrow M(X)$ and $j:\cpt{X}\hookrightarrow M(X)$ be the corresponding inclusions. We thus get two adjunctions
		$$\begin{tikzcd}
			\Fun(\Op{X}\op , \C)\ar[r,bend left,"\pfp{i}",""{name=A, below}] & \Fun(M(X)\op , \C)\ar[l,bend left,"\pb{i}",""{name=B,above}] \ar[from=A, to=B, symbol=\dashv] \ar[r,bend left,"\pb{j}",""{name=A, below}] & \Fun(\cpt{X}\op , \C).\ar[l,bend left,"\pf{j}",""{name=B,above}] \ar[from=A, to=B, symbol=\dashv]
		\end{tikzcd}$$
		More explicitly, at the level of objects the functors are given by the formulas
		$$\begin{tikzcd}[column sep = large, row sep = tiny]
			\Fun(\Op{X}\op , \C)\ar[r, "\theta"]  & \Fun(\cpt{X}\op , \C) \\
			F \arrow[r, maps to] & (K \mapsto\varinjlim\limits_{K\subseteq U}\Sec{U}{F})	
		\end{tikzcd}
		$$
		$$\begin{tikzcd}[column sep = large, row sep = tiny]
			\Fun(\cpt{X}\op , \C)\ar[r, "\psi"]  & \Fun(\Op{X}\op , \C) \\
			G \arrow[r, maps to] & (U \mapsto\varprojlim\limits_{K\subseteq U}\Sec{K}{G}),
		\end{tikzcd}
		$$
		where $\theta\coloneqq\pb{j}\pfp{i}$ and $\psi\coloneqq\pb{i}\pf{j}$. These two functors actually restrict to an equivalence, assuming $\C$ has limits and colimits, and filtered colimits are exact.
		
		\begin{theorem}\label{Ksh=sh}
			Let $\C$ be a bicomplete $\infty$-category where filtered colimits are exact, and let $X$ be a locally compact Hausdorff topological space. Then the functors $\theta$ and $\psi$ defined above restrict to an equivalence
			$$\Sh{X}{\C}\simeq\KSh{X}{\C}.$$
		\end{theorem}
		\begin{proof}
			A proof can be found in \cite[Theorem 7.3.4.9]{lurie2009higher}. There $\C$ is assumed to be presentable but, as observed in \cite[Lemma 5.5.5.3]{lurie2017higher}, one just needs $\C$ to have limits and colimits.
		\end{proof}
		
		\begin{remark}
			Since in any stable $\infty$-category $\C$ filtered colimits are exact, and since the opposite of any stable $\infty$-category is again stable, by \cref{Ksh=sh} we get equivalences
			$$\Sh{X}{\C}\simeq\KSh{X}{\C}$$
			and
			$$\coSh{X}{\C}\simeq\Sh{X}{\C\op }\op \simeq\KcoSh{X}{\C}$$
			where we define $\KcoSh{X}{\C}$ to be $\KSh{X}{\C\op }\op $.
		\end{remark}
		\begin{definition}\label{cptsec}	Let $F\in\Sh{X}{\C}$, $U\in\Op{X}$, and $K$ any closed subset of $X$. We define the \textit{sections of $F$ supported at $K$} and \textit{compactly supported sections of $F$ over $U$} respectively as
			$$\KSec{X}{F}{K} \coloneqq \text{fib}(\Sec{X}{F}\rightarrow\Sec{X\setminus K}{F})$$ 
			$$\cSec{U}{F}\coloneqq \varinjlim_{K\subseteq U}\KSec{X}{F}{K},$$
			where the colimit ranges over all compact subsets of $U$.
		\end{definition}
		\begin{remark}\label{remK} Notice that, if $K\subseteq U$ for some open $U$, we get a pullback square
			$$\begin{tikzcd}
				\Sec{X}{F} \arrow[r] \arrow[d] 
				\arrow[dr, phantom, "\usebox\pullback" , very near start, color=black]
				& \Sec{X\setminus K}{F} \arrow[d] \\
				\Sec{U}{F} \arrow[r] & \Sec{U\setminus K}{F}
			\end{tikzcd}$$
			and hence the fibers of the two horizontal maps coincide: for this reason, we will often also $\KSec{X}{F}{K}$ by $\KSec{U}{F}{K}$. Furthermore, if $S\subseteq X$ is locally closed, we define
			$$\KSec{X}{F}{S}\coloneqq\KSec{U}{\pb{j}F}{Z}$$
			where $S = U\cap Z$, with $U$ open, $Z$ closed and $j: U\hookrightarrow X$ the inclusion, but we will also use the notation $\KSec{U}{F}{S}$.
		\end{remark}
		\begin{remark}
			The definition of the sections of a sheaf $F$ on a compact $K$ is functorial both in $F$ and in $K$: since we have an obvious functor
			$$\begin{tikzcd}[row sep = tiny]
				\cpt{X} \arrow[r]    & {\Fun(\Delta^1, \Op{X}\op )} \\
				K \arrow[r, maps to] & (X\rightarrow X\setminus K)        
			\end{tikzcd}
			$$
			we get 
			$$\begin{tikzcd}
				{\cpt{X}\times\Fun(\Op{X}\op , \C)} \arrow[d] \arrow[r]                  & {\Fun(\Delta^1, \C)} \arrow[r, "\text{fib}"] & \C \\
				{\Fun(\Delta^1, \Op{X}\op )\times\Fun(\Op{X}\op , \C)} \arrow[ru, "\o"] &                                            &         
			\end{tikzcd}$$
			where the diagonal arrow is the composition of functors and the right horizontal arrow is given by taking the fiber of an arrow in $\C$, and so by adjunction we get the desired 
			$$\begin{tikzcd}[column sep = large, row sep = tiny]
				\Fun(\Op{X}\op , \C)\ar[r]  & \Fun(\cpt{X}, \C) \\
				F \arrow[r, maps to] & (K \mapsto\KSec{X}{F}{K}).	
			\end{tikzcd}$$ 
			Finally, by further composing with the functor $\psi$ defined in \Cref{Ksh=sh}, we get
			\begin{equation}\label{verdier}
				\begin{tikzcd}[column sep = large, row sep = tiny]
					\Fun(\Op{X}\op , \C)\ar[r, "\mathbb{D}_{\C}"]  & \Fun(\Op{X}, \C) \\
					F \arrow[r, maps to] & (U \mapsto\cSec{U}{F}).	
				\end{tikzcd}	
			\end{equation}
		\end{remark}
		\begin{theorem}\label{covverddual}
			Let $\C$ be a stable and bicomplete $\infty$-category. The functor (\ref{verdier}) restricts to an equivalence
			$$\mathbb{D}_{\C}:\begin{tikzcd}
				\Sh{X} {\C}\ar[r, "\simeq"] &\coSh{X}{\C}.
			\end{tikzcd}$$
		\end{theorem}
		\begin{proof}
			We first prove that, if $F$ is a sheaf, then $\mathbb{D}_{\C}(F)$ is a cosheaf. By virtue of \Cref{Ksh=sh}, it suffices to prove that the functor
			$$K \mapsto\KSec{X}{F}{K}$$
			is a $\cate{K}$-cosheaf.
			\begin{itemize}
				\item $\KSec{X}{F}{\emptyset}\simeq 0$ since $F(X)\rightarrow F(X\setminus\emptyset)$ is an equivalence.
				\item Let $K, K^{\prime}\in\cpt{X}$. The square
				$$\begin{tikzcd}
					\KSec{X}{F}{K\cap K^{\prime}} \arrow[d] \arrow[r] & \KSec{X}{F}{K} \arrow[d]      \\
					\KSec{X}{F}{K^{\prime}} \arrow[r]                 & \KSec{X}{F}{K\cup K^{\prime}}
				\end{tikzcd}$$
				is the fiber of the obvious map between the pullback squares
				$$\begin{tikzcd}
					\Sec{X}{F} \arrow[r] \arrow[d] & \Sec{X}{F} \arrow[d] &  & \Sec{X\setminus(K\cap K^{\prime})}{F} \arrow[d] \arrow[r] & \Sec{X\setminus K}{F} \arrow[d]       \\
					\Sec{X}{F} \arrow[r]           & \Sec{X}{F}           &  & \Sec{X\setminus K^{\prime}}{F} \arrow[r]                  & \Sec{X\setminus(K\cup K^{\prime})}{F},
				\end{tikzcd}$$
				and so it is a pullback. Thus, since $\C$ is stable, it's also a pushout.
				\item For any $K\in\cpt{X}$, we have a map of fiber sequences
				$$\begin{tikzcd}
					\KSec{X}{F}{K} \arrow[rr, "a"] \arrow[d] &  & \varprojlim\limits_{K\Subset K^{\prime}}\KSec{X}{F}{K^{\prime}} \arrow[d] \\
					\Sec{X}{F} \arrow[rr, "b"] \arrow[d]     &  & \varprojlim\limits_{K\Subset K^{\prime}}\Sec{X}{F} \arrow[d]              \\
					\Sec{X\setminus K}{F} \arrow[rr, "c"]    &  & \varprojlim\limits_{K\Subset K^{\prime}}\Sec{X\setminus K^{\prime}}{F}.  
				\end{tikzcd}$$
				To prove that $a$ is an equivalence, it suffices to prove that $b$ and $c$ are. But $b$ is an equivalence because the poset $\{K^{\prime}\in\cpt{X}\mid K\Subset K^{\prime}\}$ has a contractible nerve (since it is filtered) and $c$ is an equivalence because $\{U\in\Op{X}\mid U=X\setminus K^{\prime} \, \text{for some}\, K\Subset K^{\prime}\}$ gives an open covering of $X\setminus K$.
			\end{itemize}
			We will now prove that $\mathbb{D}_{\C\op }\op $ is an inverse of $\mathbb{D}_{\C}$. By symmetry, it suffices to show that it is a left inverse. Unraveling the definitions and using the equivalence of \Cref{Ksh=sh}, this amounts to check that we have a cofiber sequence 
			\begin{equation}\label{invverd}
				\begin{tikzcd}
					\cSec{X\setminus K}{F} \arrow[r] & \cSec{X}{F} \arrow[r] & \Sec{K}{F}
				\end{tikzcd}	
			\end{equation}
			natural in $K$ and $F$. 
			\par First of all, we show that it suffices to prove that, for any fixed $K\in\cpt{X}$, $U\in\Op{X}$ containing $K$ and with compact closure, and $K^{\prime}\in\cpt{X}$ containing $U$, the sequence
			\begin{equation}\label{invverdp}
				\begin{tikzcd}
					\KSec{X}{F}{K^{\prime}\setminus U} \arrow[r] & \KSec{X}{F}{K^{\prime}} \arrow[r] & \Sec{U}{F},
				\end{tikzcd}
			\end{equation}
			where the first morphism is given by the functoriality of sections supported on a compact and the second one is given by \Cref{remK}, is a cofiber sequence. To see this, we start by noticing that the sequence is natural in $K^{\prime}$ and $F$, since both morphisms are canonically induced by the restrictions of $F$. Thus we can pass to the colimit ranging over all compacts $K^{\prime}\supseteq U$ and get a fiber sequence 
			$$\begin{tikzcd}
				\varinjlim\limits_{K^{\prime}\supseteq U}\KSec{X}{F}{K^{\prime}\setminus U} \arrow[r] & \cSec{X}{F} \arrow[r] & \Sec{U}{F},
			\end{tikzcd}$$
			since the poset $\{K^{\prime}\in\cpt{X}\mid K^{\prime}\supseteq U\}$ is filtered (it is non-empty because it contains the closure of $U$) and the inclusion $\{K^{\prime}\in\cpt{X}\mid K^{\prime}\supseteq U\}\subseteq \cpt{X}$ is cofinal. Since $K^{\prime}\supseteq U$ implies that $K'\subseteq(K^{\prime}\cup\overline{U})\setminus U$, we get an equivalence 
			$$\varinjlim\limits_{K^{\prime}\supseteq U}\KSec{X}{F}{K^{\prime}\setminus U}\simeq \varinjlim\limits_{\{K^{\prime}\mid K^{\prime}\cap U = \emptyset\}}\KSec{X}{F}{K^{\prime}},$$
			and hence, adding everything up, we obtain a fiber sequence
			$$\begin{tikzcd}
				\varinjlim\limits_{\{K^{\prime}\mid K^{\prime}\cap U = \emptyset\}}\KSec{X}{F}{K^{\prime}}\arrow[r] & \cSec{X}{F} \arrow[r] & \Sec{U}{F},
			\end{tikzcd}$$
			which is natural in $U$, since the morphism $\cSec{X}{F} \rightarrow \Sec{U}{F}$ clearly is. Hence we can get the desired sequence (\ref{invverd}) by passing to the colimit ranging over $P = \{U\in\Op{X}\mid \overline{U}\in\cpt{X}\, \text{and}\, U\supseteq K\}$ because $\Sec{K}{F} = \varinjlim\limits_{U\in P}\Sec{U}{F}$  (since open subsets with compact closure form a basis of $X$), and because we have equivalences 
			\begin{align*}
				\varinjlim\limits_{U\in P}\varinjlim\limits_{\{K^{\prime}\mid K^{\prime}\cap U = \emptyset\}}\KSec{X}{F}{K^{\prime}} &\simeq \varinjlim\limits_{\bigcup\limits_{U\in P}\{K^{\prime}\mid K^{\prime}\cap U = \emptyset\}}\KSec{X}{F}{K^{\prime}} \\
				&\simeq \varinjlim\limits_{K^{\prime}\subseteq X\setminus K}\KSec{X}{F}{K^{\prime}} \\
				&\simeq\cSec{X\setminus K}{F}
			\end{align*} 
			where the first one follows from \cite[Remark 4.2.3.9]{lurie2009higher} and \cite[Corollary 4.2.3.10]{lurie2009higher}.
			\par 
			We are now left to show that (\ref{invverdp}) is a cofiber sequence. Consider the commutative diagram
			$$\begin{tikzcd}[row sep = large]
				\KSec{X}{F}{K^{\prime}\setminus U} \arrow[rr] \arrow[d] &  & 0 \arrow[d]                                                     &  &                                          \\
				\KSec{X}{F}{K^{\prime}} \arrow[rr] \arrow[d]            &  & Z \arrow[d] \arrow[rr]                                          &  & 0      \arrow[d]                                  \\
				\Sec{X}{F} \arrow[rr]                                   &  & \Sec{X\setminus(K^{\prime}\setminus U)}{F} \arrow[rr] \arrow[d] &  & \Sec{X\setminus K^{\prime}}{F} \arrow[d] \\
				&  & \Sec{U}{F} \arrow[rr]                                           &  & 0                                       
			\end{tikzcd}$$
			where $Z\coloneqq \text{fib}(\Sec{X\setminus(K^{\prime}\setminus U)}{F} \rightarrow \Sec{X\setminus K^{\prime}}{F})$. Since the middle big horizontal rectangle is a pullback, it follows that also the left middle square is. But since the left vertical rectangle is pullback, then also the upper left square is. Then it suffices to prove that the composition $$Z\rightarrow \Sec{X\setminus(K^{\prime}\setminus U)}{F} \rightarrow \Sec{U}{F}$$
			is an equivalence, but this is clear since the lower right square is pullback because $F$ is a sheaf.
		\end{proof}

		\subsection{Dualizability of spectral sheaves}
		
		\begin{lemma}\label{Kshretract}
			Let $\C$ be a bicomplete $\infty$-category where filtered colimits are exact. Then there exists a functor $\varphi: \Sh{\cpt{X}}{\C}\rightarrow\KSh{X}{\C}$ satisfying the following properties
			\begin{enumerate}
				\item for any $F\in\Sh{\cpt{X}}{\C}$ and $K\in\cpt{X}$, we have
				$$\Sec{K}{\varphi F}\simeq\varinjlim_{K\Subset K^{\prime}}\Sec{K^{\prime}}{F},$$
				\item $\varphi$ preserves filtered colimits,
				\item it is a left inverse to the inclusion $\KSh{X}{\C}\hookrightarrow\Sh{\cpt{X}}{\C}$.
			\end{enumerate}
		\end{lemma}
		\begin{proof} 
			Let $M$ be the set whose elements are pairs $(K, i)$ with $K\in\cpt{X}$ and $i = 0, 1$, where we define $(K, i)\leq(K', j)$ if $K'\subseteq K$  and  $i =  j$ or $K'\Subset K$  and  $i < j$. It is easy to see that $\leq$ actually defines a partial order on $M$. We have two functors
			$$
			\begin{tikzcd}[row sep = tiny]
				\cpt{X}\op \ar[r, "i_0"] & M \\
				K \ar[r, mapsto] & (K, 0)
			\end{tikzcd}
			\begin{tikzcd}[row sep = tiny]
				\cpt{X}\op \ar[r, "i_1"] & M \\
				K \ar[r, mapsto] & (K, 1)
			\end{tikzcd}
			$$
			and consider the cocontinuous functor
			$$\varphi \coloneqq \pb{(i_1)}\pfp{(i_0)} : \Fun(\cpt{X}\op , \C)\rightarrow\Fun(\cpt{X}\op , \C).$$
			Unravelling the definition, we see that
			$$\Sec{K}{\varphi F}\simeq\varinjlim_{(K', 0)\rightarrow (K, 1)}\Sec{K^{\prime}}{F}\simeq\varinjlim_{K\Subset K^{\prime}}\Sec{K^{\prime}}{F},$$
			and so $\varphi F\simeq F$ whenever $F$ is a $\cate{K}$-sheaf. Hence, to conclude the proof, it suffices to show that the essential image of $\varphi_{|_{\Sh{\cpt{X}}{\C}}}$ is contained in $\KSh{X}{\C}$.
			\par 
			Indeed, for any $F\in\Sh{\cpt{X}}{\C}$ we have
			\begin{align*}
				\varinjlim_{K\Subset K''}\Sec{K''}{\varphi F}&\simeq\varinjlim_{K\Subset K''}\varinjlim_{K''\Subset K'}\Sec{K'}{F} \\
				&\simeq \varinjlim_{K\Subset K'}\Sec{K'}{F} \\
				& \simeq\Sec{K}{\varphi F},
			\end{align*}
			where the second equivalence holds by \cite[Remark 4.2.3.9]{lurie2009higher} and \cite[Corollary 4.2.3.10]{lurie2009higher} since for any $K$ compact, the full subposet of $\cpt{X}\op $ spanned by those $K'$ such that $K\Subset K'$ is filtered, and we have
			$$\bigcup_{\{K''| K\Subset K''\}}\{K'| K''\Subset K'\} = \{K'| K\Subset K'\}.$$
			Moreover, since filtered colimits are exact in $\C$, $\varphi F$ belongs to $\Sh{\cpt{X}}{\C}$, and so it is a $\cate{K}$-sheaf.
		\end{proof}
		
		\cref{Kshretract} is useful to describe conveniently pushforwards of sheaves along proper maps. Let $f : X\rightarrow Y$ be a proper continuous map between topological spaces, and let again $\C$ be a bicomplete $\infty$-category where filtered colimits are exact. Since for any $K\in\cpt{Y}$, the preimage $f^{-1}(K)$ is compact, we obtain a functor 
		$$
		\begin{tikzcd}[row sep = tiny]
			\Fun(\cpt{X}\op , \C)\ar[r, "{f_+}"] & \Fun(\cpt{Y}\op , \C) \\
			F\ar[r, mapsto] & (K\mapsto\Sec{f^{-1}(K)}{F}).
		\end{tikzcd}
		$$ 
		Notice that the restriction of $f_+$ to $\KSh{X}{\C}$ lands in $\Sh{\cpt{Y}}{\C}$, but a priori not in $\KSh{Y}{\C}$, therefore we define
		$$\pf{f}^{\cate{K}} : \KSh{X}{\C}\rightarrow\KSh{Y}{\C}$$
		as the composition of $f_+$ restricted to $\KSh{X}{\C}$ and $\varphi: \Sh{\cpt{Y}}{\C}\rightarrow\KSh{Y}{\C}$.
		\begin{lemma}\label{properpsfcocont}
			Let $f : X\rightarrow Y$ be a proper continuous map between topological spaces, and let $\C$ be bicomplete $\infty$-category where filtered colimits are exact. Then there is a commutative diagram
			$$
			\begin{tikzcd}
				\Sh{X}{\C}\ar[r, "\theta"]\ar[d, "\pf{f}"] & \KSh{X}{\C}\ar[d, "{\pf{f}^{\cate{K}}}"] \\
				\Sh{Y}{\C}\ar[r, "\theta"] & \KSh{Y}{\C},
			\end{tikzcd}
			$$
			where $\theta$ is as in \cref{Ksh=sh}.
			In particular, $\pf{f}$ preserves filtered colimits, and when $\C$ is stable it preserves all colimits.
		\end{lemma}
		\begin{proof}
			For any $K\in\cpt{Y}$, we define
			$$T = \{U\in\Op{X} \mid \exists K'\Supset K \:\text{with}\: f^{-1}(K')\subseteq U\}.$$
			Notice that if $V\in\Op{Y}$ contains $K$, then there exists an open neighbourhood $W$ of $K$ with compact closure in $V$, and thus, since $f$ is proper, $f^{-1}(V)\in T$. In particular we obtain a functor 
			$$\alpha: \{V\in\Op{Y}\mid K\subseteq V\}\rightarrow T$$
			which is obviously final. We have
			\begin{align*}
				\Sec{K}{\pf{f}^{\cate{K}}\theta F}&\simeq \varinjlim_{K\Subset K'}\varinjlim_{f^{-1}(K')\subseteq U}\Sec{U}{F} \\
				&\simeq \varinjlim_{U\in T}\Sec{U}{F} \\
				&\simeq \varinjlim_{K\subseteq V}\Sec{f^{-1}(V)}{F} \\
				&\simeq\Sec{K}{\theta\pf{f}F}
			\end{align*}
			where the second equivalence follows from \cite[Remark 4.2.3.9]{lurie2009higher} and \cite[Corollary 4.2.3.10]{lurie2009higher}, and the third one since $\alpha$ is final.
		\end{proof}
		
		Another straightforward application of \cref{Kshretract} is the following theorem, that will be of crucial importance for what follows.
		
		\begin{theorem}\label{coshsttens}
			The $\infty$-category $\Sh{X}{\spectra}$ is a strongly dualizable object in $\Cocont^{st}_{\infty}$. In particular, for any $\C\in\Cocont^{st}_{\infty}$, the canonical functor
			$$\coSh{X}{\spectra}\otimes\C\rightarrow\coSh{X}{\C}$$
			is an equivalence.
		\end{theorem}
		\begin{proof}
			By \Cref{cptlygenstdual}, \Cref{Ksh=sh} and \cref{dualretract}, it suffices to show that $\KSh{X}{\spectra}$ is a retract in $\Cocont^{st}_{\infty}$ of a compactly generated $\infty$-category. We see that $\Sh{\cpt{X}}{\spectra}$ is clearly compactly generated, as it is equivalent to $\Fun_{lex}(\cpt{X}\op ,\spectra)$. The proof is then concluded by noticing that the inclusion $\KSh{X}{\spectra}\subseteq\Sh{\cpt{X}}{\spectra}$ and $\varphi: \Sh{\cpt{X}}{\spectra}\rightarrow\KSh{X}{\spectra}$ are exact and preserve filtered colimits since filtered colimits in $\spectra$ are exact, and thus preserves all colimits since $\spectra$ is stable. One then concludes the proof by identifying $\coSh{X}{\spectra}$ with the dual of $\Sh{X}{\spectra}$, via \cref{coshdistr}.
		\end{proof}
		
		\begin{corollary}\label{shsptensCtoshC}
			There is a unique equivalence
			\begin{equation}
				\eta : \Sh{X}{\spectra}\otimes\C\rightarrow\Sh{X}{\C}.
			\end{equation}
			making the diagram 
			$$
			\begin{tikzcd}
				\Sh{X}{\spectra}\otimes\C\arrow[r, "\eta"]\arrow[d, "{\mathbb{D}_{\spectra}\otimes\C}"] & \Sh{X}{\C} \arrow[d, "{\mathbb{D}_{\C}}"] \\
				\coSh{X}{\spectra}\otimes\C \arrow[r]  & \coSh{X}{\C}
			\end{tikzcd}
			$$
		\end{corollary}
		\begin{proof}
			The functor $\eta$ is obtained by composing
			$$
			\begin{tikzcd}
				\Sh{X}{\spectra}\otimes\C\ar[r, "{\mathbb{D}_{\spectra}\otimes\C}"] & \coSh{X}{\spectra}\otimes\C \ar[r, "\simeq"] & \coSh{X}{\C}
				\ar[r, "{\mathbb{D}_{\C}^{-1}}"] & \Sh{X}{\C}.
			\end{tikzcd}
			$$
			where the middle map is the one in \cref{coshsttens}.
			More concretely, for any $F\in\Sh{X}{\spectra}$ and $M\in\C$, we have $\eta(F\boxtimes^{st} M)\simeq\mathbb{D}_{\C}^{-1}(M\o\mathbb{D}_{\spectra}F)$, where $M$ on the right-hand side denotes the essentially unique colimit preserving functor $\spectra\rightarrow\C$ corresponding to $M$.
		\end{proof}
		
		\subsection{The pullback $\pb{f}_{\mathscr{C}}$}
		
		\begin{lemma}\label{pushforwardandVD}
			Let $f:X\rightarrow Y$ be a continuous map of locally compact Hausdorff spaces, and let $\C$ be any stable bicomplete $\infty$-category. Then there is a natural transformation $(\pf{f}^{\C\op})\op\mathbb{D}_{\C}\rightarrow\mathbb{D}_{\C}\pf{f}^{\C}$, which is invertible when $f$ is proper.
		\end{lemma}
		
		\begin{proof}
			We first construct a natural transformation $\pf{f}\mathbb{D}\rightarrow\mathbb{D}\pf{f}$. Fix $V\in\Op{Y}$ and a compact $K\subseteq f^{-1}(V)$, so that $f(K)$ is a compact subset of $V$. For any $F\in\Sh{X}{\C}$, the commutative triangle
			$$
			\begin{tikzcd}[column sep = small]
				\Sec{X\setminus K}{F}\ar[rr] & &\Sec{X\setminus f^{-1}(f(K))}{F}\\
				&\Sec{X}{F}\ar[ru]\ar[lu]&
			\end{tikzcd}
			$$	
			provides a morphism
			$$\KSec{X}{F}{K}\rightarrow\KSec{Y}{\pf{f}F}{f(K)}\rightarrow\cSec{V}{\pf{f}F}.$$
			Since all morphisms are induced by the restrictions of $F$, the resulting map is natural in $K$ and $V$, and hence gives rise to the desired transformation as $K$ varies. Furthermore, when $f$ is proper, each compact $K\subseteq X$ is contained in the compact $f^{-1}f(K)$, so by cofinality we obtain an equivalence
			$$\varinjlim_{K\subseteq f^{-1}(V)}\KSec{X}{F}{K}\simeq\varinjlim_{C\subseteq V}\KSec{Y}{\pf{f}F}{C}$$
			where $C$ varies over the compact subsets of $V$, and thus we may conclude.
		\end{proof}
		
		\begin{proposition}\label{proppshtens}
			Let $f: X\rightarrow Y$ be a proper map, and denote by $\pf{f}^{\C} : \Sh{X}{\C}\rightarrow\Sh{Y}{\C}$ the pushforward. Then we have a commutative square
			$$
			\begin{tikzcd}
				\Sh{X}{\spectra}\otimes\C\ar[r, "\eta"]\ar[d, "{\pf{f}^{\spectra}\otimes\C}"] & \Sh{X}{\C}\ar[d, "{\pf{f}^{\C}}"] \\
				\Sh{Y}{\spectra}\otimes\C\ar[r, "\eta"] & \Sh{Y}{\C}.
			\end{tikzcd}
			$$
			In particular, $\pf{f}^{\C}$ admits a left adjoint which is identified through $\eta$ with $\pb{f}_{\spectra}\otimes\C$. 
		\end{proposition}
		\begin{proof}
			By a slight abuse of notation, denote as $\pf{f}^{\C} : \coSh{X}{\C}\rightarrow\coSh{X}{\C}$ the pushforward for cosheaves (i.e. $(\pf{f}^{\C\op })\op $). We have that the square
			$$
			\begin{tikzcd}
				\coSh{X}{\spectra}\otimes\C\ar[r, "\simeq"]\ar[d, "{\pf{f}^{\spectra}\otimes\C}"] & \coSh{X}{\C}\ar[d, "{\pf{f}^{\C}}"] \\
				\coSh{Y}{\spectra}\otimes\C\ar[r, "\simeq"] & \coSh{Y}{\C}
			\end{tikzcd}
			$$
			commutes since the horizontal arrows can be modelled by a composition of functors, and thus we are only left to show that the square
			$$
			\begin{tikzcd}
				\Sh{X}{\C}\ar[r, "{\mathbb{D}}"]\ar[d, "{\pf{f}}"] & \coSh{X}{\C}\ar[d, "{\pf{f}}"] \\
				\Sh{Y}{\C}\ar[r, "{\mathbb{D}}"] & \coSh{Y}{\C}
			\end{tikzcd}
			$$
			commutes, which is proven in \cref{pushforwardandVD}.
		\end{proof}
		
		\begin{lemma}\label{openpbVD}
			Let $j:U\hookrightarrow X$ be a open immersion of locally compact Hausdorff spaces. Then there is a natural isomorphism $\mathbb{D}_{\C}(\pb{j}_{\C\op})\op\simeq\pb{j}_{\C}\mathbb{D}_{\C}$.
		\end{lemma}
		
		\begin{proof}
			This follows immediately by \Cref{remK}, and by observing that $\Sec{V}{\pb{j}F}\simeq\Sec{V}{F}$ for any $V\subseteq U$, we have that
			$$\cSec{V}{\pb{j}F}\simeq\cSec{V}{F}$$
			functorially on $F$ and $V$.
		\end{proof}
		
		\begin{lemma}\label{openrestrtens}
			Let $j:U \hookrightarrow X$ be an open immersion, and denote by $\pb{j}_{\C} : \Sh{X}{\C}\rightarrow\Sh{U}{\C}$ the restriction. Then we have a commutative square
			$$
			\begin{tikzcd}
				\Sh{X}{\spectra}\otimes\C\ar[r, "\eta"]\ar[d, "{\pb{j}_{\spectra}\otimes\C}"] & \Sh{X}{\C}\ar[d, "{\pb{j}_{\C}}"] \\
				\Sh{U}{\spectra}\otimes\C\ar[r, "\eta"] & \Sh{U}{\C}.
			\end{tikzcd}
			$$
			In particular, $\pb{j}_{\C}$ admits a left adjoint which is identified through $\eta$ with $\pfs{j}^{\spectra}\otimes\C$.
		\end{lemma}
		\begin{proof}
			By an abuse of notation, denote by $\pb{j}_{\C} : \coSh{X}{\C}\rightarrow\coSh{U}{\C}$ the restriction for cosheaves. Again we see that the square
			$$
			\begin{tikzcd}
				\coSh{X}{\spectra}\otimes\C\ar[r, "\simeq"]\ar[d, "{\pb{j}_{\spectra}\otimes\C}"] & \coSh{X}{\C}\ar[d, "{\pb{j}_{\C}}"] \\
				\coSh{U}{\spectra}\otimes\C\ar[r, "\simeq"] & \coSh{U}{\C}
			\end{tikzcd}
			$$
			is obviously commutative, and thus we only have to show that the square
			$$
			\begin{tikzcd}
				\Sh{X}{\C}\ar[r, "{\mathbb{D}}"]\ar[d, "{\pb{j}}"] & \coSh{X}{\C}\ar[d, "{\pb{j}}"] \\
				\Sh{U}{\C}\ar[r, "{\mathbb{D}}"] & \coSh{U}{\C}
			\end{tikzcd}
			$$
			commutes, which is given by \cref{openpbVD}.
		\end{proof}
		\begin{corollary}\label{pbstablecoeff}
			Let $f : X\rightarrow Y$ be any continuous map. Then the pushforward $\pf{f}^{\C} : \Sh{X}{\C}\rightarrow\Sh{Y}{\C}$ admits a left adjoint $\pb{f}_{\C}$ such that there exists a commutative square 
			$$
			\begin{tikzcd}
				\Sh{Y}{\spectra}\otimes\C\ar[r, "\eta"]\ar[d, "{\pb{f}_{\spectra}\otimes\C}"] & \Sh{Y}{\C}\ar[d, "{\pb{f}_{\C}}"] \\
				\Sh{X}{\spectra}\otimes\C\ar[r, "\eta"] & \Sh{X}{\C}.
			\end{tikzcd}
			$$
		\end{corollary}
		\begin{proof}
			Using the factorization \ref{factorization}, this follows immediately by \Cref{proppshtens} and \Cref{openrestrtens}.
		\end{proof}
		\begin{theorem}\label{sheafification}
			Let $i_{\C} : \Sh{X}{\C}\hookrightarrow\Fun(\Op{X}\op , \C)$ be the inclusion functor, and let $L_{\spectra} : \Fun(\Op{X}\op , \spectra)\rightarrow\Sh{X}{\spectra}$ be the left adjoint of $i_{\spectra}$. Denote by $L_{\C}$ the composition
			$$
			\begin{tikzcd}
				\Fun(\Op{X}\op , \spectra)\otimes\C\ar[r, "L_{\spectra}\otimes\C"] & \Sh{X}{\spectra}\otimes\C\ar[d, "\eta"] \\
				\Fun(\Op{X}\op ,\C)\ar[u, "\simeq"]\ar[r, dotted, "L_{\C}"] & \Sh{X}{\C}.
			\end{tikzcd}
			$$
			Then $L_{\C}$ is left adjoint to $i_{\C}$.
		\end{theorem}
		\begin{proof}
			From the proof of \cref{preshdual} and \cref{relyonstable}, we see that we may write any $F\in\Fun(\Op{X}\op , \C)$ as a colimit
			$$F\simeq \varinjlim_{M\rightarrow\Sec{U}{F}}U\boxtimes^{st} M$$
			where the indexing $\infty$-category is Grothendieck construction of the functor $$(U,M)\mapsto\Hom{\C}{M}{\Sec{U}{F}}.$$ By definition, the presheaf of spectra $U\boxtimes^{st}M$ represents the functor taking sections at $U$, and thus we get an equivalence
			$$\s{U} = L_{\spectra}(U\boxtimes\s{})\simeq \pfs{j}^{\spectra}\pb{a}_{\spectra}\s{}$$
			which is natural in $U$, where $j : U\hookrightarrow X$ denotes an open inclusion, $a:U\rightarrow\ast$ the unique map.
			We have equivalences
			\begin{align*}
				\Hom{\Sh{X}{\C}}{L_{\C }F}{G}&\simeq\varprojlim_{M\rightarrow\Sec{U}{F}}\Hom{\Sh{X}{\spectra}\otimes\C}{\pfs{j}^{\spectra}\pb{a}_{\spectra}\s{}\boxtimes^{st} M}{\eta^{-1} G} \\
				&\simeq \varprojlim_{M\rightarrow\Sec{U}{F}}\Hom{\Sh{X}{\C}}{\pfs{j}^{\C}\pb{a}_{\C} M}{ G} \\
				&\simeq \varprojlim_{M\rightarrow\Sec{U}{F}}\Hom{\C}{M}{\Sec{U}{G}} \\
				&\simeq \varprojlim_{M\rightarrow\Sec{U}{F}}\Hom{\Fun(\Op{X}\op ,\C)}{U\boxtimes M}{i^{\C}G} \\
				&\simeq \Hom{\Fun(\Op{X}\op ,\C)}{F}{i^{\C}G}
			\end{align*}
			where the first equivalence follows from the observations above, and the second is a consequence of \Cref{openrestrtens} and \Cref{pbstablecoeff}. Since all identifications are functorial on $F$ and $G$, we obtain the thesis.
		\end{proof}
		\begin{remark}\label{fromprestostbic}
			After the results in this section, we are now able to extend everything we have proven so far for sheaves with presentable stable coefficients to sheaves with values in a stable bicomplete $\infty$-category. The only detail we have to handle with more care is the functor (\ref{tensshCD}): since the tensor product of two stable bicomplete $\infty$-categories is not again complete in general, (\ref{tensshCD}) will now take values in $\Sh{X}{\spectra}\otimes(\C\otimes\D)$. Nevertheless, when $\C$ has a symmetric monoidal structure such that its tensor $\otimes_{\C}$ preserves colimits in both variables, the composition of (\ref{tensshCD}) with the obvious functor 
			$$\Sh{X}{\spectra}\otimes(\C\otimes\C)\rightarrow\Sh{X}{\C}$$ may still be enhanced to a binary multiplication in a symmetric monoidal structure on $\Sh{X}{\C}$. As a consequence, we see that the equivalences in \Cref{monpb} and \Cref{smoothproj} still hold in $\Sh{X}{\spectra}\otimes(\C\otimes\D)$ and in $\Sh{X}{\C}$ when (\ref{tensshCD}) is exchanged with the tensor product in $\Sh{X}{\C}$ described above.
		\end{remark}

		\section{Six functor formalism} 
		In this section, we will define the operations $\pfp{f}$ and $\pbp{f}$, and prove all the usual formulas that one expects for these functors. A first attempt towards these results for sheaves of spectra can be found in the paper \cite{block1996homotopy}, even though it almost totally lacks proofs. The idea of using covariant Verdier duality to define the !-operations first appeared in \cite[Lecture 21]{lurietamagawa}. In this work, the author states that, under some unspecified mild assumptions on the coefficients, the pushforward of cosheaves admits a right adjoint. Its Verdier dual is then defined to be the exceptional pullback. However, \cite{lurietamagawa} does not provide any explanation for why one should expect to have this right adjoint for stable and bicomplete coefficients. The contents of the previous section fill this gap. A proof of the proper base change theorem with unstable coefficients was provided in \cite[Corollary 7.3.1.18]{lurie2009higher}, and it was later extended to spectral coefficients in \cite{haine2025nonabelian}. Our only contribution to this theorem is to explain how to further generalize it to general stable bicomplete coefficients. 
		
		A novelty of the approach presented in this section is our expression of the formulas involving tensor products, such as projection or K\"unneth formula. Here we do not a priori require the coefficients of our sheaves to be equipped with a symmetric monoidal structure, but rely instead on the tensor product of stable cocomplete $\infty$-categories. The advantage of this perspective is that, using the observations in \cref{fromprestostbic}, it clarifies how to obtain all these formulas for a general stable bicomplete $\infty$-category equipped with a closed symmetric monoidal structure. At the end of the section, building up on the our discussion of Section 3 related to shape theory, we will explain how to prove the formula
		$$\pbp{f}(\mathds{1})\otimes\pb{f}\simeq\pbp{f}$$
		for any map $f$ which induces a locally contractible geometric morphisms. 
		
		\subsection{The formulas for $\pfp{f}^{\C}$}
		
		Throughout this section, $\C$ is going to be any stable and bicomplete $\infty$-category. We will also implicitely assume all topological spaces to locally compact and Hausdorff. For any continuous map $f: X\rightarrow Y$, consider the functor 
		$$(\pf{f}^{\C^{op}})\op : \Sh{X}{\C\op}\op = \coSh{X}{\C}\rightarrow\Sh{Y}{\C\op}\op = \coSh{Y}{\C}.$$
		Since taking opposite categories switches left with right adjoints, the functor
		$$(\pb{f}_{\C\op})\op : \Sh{Y}{\C\op}\op = \coSh{Y}{\C}\rightarrow\Sh{X}{\C\op}\op = \coSh{X}{\C}$$
		is right adjoint to $(\pf{f}^{\C\op})\op$. Hence, by \cref{covverddual}, we get a corresponding adjunction at the level of sheaves
		$$\begin{tikzcd}
			\Sh{X}{\C}\ar[r,bend left,"\pfp{f}^{\C}",""{name=A, below}] & \Sh{Y}{\C}.\ar[l,bend left,"\pbp{f}_{\C}",""{name=B,above}] \ar[from=A, to=B, symbol=\dashv]
		\end{tikzcd}$$
		where $\pfp{f}^{\C}$ and $\pbp{f}_{\C}$ are defined to be the unique functors fitting in commutative squares
		$$
		\begin{tikzcd}
			\Sh{X}{\C} \arrow[d, "\mathbb{D}"'] \arrow[r, "\pfp{f}^{\C}"] & \Sh{Y}{\C} \arrow[d, "\mathbb{D}"] &  & \Sh{Y}{\C} \arrow[r, "\pbp{f}_{\C}"] \arrow[d, "\mathbb{D}"'] & \Sh{X}{\C} \arrow[d, "\mathbb{D}"] \\
			\coSh{X}{\C} \arrow[r, "{(\pf{f}^{\C^{op}})\op }"]                                   & \coSh{Y}{\C}                       &  & \coSh{Y}{\C} \arrow[r, "{(\pb{f}_{\C\op})\op}"]                                   & \coSh{X}{\C}.                      
		\end{tikzcd}
		$$
		\begin{definition}
			The functors $\pfp{f}^{\C}$ and $\pbp{f}_{\C}$ constructed as above are called respectively \textit{pushforward with proper support} and \textit{exceptional pullback}. Unless it is required from the context, we will often omit to include $\C$ in subscripts or superscripts in our notation.
		\end{definition}
		
		More concretely, $\pfp{f}$ is the functor uniquely determined by the formula
		$$\cSec{U}{\pfp{f}F} = \cSec{f^{-1}(U)}{F}$$
		for all $U\in\Op{Y}$. In particular, when $a:X\rightarrow \ast$ is the unique map, we get 
		$$\pfp{a}F \simeq\cSec{X}{F}.$$
		In different geometric contexts, when one deals with six functor formalisms, it is common to define the shriek operations making use of appropriate compactifications of maps analogous to (\ref{factorization}). This approach, however, makes it a bit tricky to verify that $\pfp{f}$ behaves well under compositions. An advantage of our definition of $\pfp{f}$ is that its functoriality is more or less immediate, as illustrated by the following lemma.
		
		\begin{lemma}\label{functpfp}
			Let $\text{LCH}$ be the category of locally compact Hausdorff spaces. Then, for any $\C$ stable and bicomplete, there is a functor
			$$\spaces\text{hv}_!(-;\C) : \text{LCH}\rightarrow\Cocont_{\infty}^{st}$$
			whose values on an object $X$ is given by $\Sh{X}{\C}$, and on a morphism $f$ is given by $\pfp{f}^{\C}$.
		\end{lemma}
		
		\begin{proof}
			We first observe that there is a functor
			$$\spaces\text{hv}_*(-;\C) : \text{LCH}\rightarrow\infcat$$
			whose value on a morphism $f:X\rightarrow Y$ is given by the pushforward $\pf{f}^{\C}$. By definition, $\pf{f}^{\C}$ is given by precomposing with the functor
			$f^{-1} : \cate{U}(Y)\rightarrow\cate{U}(X).$
			Thus, the functoriality of $\spaces\text{hv}_*(-;\C)$ descends from the functoriality of internal-homs in $\infcat$, which is straightforward. By passing to opposite categories, we obtain a similar functor $\C\text{o}\spaces\text{hv}_*(-;\C)$.
			
			Let $J$ be the interval object for the Joyal model structure (see \cite[Definition 3.3.3]{cisinski2019higher}), and consider the monomorphism of simplicial sets $Ob(\text{LCH})\hookrightarrow\text{LCH}$. By \cref{covverddual}, we have a functor	
			$$
			\begin{tikzcd}[column sep= huge]
				Ob(\text{LCH})\times J\cup \text{LCH}\times\{1\}\arrow[r, "{\mathbb{D}_{\C}\cup\C\text{o}\spaces\text{hv}_*(-;\C)}"] & \infcat.
			\end{tikzcd}
			$$
			which admits a lifting
			$$
			\begin{tikzcd}
				Ob(\text{LCH})\times J\cup \text{LCH}\times\{1\} \arrow[r] \arrow[d, hook] & \infcat \\
				\text{LCH}\times J \arrow[ru, dotted]                                      &        
			\end{tikzcd}
			$$
			since the vertical arrow is a categorical anodyne extension. Thus we get the desired functor by restricting the dotted arrow to $\text{LCH}\times\{0\}$.
		\end{proof}
		
		\begin{lemma}\label{properpshpropermap}
			There exists a natural transformation $\pfp{f}\rightarrow\pf{f}$ which is an equivalence when $f$ is proper.
		\end{lemma}
		\begin{proof}
			By Verdier duality, it suffices to construct a natural tranformation between $\mathbb{D}\pfp{f}\rightarrow\mathbb{D}\pf{f}$ and show it is an equivalence when $f$ is proper, which is the content of \cref{pushforwardandVD}.
		\end{proof}
		\begin{corollary}
			Let $i:Z\hookrightarrow X$ be a closed immersion, $\C$ be any stable bicomplete $\infty$-category. Then the functor $\pbp{i} : \Sh{X}{\C}\rightarrow\Sh{Z}{\C}$ coincides with the one defined in \Cref{expbclosed}. 
		\end{corollary}
		\begin{proof}
			This follows immediately by \cref{properpshpropermap}, since any closed immersion is proper.
		\end{proof}
		\begin{lemma}\label{open}
			Let $j : U\hookrightarrow X$ be an open immersion. Then we have $\pfp{j}\dashv\pb{j}$ or equivalently $\pb{j}\simeq\pbp{j}$.
		\end{lemma}
		\begin{proof}
			By \Cref{openpbVD}, we have a natural equivalence $\mathbb{D}\pb{j}\simeq\mathbb{D}\pbp{j}$, and thus we may conclude.
		\end{proof}
		
		\begin{remark}\label{compactifyproperpushforward}
			A useful consequence of \cref{functpfp}, \cref{open}, \cref{properpshpropermap} and (\ref{factorization}) is that the functor $\pfp{f}^{\C}$ is uniquely determined by the fact that it is right adjoint to $\pb{f}_{\C}$ when $f$ is proper and left adjoint to $\pb{f}_{\C}$ when $f$ is an open immersion. 
		\end{remark}
		
		From \cref{compactifyproperpushforward} we also see that $\pfp{f}$ behaves well with respect to tensor products in $\Cocont_{\infty}^{st}$.
		
		\begin{corollary}\label{lowershriekpushtens}
			Let $f:X\rightarrow Y$ be any continuous map. Then there is a commutative square
			$$
			\begin{tikzcd}
				\Sh{X}{\spectra}\otimes\C\ar[r, "\eta"]\ar[d, "{\pfp{f}^{\spectra}\otimes\C}"] & \Sh{X}{\C}\ar[d, "{\pfp{f}^{\C}}"] \\
				\Sh{Y}{\spectra}\otimes\C\ar[r, "\eta"] & \Sh{Y}{\C}.
			\end{tikzcd}
			$$
			
		\end{corollary}
		\begin{proof}
			This follows immediately by \cref{compactifyproperpushforward}, \cref{proppshtens} and \cref{openrestrtens}. 
		\end{proof}
		
		\begin{remark}
			Let $j: U\hookrightarrow X$ be the inclusion of any open subset with compact closure. Then a simple computation involving \Cref{properpshpropermap} and the closure of $U$ shows that, for any sheaf $F$ on $X$, one has
			$$\Sec{U}{F}\simeq\cSec{X}{\pf{j}\pb{j}F}.$$
			Since $U$ has compact closure, any closed subset of $U$ can be written as the intersection of $U$ with some compact subset of $X$, and thus we obtain
			\begin{align*}
				\cSec{X}{\pf{j}\pb{j}F}&\simeq \varinjlim_{K\subseteq X}\KSec{U}{F}{U\cap K} \\
				&\simeq	\varinjlim_{S\subseteq U}\KSec{U}{F}{S},
			\end{align*}
			where the last colimit ranges over all closed subsets of $U$.
		\end{remark}
		
		\begin{proposition}[Base change]\label{basechange}
			For every given pullback square
			$$\begin{tikzcd}
				X' \arrow[r, "f'"] \arrow[d, "g'"'] 
				\arrow[dr, phantom, "\usebox\pullback" , very near start, color=black]
				& X \arrow[d, "g"] \\
				Y' \arrow[r, "f"] & Y
			\end{tikzcd}$$
			of topological spaces, there is a natural equivalence
			$$\pb{f}\pfp{g}\simeq\pfp{g'}\pb{f'}$$
			and, by transposition, also
			$$\pbp{g}\pf{f}\simeq\pf{f'}\pbp{g'}.$$
		\end{proposition}
		\begin{proof}
			By (\ref{stonecech}), we have $g = p\o j$, where $j$ is an open immersion and $p$ is proper. Decompose the given pullback diagram as
			$$
			\begin{tikzcd}
				X' \arrow[d, "j'"', hook] \arrow[r, "f'"]                 & X \arrow[d, "j", hook]      \\
				\overline{X}' \arrow[d, "p'"'] \arrow[r, "\overline{f}'"] & \overline{X} \arrow[d, "p"] \\
				Y' \arrow[r, "f"]                                         & Y.                          
			\end{tikzcd}
			$$
			By \cref{open} and \cref{properpshpropermap}, we have natural transformations
			\begin{equation}\label{openbasechangetransf}
				\pfp{j'}\pb{f'}\rightarrow\pb{\overline{f}'}\pfp{j},
			\end{equation}
			\begin{equation}\label{properbasechangetransf}
				\pb{f}\pfp{p}\rightarrow\pfp{p'}\pb{\overline{f}'},
			\end{equation}
			and therefore a zig zag of natural transormations
			$$\pb{f}\pfp{g}\simeq\pb{f}\pfp{p}\pfp{j}\rightarrow\pfp{p'}\pb{\overline{f}'}\pfp{j}\leftarrow\pfp{p'}\pfp{j'}\pb{f'}\simeq\pfp{g'}\pb{f'}.$$ Hence, to get the desired invertible natural transformation, it will suffice to prove that (\ref{openbasechangetransf}) and (\ref{properbasechangetransf}) are invertible.
			
			Since all functors appearing are colimit preserving, by \cref{lowershriekpushtens} it suffices to prove the proposition in the case of sheaves of spectra. The open immersion case follows immediately by \cref{open} and \cref{smoothbc}, while the proper case follows from \cite[Corollary 3.2]{haine2025nonabelian}. For the reader's convenience, let us briefly summarize the strategy of \cite{haine2025nonabelian}.
			
			By \cite[Corollary 7.3.1.18]{lurie2009higher} we know that the statement of the theorem is true for sheaves of spaces. Since $\spectra$ is compactly generated, by \cref{tenspres} we 
			\begin{align*}
				\Sh{X}{\spectra}&\simeq\Fun_{\ast}(\spectra\op, \Shsp{X}) \\
				& \simeq \Fun_{lex}((\spectra^{\omega})\op, \Shsp{X})
			\end{align*}
			where $\spectra^{\omega}$ denotes the full subcatgory of $\spectra$ spanned by all compact objects. One checks easily that, for any continuous map $h: W\rightarrow T$ of topological spaces, there is a commutative square 
			$$
			\begin{tikzcd}
				\Sh{W}{\spectra}\arrow[r, "\simeq"]\arrow[d, "{\pf{h}^{\spectra}}"] & \Fun_{lex}((\spectra^{\omega})\op, \Shsp{W})\arrow[d, "{\pf{h}^{\spaces}\o(-)}"] \\
				\Sh{T}{\spectra}\arrow[r, "\simeq"] & \Fun_{lex}((\spectra^{\omega})\op, \Shsp{T}),
			\end{tikzcd}
			$$
			where the right hand vertical square denotes a post-composition with the pushforward $\pf{h}^{\spaces}$ of sheaves of spaces. Reasoning analogously to \cref{tensoringadj}, we see that the functor $$\Fun_{lex}((\spectra^{\omega})\op, -)$$ preserves adjunctions between left exact functors, and so we get a similar commutative square involving the pullbacks $\pb{h}_{\spectra}$ and $\pb{h}_{\spaces}$. Hence, we obtain base change for spectral sheaves by applying $\Fun_{lex}((\spectra^{\omega})\op, -)$ to Lurie's nonabelian proper base change in \cite[Corollary 7.3.1.18]{lurie2009higher}.
		\end{proof} 
		
		\begin{remark}\label{tensstablesh}
			It follows from \Cref{presptdst} and the definition of the tensor (\ref{tensshCD}) that for any topological space $X$ and any bicomplete stable $\infty$-category $\C$, $\Sh{X}{\C}$ is tensored over $\Sh{X}{\spectra}$. When there is no possibility of confusion we will denote by $F\otimes G$ the image through the canonical variablewise colimit preserving functor 
			$$\Sh{X}{\C}\times\Sh{X}{\spectra}\rightarrow\Sh{X}{\C}$$ of a pair $(F, G)$, and, when $G\in\Sh{X}{\spectra}$ by $\sHom{X}{G}{F}$ the image of any $F\in\Sh{X}{\C}$ through the right adjoint of $-\otimes G$.
		\end{remark}
		
		Let $f: X\rightarrow X'$ and $g: Y\rightarrow Y'$ be morphisms of topological spaces, and let $f\times g: X\times Y\rightarrow X'\times Y'$ be the induced map on the products. For any two stable bicomplete $\infty$-categories $\C$ and $\D$, the variable-wise colimit preserving functor
		$$
		\begin{tikzcd}
			\Sh{X}{\C}\times\Sh{Y}{\D}\arrow[r, "{\pfp{f}\times\pfp{g}}"] & \Sh{X'}{\C}\times\Sh{Y'}{\D} \arrow[r, "\boxtimes"] & \Sh{X'\times Y'}{\spectra}\otimes(\C\otimes\D)
		\end{tikzcd}
		$$
		induces a functor 
		$$\pfp{f}\boxtimes\pfp{g} : \Sh{X\times Y}{\spectra}\otimes(\C\otimes\D)\rightarrow \Sh{X'\times Y'}{\spectra}\otimes(\C\otimes\D).$$
		\begin{proposition}[K\"{u}nneth formula]\label{kunneth}
			We have a natural equivalence
			$$\pfp{f}\boxtimes\pfp{g}\simeq\pfp{(f\times g)}.$$
			Notice that, since $X$ and $Y$ are locally compact, $\pfp{f}\boxtimes\pfp{g}$ is the image of the pair $(\pfp{f}, \pfp{g})$ through the bifunctor given by the tensor product of cocomplete $\infty$-categories.
		\end{proposition}
		\begin{proof}
			By the factorization (\ref{factorization}), it will suffice to prove the statement when both $f$, $g$ and $f\times g$ are either open immersions or proper maps. By uniqueness of adjoints and \cref{tensoringadj}, both cases will then follow by \cref{tensprodspaces}. More precisely, in the open immersions case we use that $\pfp{f}\boxtimes\pfp{g}$ is left adjoint to $\pb{f}\boxtimes\pb{g}$, while in the proper case we use that $\pfp{f}\boxtimes\pfp{g}$ is right adjoint to $\pb{f}\boxtimes\pb{g}$, and by \cref{tensprodspaces} we always have an equivalence $\pb{f}\boxtimes\pb{g}\simeq\pb{(f\times g)}$.
		\end{proof}
		
		\begin{proposition}[Projection formula]\label{proj}
			Let $f: X\rightarrow Y$ be a continuous map of locally compact Hausdorff spaces, and let $\C$ and $\D$ be two stable and bicomplete $\infty$-categories. Then, for any $F\in\Sh{X}{\C}$ and $G, H\in\Sh{Y}{\D}$, we have a canonical equivalence
			$$\pfp{f}F\otimes G\simeq\pfp{f}(F\otimes\pb{f}G).$$ 
		    When $\C = \D$ has a closed symmetric monoidal structure, by transposition
			$$\pf{f}\sHom{X}{F}{\pbp{f}G}\simeq\sHom{Y}{\pfp{f}F}{G}$$
			and
			$$\pbp{f}\sHom{Y}{G}{H}\simeq\sHom{X}{\pb{f}G}{\pbp{f}H}.$$
		\end{proposition}
		\begin{proof}
			Exactly as for the \cref{smoothproj}, one may deduce this result from \cref{basechange} and \cref{kunneth} applied to $\pfp{(f\times \text{id}_Y)}$.
		\end{proof}
		
		\begin{corollary}\label{0ext}
			Let $k: Z\rightarrow X$ be the inclusion of a locally closed subset of $X$, $F\in\Sh{X}{\C}$ with $\C$ stable and bicomplete. Let $\D$ be another stable bicomplete $\infty$-category, $M\in\D$ any object. Then we have a canonical equivalence
			$$\pfp{k}\pb{k}(F\otimes M)\simeq F\otimes M_Z.$$
			Moreover, when $\C$ has a closed symmetric monoidal structure, we have
			$$\pfp{k}\pb{k}F\simeq F\otimes_{\C} \mathds{1}_Z$$
			or equivalently
			$$\pf{k}\pbp{k}F\simeq\sHom{X}{\mathds{1}_Z}{F},$$
			where $\mathds{1}$ and is the monoidal unit of $\C$.
		\end{corollary}
		\begin{proof}
			This is an immediate consequence of \cref{proj} and \cref{monpb}.
		\end{proof}
		
		\begin{corollary}\label{ex!*}
			Given a pullback square
			$$\begin{tikzcd}
				X' \arrow[r, "f'"] \arrow[d, "g'"'] 
				\arrow[dr, phantom, "\usebox\pullback" , very near start, color=black]
				& X \arrow[d, "g"] \\
				Y' \arrow[r, "f"] & Y
			\end{tikzcd}$$
			of locally compact Hausdorff spaces where $f$ and $f'$ are shape submersions, there is a natural equivalence
			$$\pfp{g}\pfs{f'}\simeq\pfs{f}\pfp{g'}$$
			or equivalently
			$$\pb{f'}\pbp{g}\simeq\pbp{g'}\pb{f}.$$
		\end{corollary}
		\begin{proof}
			First of all we construct a natural transformation. Applying $\pfp{g'}$ to the unit of the adjunction $\pfs{f'}\dashv\pb{f'}$ and using base change we obtain
			$$\pfp{g'}\rightarrow\pfp{g}\pb{f'}\pfs{f}\simeq\pb{f}\pfp{g}\pfs{f'}$$
			and hence by transposition the desired transformation. By factorization (\ref{factorization}) and \Cref{pfsclosed}, we are only left to prove the case of a pullback square
			$$\begin{tikzcd}
				X\times Y' \arrow[r, "\text{id}_X\times f"] \arrow[d, "a\times\text{id}_{Y'}"'] 
				\arrow[dr, phantom, "\usebox\pullback" , very near start, color=black]
				& X\times Y \arrow[d, "a\times\text{id}_Y"] \\
				Y' \arrow[r, "f"] & Y
			\end{tikzcd}$$
			where $X$ is compact and $a:X\rightarrow\ast$ is the unique map, but this follows from \cref{kunneth} and by the functoriality of the tensor product of cocomplete $\infty$-categories as follows
			\begin{align*}
				\pfp{(a\times\text{id}_Y)}\pfs{(\text{id}_X\times f)}&\simeq(\pfp{a}\otimes\pfp{(\text{id}_Y)})(\pfs{(\text{id}_X)}\otimes\pfs{f}) \\
				&\simeq\pfp{a}\pfs{f} \\
				&\simeq(\pfs{(\text{id}_{\ast})}\pfp{a})\otimes(\pfs{f}\pfp{(\text{id}_{Y'})}) \\
				&\simeq(\pfs{(\text{id}_{\ast})}\otimes\pfs{f})(\pfp{a}\otimes\pfp{(\text{id}_{Y'})}).\qedhere
			\end{align*}
		\end{proof}
		
		\begin{remark}\label{mannsix}
			In \cite{mann2022}, the author provides a list of properties for a functor $\text{LCH}\op\rightarrow\infcat$ to extend to a \textit{coherent six functor formalism}, that is a lax symmetric monoidal functor 
			$\text{Corr(LCH)}\rightarrow\infcat$ (see \cite[Definition A.5.9]{mann2022} for more details). We don't spell out these properties in detail here for the sake of brevity, but we refer the reader to \cite[Proposition A.5.10]{mann2022} for a complete list. One can check that our results (\ref{stonecech}), \cref{smoothbc}, \cref{smoothproj}, \cref{pbstablecoeff}, \cref{basechange}, \cref{proj}, \cref{ex!*} imply that all the assumptions of \cite[Proposition A.5.10]{mann2022} are verified for the functor 
			$$
			\begin{tikzcd}[row sep= tiny]
				\text{LCH}\op \arrow[r]              & \infcat     \\
				X \arrow[r, maps to]                 & \Sh{X}{\C}  \\
				(f: X\rightarrow Y) \arrow[r, maps to] & \pb{f}_{\C},
			\end{tikzcd}
			$$
			where $\C$ is any stable and bicomplete $\infty$-category. As a consequence, we deduce that $\Sh{-}{\C}$ to a coherent six functor formalism on the category of locally compact Hausdorff spaces.
		\end{remark}
		
		\subsection{$\pbp{f}_{\C}$ when $f$ is a locally contractible geometric morphism}
		
		Let $f:X\rightarrow Y$ be a continuous map inducing an essential geometric morphism (see \cref{dfnloccontr}). In particular, we have an adjunction $\pfs{f}\dashv\pb{f}$ for $\C\op $-valued sheaves, and thus, after passing to opposite $\infty$-categories and applying \cref{covverddual}, we get an adjunction
		$$\begin{tikzcd}
			\Sh{Y}{\C}\ar[r,bend left,"\pbp{f}",""{name=A, below}] & \Sh{X}{\C}.\ar[l,bend left,"f_{\o}",""{name=B,above}] \ar[from=A, to=B, symbol=\dashv]
		\end{tikzcd}$$
		We now want to show that, in the special case when $f$ induces a locally contractible geometric morphism (see \cref{dfnloccontr}), the exceptional pullback coincides with the usual pullback up to a twist. In this way, we will vastly generalize the classical formula relating the $\pbp{f}$ and $\pb{f}$ when $f$ is a topological submersion.
		
		\begin{proposition}\label{smoothpb}
			Let $f:X\rightarrow Y$ be a continuous map inducing a locally contractible geometric morphism, and let $\C$ and $\D$ be two stable and bicomplete $\infty$-categories. Then, for any $F\in\Sh{Y}{\C}$ and $G\in\Sh{Y}{\D}$, we have a natural equivalence
			$$\pbp{f}_{\C}F\otimes\pb{f}_{\D}G\simeq\pbp{f}_{\C\otimes\D}(F\otimes G)$$
			of functors $\Sh{X}{\C}\times\Sh{Y}{\D}\rightarrow\Sh{X\times Y}{\C\otimes\D}$. When $\C = \D$ and $\C$ has a closed symmetric monoidal structure, we equivalently get a canonical equivalence for any $K\in\Sh{X}{\C}$
			$$\sHom{Y}{F}{f_{\o}K}\rightarrow\pf{f}\sHom{X}{\pbp{f}F}{K}.$$
		\end{proposition}
		\begin{proof}
			For any $F\in\Sh{Y}{\C}$, $G\in\Sh{Y}{\D}$ and $H\in\Sh{Y}{\C\otimes\D}$, we have the map $$\text{counit}\otimes G : \pfp{f}\pbp{f}F\otimes G\rightarrow F\otimes G$$
			which by adjunction and \cref{proj} gives the desired natural transformation 
			$$\pbp{f}(-)\otimes\pb{f}(-)\rightarrow\pbp{f}(-\otimes -).$$
			Since all functors appearing are cocontinuous, by \cref{tensstablesh} it suffices to prove the invertibility of the map when $\C= \D= \spectra$, and after evaluation on pairs of the type $(F, \s{U})$, where $F\in\Sh{Y}{\spectra}$, $\s{}$ denotes the sphere spectrum, and $j:U\hookrightarrow Y$ is an open subset.
			Keeping the same notations as in \cref{opensmoothbasechange}, we see that by \Cref{covverddual} we have an equivalence 
			$$\pbp{f}_{\C}\pfp{j}^{\C}\simeq\pfp{(j')}^{\C}\pbp{(f')}_{\C}.$$
			Thus, we get
			\begin{align*}
				\pbp{f}(F)\otimes\pb{f}(\s{U})&\simeq\pbp{f}(\s{Y})\otimes \s{f^{-1}(U)} \\
				& \simeq \pfp{j'}\pbp{(f')}\pb{j}(F) \\
				& \simeq \pbp{f}\pfp{j}\pb{j}(F) \\
				& \simeq \pbp{f}(F\otimes\s{U}),
			\end{align*}
			where the second equivalence follows from \Cref{0ext} and \cref{open}, the third by \cref{opensmoothbasechange} and the fourth again by \Cref{0ext}.
		\end{proof}
		
		\begin{remark}
			We thank Marc Hoyois for pointing out that a result similar to \cref{smoothpb} can be found in \cite[Section 5]{verdier1965dualite}. By adapting our proof of \cref{smoothproj}, one can actually deduce the theorem in \cite{verdier1965dualite} from \cref{smoothpb}.
		\end{remark}
		
		We need the following elementary lemma.
		
		\begin{lemma}\label{invertibilitylocally}
			Let $X$ be any topological space, and let $F\in\Sh{X}{\spectra}$ be a locally constant sheaf with the property that, for any $x\in X$, the stalk $F_x\in\spectra$ is invertible with respect to smash product of spectra. Then $F$ is invertible.
		\end{lemma}
		
		\begin{proof}
			Let $a:X\rightarrow\ast$ be the unique map, and let $M\in\spectra$ be invertible. Since $\pb{a}$ is symmetric monoidal, $\pb{a}M$ is invertible. Now assume that $F$ is as in the statement of the lemma. We want to show that the canonical map $\sHom{X}{F}{\s{X}}\otimes F\rightarrow\s{X}$ is invertible. This can be checked locally on $X$, since restriction to an open commutes both with taking tensor products and internal homs. The proof is then concluded by observing that the assumption for $F$ in the lemma is equivalent to requiring $F$ to be locally of the form $\pb{a}M$ for some invertible spectrum $M$.
		\end{proof}
		
		\begin{proposition}\label{topsubmpb}
			Let $f:X\rightarrow Y$ be a topological submersion of fiber dimension $n$, $\C$ be any stable bicomplete $\infty$-category. Then:
			\begin{enumerate}[(i)]
				\item if $f$ is a trivial submersion, then there is a canonical equivalence $$\pbp{f}\simeq\Sigma^n\pb{f};$$
				\item $\pbp{f}$ preseves locally constant sheaves;
				\item for all $F\in\Sh{Y}{\C}$ there is a canonical equivalence
				$$\pbp{f}\s{Y}\otimes\pb{f}F\simeq\pbp{f}F$$
				or equivalently by adjunction, for every $G\in\Sh{X}{\C}$
				$$\pfs{f}G\simeq\pfp{f}(G\otimes\pbp{f}\s{Y}).$$
			\end{enumerate}
		\end{proposition}
		\begin{proof}
			Since $\pb{f}$ preserves locally constant sheaves, we see that (ii) follows from (i). Hence, we now prove (i). By assumption, $f$ is a projection $p :X\times\mathbb{R}^n\rightarrow X$. By \cref{smoothpb}, it suffices to show that $\pbp{p}\s{X}\simeq\Sigma^n\s{X\times\mathbb{R}^n}$. Since $p$ induces a locally contractible geometric morphism, we have that $\pbp{p}$ preserves colimits. Hence, by \cref{kunneth} and the uniqueness of adjoints, we may assume that $p$ is the unique map $a : \mathbb{R}^n\rightarrow\ast$. 
			
			We first show that $\pbp{p}M$ is locally constant for each $M\in\spectra$. By a standard argument (see \cite[Proposition 3.1]{haine2020homotopy} or \cite[Proposition 3.14]{volpe2025verdier}) it suffices to show that for any $U\subseteq\mathbb{R}^n$ euclidean chart, the restriction
			$$\Sec{\mathbb{R}^n}{\pbp{p}M}\rightarrow\Sec{U}{\pbp{p}M}$$ is an equivalence. By adjunction, we know that for any open $V$ there is an equivalence
			$$\Sec{V}{\pbp{p}M}\simeq\sHom{\spectra}{\cSec{V}{\s{\mathbb{R}^n}}}{M}.$$
			Thus, it will suffice to prove that the canonical map $$\cSec{U}{\s{\mathbb{R}^n}}\rightarrow\cSec{\mathbb{R}^n}{\s{\mathbb{R}^n}}$$ is invertible. Since any compact subset of a vector space is contained in some compact closed ball, it suffices to show that, for any $K$ compact closed ball in $U$, the map $$\KSec{U}{\s{\mathbb{R}^n}}{K}\rightarrow\KSec{\mathbb{R}^n}{\s{\mathbb{R}^n}}{K}$$ is invertible. But this follows from homotopy invariance of the shape \cref{hmtpyinv}, since the inclusion $U\setminus K\hookrightarrow\mathbb{R}^n\setminus K$ is a homotopy equivalence. In particular, we also see that the canonical map
			$$\KSec{\mathbb{R}}{\s{\mathbb{R}}}{\{0\}}\rightarrow\cSec{\mathbb{R}}{\s{\mathbb{R}}}$$ 
			is an equivalence.
			
			Since all locally constant sheaves on $\mathbb{R}^n$ are constant (see \cite[Corollary 4.13]{haine2020homotopy}), to conclude the proof of (i) it suffices to check that the global sections of $\pbp{p}M$ are equivalent to $\Sigma^n M$. We start with the case $n=1$. Arguing as above, this amounts to proving that we have an equivalence $$\KSec{\mathbb{R}}{\s{\mathbb{R}}}{\{0\}}\simeq\Omega\s{}.$$
			By \cref{hmtpyinv}, we may identify $\KSec{\mathbb{R}}{\s{\mathbb{R}}}{\{0\}}$ with the fiber of the diagonal map
			$$\Sec{\mathbb{R}}{\s{\mathbb{R}}}\rightarrow\Sec{\mathbb{R}}{\s{\mathbb{R}}}\oplus\Sec{\mathbb{R}}{\s{\mathbb{R}}}.$$
			It is then a straightforward exercise in pasting pushouts to verify that, for any object $A$ of a stable $\infty$-category, the fiber of the diagonal map is equivalent to $\Omega A$. For $n>1$, we see that by \cref{kunneth} we have
			$$\cSec{\mathbb{R}^n}{\s{\mathbb{R}^n}}\simeq\cSec{\mathbb{R}}{\s{\mathbb{R}}}\otimes\cdots\otimes\cSec{\mathbb{R}}{\s{\mathbb{R}}},$$
			and thus what we wanted.
			
			To conclude, we observe that (iii) follows from \Cref{smoothpb} and \cref{invertibilitylocally}.
		\end{proof}
		
		\section{Relative Atiyah Duality}
		Recall that, if $X$ is a compact smooth manifold, then the \textit{Atiyah duality} states that $\infsusp_{+}X$ is strongly dual to the \textit{Thom spectrum} associated to the virtual vector bundle given by the inverse of the tangent bundle of $X$, which is denoted by $\text{Th}(-TX)$. In this setion we will revisit Atiyah duality using the six functor formalism, in the spirit of motivic homotopy theory. 
		
		By what we have achieved up to now, we can see very easily that whenever $f:X\rightarrow Y$ is a proper map inducing a locally contractible geometric morphism, $\pfs{f}(\s{X})\in\Sh{Y}{\spectra}$ is strongly dualizable with dual $\pfp{f}(\s{X})$. The question is then about identifying $\pfp{f}(\s{X})$ with some sheaf theoretic construction reminiscent of Thom spectra, at least in geometric situations (e.g. when $X$ and $Y$ are manifolds). We will provide a natural trasformation 
		$$\text{Th}:\text{Vect}(X)\rightarrow\text{Pic}(\Sh{X}{\spectra})$$
		where the left-hand side denotes the $\infty$-groupoid of real vector bundles over $X$, the right-hand side is the $\infty$-groupoid of invertible sheaves of spectra, and $X$ ranges through all paracompact Hausdorff spaces. By observing that the left-hand side is a constant sheaf on the site of all paracompact Hausdorff spaces, the natural transformation above will automatically be induced by a functor
		$$\text{Vect}(\ast)\rightarrow\text{Pic}(\spectra)$$
		given by one-point compactification followed by infinite suspension. After that, we will show that for any vector bundle $E$, $\text{Th}(E)$ can be described through six operations analogously to the definition in motivic homotopy theory. The advantage of our perspective on the definition of $\text{Th}(E)$ is that it makes the verification of all its expected properties, such as compatibility with pulling back vector bundles or short exact sequences, essentially trivial. We will then conclude the section by showing that for a submersion $f$ between smooth manifolds, there is an equivalence
		$$\pbp{f}(\s{Y})\simeq\text{Th}(-T_f)$$
		and hence obtaining a generalization of Atiyah duality. 
		
		\subsection{Thom spaces and the J-homomorphism} 
		
		Let $\text{Vect}$ be the $\infty$-groupoid obtained by taking the coherent nerve of the topological groupoid whose objects are finite dimensional real vector spaces, and morphisms are spaces of linear isomorphisms between. Here we consider $\text{GL}_n(\mathbb{R})$ equipped with the usual topology of a manifold (or equivalently the compact-open topology). As usual, one may equip Vect with a symmetric monoidal structure given by sum, and hence we may regard it as an object of $\text{CMon}(\spaces)$. Here $\text{CMon}(\spaces)$ denotes the $\infty$-category of commutative monoids in $\spaces$ (see \cite[Definition 2.4.2.1, Remark 2.4.2.2]{lurie2017higher}). Notice that we have an homotopy equivalence $\text{Vect}\simeq\coprod\limits_{n\in\mathbb{N}}\text{BGL}_n(\mathbb{R})$.
		
		Let $\text{PH}$ be the category of paracompact Hausdorff spaces with continuous maps between them. Throughout this section, we will assume all topological spaces appearing to be in PH. Recall also that $\text{PH}$ can be equipped with a Grothendieck topology with the usual open coverings. Let $\mathbf{U}$ be a universe such that PH is $\mathbf{U}$-small, and denote by $\spaces'$ the $\infty$-category of $\mathbf{U}$-small spaces. Hence, the inclusion of $\Sh{\text{PH}}{\spaces'}\hookrightarrow\Fun(\text{PH}\op, \spaces')$ admits a left adjoint.
		\begin{definition}
			We denote by $$\text{Vect}_{\text{PH}} : \text{PH}\op\rightarrow\text{CMon}(\spaces')$$ the sheafification of the constant presheaf with value $\text{Vect}$.
		\end{definition} 
		The next lemma will justify our choice of notation.
		
		\begin{lemma}\label{antonio}
			There is an equivalence
			$$\text{Vect}_{\text{PH}}(X)\simeq \text{Sing}(\text{Map}(X,|\text{Vect}|)),$$
			which is natural on $X\in\text{PH}$. Here $\text{Map}$ denotes the mapping space equipped with the compact-open topology, and $|\text{Vect}|$ is a choice of a CW-complex whose weak homotopy type is equivalent to Vect. In particular, $\pi_0(\text{Vect}_{\text{PH}}(X))$ is in bijection with the set of equivalence classes of real vector bundles over $X$.
		\end{lemma}
		
		\begin{proof}
			The first statement follows from \cite[Corollary 7.1.4.4]{lurie2009higher} and \cite[Proposition 7.1.5.1]{lurie2009higher} (see also the discussion before \cite[Proposition 7.1.5.1]{lurie2009higher} for further clarification). The last part of the statement is standard: see for example \cite[Theorem 1.16]{hatcher2003vector}.
		\end{proof}
		Consider the map
		$$\text{GL}_n(\mathbb{R})\rightarrow\Hom{\spaces_{\ast}}{\text{S}^n}{\text{S}^n}\simeq\Omega^n\text{S}^n$$
		given by the functoriality of the one-point compactification. The restriction of the map above to $\text{O}_n$ is what's known as the \textit{J-homomorphism}. We thus obtain a functor
		$$\text{Vect}\rightarrow\spaces_{\ast}^{\simeq}$$
		that at the level of objects sends a finite dimensional real vector space $V$ to its one-point compactification $V^{+}$.
		Moreover, this is easily seen to be symmetric monoidal, where the right hand side is equipped with the smash product. By post composing with $\infsusp$, we get $\text{Vect}\rightarrow\spectra^{\simeq}$. Since $V^{+}$ is homeomorphic to a sphere, this factors as $\text{Vect}\rightarrow\text{Pic}(\spectra)$.
		Here $\text{Pic}(\spectra)$ denotes the \textit{Picard $\infty$-groupoid} of $\spectra$, i.e. the full subcategory of $\spectra^{\simeq}$ spanned by the objects which are invertible with respect to the smash product of spectra. 
		
		Let $f:X\rightarrow Y$ be any map of paracompact Hausdorff spaces. Since the pullback functor $\pb{f} : \Sh{Y}{\spectra}\rightarrow\Sh{X}{\spectra}$ is symmetric monoidal, we obtain consequently a functor
		$$
		\begin{tikzcd}[row sep = tiny]
			\text{PH}\op  \arrow[r] & \text{CMon}(\spaces) \\
			X\arrow[r, mapsto] & \text{Pic}(\Sh{X}{\spectra})
		\end{tikzcd}
		$$
		whose global sections are $\text{Pic}(\Sh{\ast}{\spectra})\simeq\text{Pic}(\spectra)$. Since taking spectrum objects and Picard $\infty$-groupoids both commute with limits, such functor is a sheaf. Thus, \cref{antonio} yields a map
		\begin{equation}\label{stJshvect}
			\text{Th}:\text{Vect}_{\text{PH}}\rightarrow\text{Pic}(\Sh{-}{\spectra}).
		\end{equation}
		in $\Sh{\text{PH}}{\text{CMon}(\spaces)}$. 
		
		\begin{definition}
			For each vector bundle $E\rightarrow X$, we define $\text{Th}(E)$ to be the image of $E$ under the morphism (\ref{stJshvect}).
		\end{definition}
		
		\begin{remark}\label{splitandinvert}
			Denote by $\text{CMon}^{gp}(\spaces)$ the full subcategory of $\text{CMon}(\spaces)$ spanned by those commutative monoids $M$ such that $\pi_0(M)$ is a group. The inclusion $\text{CMon}^{gp}(\spaces)\hookrightarrow\text{CMon}(\spaces)$ admits a left adjoint, denoted by $(-)^{gp}$, which is called the \textit{group completion}. Since $\pi_0(\text{Pic}(\Sh{X}{\spectra}))$ is a group, for any $X\in\text{PH}$ we obtain a morphism
			$$\text{Vect}_{\text{PH}}(X)^{gp}\rightarrow \text{Pic}(\Sh{X}{\spectra}).$$
			In particular we see that, for any vector bundle $E\rightarrow X$ it makes sense to define $\text{Th}(-E)$ as the tensor inverse of $\text{Th}(E)$, where $-E$ denotes the inverse of the class of $E$ in $\pi_0(\text{Vect}_{\text{PH}}(X))^{gp}$. Moreover, for any split exact sequence
			$$0\rightarrow E\rightarrow V \rightarrow E'\rightarrow 0$$
			we get an equivalence
			$$\text{Th}(V)\simeq\text{Th}(E)\otimes\text{Th}(E').$$
		\end{remark}
		
		Out next goal is to describe $\text{Th}(E)$ in terms of the six operations. Let $p:E\rightarrow X$ be a real vector bundle over $X$, and denote by $s:X\hookrightarrow E$ its zero section and by $j:E^{\times}\hookrightarrow E$ its open complement. Consider the sheaf
		\begin{equation}\label{thomascofiber}
			\pfs{p}\text{cofib}(\pfs{j}\pb{j}\s{E}\rightarrow\s{E})\in\Sh{X}{\spectra}
		\end{equation}
		where the morphism $\pfs{j}\pb{j}\s{E}\rightarrow\s{E}$ is the counit. Notice that $\pfs{p}$ exists since any vector bundle is obviously a shape submersion, so indeed the definition above makes sense. Notice also that, using \cref{localizationtheorem}, one has $\pfs{p}\text{cofib}(\pfs{j}\pb{j}\s{E}\rightarrow\s{E})\simeq\pfs{p}\pfp{s}\s{X}$.
		
		We will need the following lemma.
		
		\begin{lemma}\label{lchpb}
			Let 
			$$
			\begin{tikzcd}
				X'\arrow[r] \arrow[d] & X \arrow[d, "f"] \\
				Y'\arrow[r] & Y
			\end{tikzcd}
			$$
			be any pullback square of topological spaces. Assume that $f$ is a fiber bundle projection, with fiber given by a locally compact space $F$. Then the corresponding diagram 
			$$
			\begin{tikzcd}
				\Shsp{X'}\arrow[r] \arrow[d] & \Shsp{X} \arrow[d] \\
				\Shsp{Y'}\arrow[r] & \Shsp{Y}
			\end{tikzcd}
			$$
			is a pullback of $\infty$-topoi.
		\end{lemma}
		
		\begin{proof}
			%It suffices to prove the lemma separately in the two cases when $f$ is an open immersion or a fiber bundle projection. The case of an open immersion is treated in \cite[Remark 6.3.5.8]{lurie2009higher}. Assume now that $f$ is a fiber bundle with locally compact fiber $F$. 
			Consider the geometric morphism $g:\Shsp{X'}\rightarrow\Shsp{Y'}\times_{\Shsp{Y}}\Shsp{X}$ and let $\{U_i\}_{i\in I}$ be a covering sieve of $Y$ along which $f$ is trivialized. Using descent, and the fact the left adjoint in a geometric morphism preserves effective epimorphisms, we see that it suffices to show that $g$ is an equivalence after pulling back along $\Shsp{U_i}\rightarrow\Shsp{Y}$ for each $i\in I$. Therefore, we may assume that $f$ is the standard projection $Y\times F\rightarrow Y$. Hence, we have $X'=Y'\times F$. Consider the squares in $\Top$
			 	$$
			 \begin{tikzcd}
			 	\Shsp{Y'\times F}\arrow[r] \arrow[d] & \Shsp{Y\times F} \arrow[d] \arrow[r] & \Shsp{F} \arrow[d] \\
			 	\Shsp{Y'}\arrow[r] & \Shsp{Y} \arrow[r] & \Shsp{\ast}.
			 \end{tikzcd}
			 $$
			 By \cite[Proposition 7.3.1.11]{lurie2009higher}, the outer square and the right square are both pullbacks. Therefore, the left square is a pullback too, and so the proof is concluded.
		\end{proof}
		
		\begin{proposition}\label{natthom}
			The assignment \begin{equation}\label{thomascofibernat}
				E\mapsto\pfs{p}\text{cofib}(\pfs{j}\pb{j}\s{E}\rightarrow\s{E})\in\Sh{X}{\spectra}
			\end{equation} defines a natural transformation
			$$\text{Vect}_{\text{PH}}\rightarrow\text{Pic}(\Sh{-}{\spectra})$$
			which is equivalent to (\ref{stJshvect}).
		\end{proposition}
		\begin{proof}
			First of all, we show that (\ref{thomascofibernat}) induces a natural transformation
			\begin{equation}\label{Thnat}
				\text{Vect}_{\text{PH}}\rightarrow\Sh{-}{\spectra}
			\end{equation}
			of presheaves of $\infty$-categories. Let $p: E\rightarrow X$ be a vector bundle, $p^{\times}: E^{\times}\rightarrow X$ be the induced map on the complement of the zero section. It will suffice to show that the associations $E\mapsto \pfs{p}\s{E}$ and $E\mapsto\pfs{p^{\times}}\s{E^{\times}}$ induce natural transformations. Consider a pullback square in $\text{Top}$ 
			$$\begin{tikzcd}
				\pb{f}E \arrow[r, "f'"] \arrow[d, "p'"'] 
				\arrow[dr, phantom, "\usebox\pullback" , very near start, color=black]
				& E \arrow[d, "p"] \\
				X' \arrow[r, "f"] & X
			\end{tikzcd}$$
			where $p$ is a vector bundle. By \cref{lchpb} 
			$$\begin{tikzcd}
				\Shsp{\pb{f}E} \arrow[r, "f'"] \arrow[d, "p'"'] 
				& \Shsp{E} \arrow[d, "p"] \\
				\Shsp{X'} \arrow[r, "f"] & \Shsp{X}
			\end{tikzcd}$$
			is a pullback square in $\Top$. Similarly one has that the square
			$$\begin{tikzcd}
				\Shsp{\pb{f}E^{\times}} \arrow[r, "f'"] \arrow[d, "p'^{\times}"'] 
				& \Shsp{E^{\times}} \arrow[d, "p^{\times}"] \\
				\Shsp{X'} \arrow[r, "f"] & \Shsp{X}
			\end{tikzcd}$$ is a pullback. Hence we have two natural transformations
			$$\text{Vect}_{\text{PH}}\rightarrow\Top_{/\Shsp{-}}$$
			given respectively by sending a vector bundle $E\rightarrow X$ to $\Shsp{E}\rightarrow\Shsp{X}$ and to $\Shsp{E^{\times}}\rightarrow\Shsp{X}$, and thus, by further composing with the relative shape, by \Cref{laxshape} we obtain lax natural transformations
			$$\text{Vect}_{\text{PH}}\rightarrow\Pro(\Shsp{-})$$
			which factors as 
			$$\text{Vect}_{\text{PH}}\rightarrow\Shsp{-}$$
			since any shape submersion induces a locally contractible geometric morphism by \cref{smoothproj}.
			Furthermore, by \cref{smoothbc}, we see that these are actually natural transformations, and thus, composing with
			$$\Shsp{-}\rightarrow\Sh{-}{\spectra}$$
			we get the natural transformation (\ref{Thnat}).
			\par 
			Since $\text{Vect}_{{\text{PH}}}$ is the constant sheaf associated to $\text{Vect}$, for any sheaf $F\in\Sh{\text{PH}}{\text{CMon}(\spaces)}$, we have 
			$$
			\begin{tikzcd}
				\Hom{}{\text{Vect}_{\text{PH}}}{\text{Pic}(F)}\arrow[r, "\simeq"] & \Hom{\text{CMon}}{\text{Vect}}{\text{Pic}(F(\ast))}.
			\end{tikzcd}
			$$
			Therefore, to conclude the proof it will suffice to show that (\ref{Thnat}) agrees with Th after taking global sections. But this follows from observing that, for any real vector space $E$, the cofiber of the inclusion $E\setminus\{0\}\hookrightarrow E$ is homotopy equivalent to $\overline{E}$.
		\end{proof}
		
		\begin{remark}\label{classicthom}
			Assume that $X$ is essential, and let $a:X\rightarrow\ast$ be the unique map. Then, for any vector bundle $E$ over $X$, $\pfs{a}\text{Th}(E)$ is equivalent to the Thom spectrum of $E$ (see \cite{thom1954quelques} or \cite{atiyah1961thom}). Indeed, let $b: E\rightarrow\ast$ and $c:E^{\times}\rightarrow\ast$ be the unique maps. By \Cref{loccontrhyp}, we have equivalences
			\begin{align*}
				\pfs{a}\text{Th}(E) &\simeq\pfs{b}\text{cofib}(\pfs{j}\pb{j}\s{E}\rightarrow\s{E}) \\
				&\simeq\text{cofib}(\pfs{c}\s{E^{\times}}\rightarrow\pfs{b}\s{E}) \\
				&\simeq\text{cofib}(\infsusp_{+}E^{\times}\rightarrow\infsusp_{+}E)
			\end{align*}
			and the spectrum on the last line coincides with the Thom spectum of $E$.
		\end{remark}
		
		\begin{corollary}\label{thominverse}
			Let $p:E\rightarrow X$ be a real vector bundle over $X$, and denote $s:X\hookrightarrow E$ its zero section. Then $\text{Th}(E)$ is invertible with inverse given by $\pbp{s}\s{E}$.
		\end{corollary}
		\begin{proof}
			By definition, we already know that $\text{Th}(E)$ is invertible, and thus to compute its inverse we just need to look at its dual $\sHom{X}{\text{Th}(E)}{\s{X}}$. Then we may conclude by \cref{smoothproj} and \cref{proj} since
			\begin{align*}
				\sHom{X}{\pfs{p}\pfp{s}\s{X}}{\s{X}}&\simeq\pf{p}\sHom{E}{\pfp{s}\s{X}}{\s{E}} \\
				&\simeq\pf{p}\pf{s}\pbp{s}\s{E} \\
				&\simeq\pbp{s}\s{E}.\qedhere
			\end{align*}
		\end{proof}
		\subsection{Relative Atiyah duality}
		Let $f:X\rightarrow Y$ be a submersion of smooth manifolds. Recall that the \textit{relative tangent bundle of $f$}, denoted by $T_f$, is the vector bundle over $X$ defined by the short exact sequence 
		\begin{equation}\label{reltgt}
			0\rightarrow T_f\rightarrow TX\rightarrow\pb{f}TY\rightarrow0.
		\end{equation}
		\begin{theorem}
			Let $f:X\rightarrow Y$ be a submersion between smooth manifolds. Then we have an equivalence $\pbp{f}\s{Y}\simeq\text{Th}(T_f)$.
		\end{theorem}
		\begin{proof}
			First of all we prove the case when $f$ is the unique map $a:X\rightarrow\ast$ and $X$ is a smooth manifold. Choose a closed embedding $i: X\hookrightarrow\mathbb{R}^n$, and let $a' : \mathbb{R}^n\rightarrow\ast$ be the unique map. By \Cref{topsubmpb}, we have $\pbp{a'}\s{}\simeq\Sigma^n\s{\mathbb{R}^n}$ and thus, since $T\mathbb{R}^n$ is a trivial vector bundle of fiber dimension $n$, we have $\text{Th}(T\mathbb{R}^n)\simeq\Sigma^n\s{\mathbb{R}^n}\simeq\pbp{a'}\s{}$. Let $p:N_i\rightarrow X$ be the conormal bundle of the embedding $i$, defined by the short exact sequence 
			\begin{equation}\label{conbundle}
				0\rightarrow TX\rightarrow\pb{i}T\mathbb{R}^n\rightarrow N_i\rightarrow0,
			\end{equation} $s : X\hookrightarrow N_i$ its zero section. Let $k:U\hookrightarrow N_i$ be a tubular neighbourhood of $X$ in $\mathbb{R}^n$. Thus, we get a commutative triangle
			$$
			\begin{tikzcd}
				& X \arrow[ld, "\tilde{s}"', hook] \arrow[rd, "i", hook] &              \\
				U \arrow[rr, "g", hook] &                                                        & \mathbb{R}^n
			\end{tikzcd}
			$$ 
			where $g$ is an open immersion and $\tilde{s}$ is a closed immersion. Hence we get an equivalence
			$$\pbp{i}\s{\mathbb{R}^n}\simeq\pbp{\tilde{s}}\s{U}.$$ 
			Then, by \Cref{topsubmpb}, (\ref{conbundle}) and \cref{thominverse} we have
			\begin{align*}
				\pbp{a}\s{}&\simeq\pbp{i}\pbp{a'}\s{} \\
				&\simeq\pbp{i}\s{}\otimes\pb{i}\text{Th}(T\mathbb{R}^n) \\
				&\simeq\text{Th}(N_i)^{-1}\otimes\text{Th}(\pb{i}T\mathbb{R}^n)\\
				&\simeq\text{Th}(TX).
			\end{align*}
			Suppose now that $f:X\rightarrow Y$ is any submersion between smooth manifolds, $a:X\rightarrow \ast$ and $b:Y\rightarrow\ast$ be the unique maps. Then, by \Cref{topsubmpb} and \cref{splitandinvert}, we have
			\begin{align*}
				\text{Th}(T_f)&\simeq\text{Th}(TX)\otimes\text{Th}(\pb{f}TY)^{-1} \\
				&\simeq\pbp{a}\s{}\otimes(\pb{f}\pbp{b}\s{})^{-1} \\
				&\simeq\pbp{a}\s{}\otimes(\pbp{f}\pbp{b}\s{})^{-1}\otimes\pbp{f}\s{Y} \\
				&\simeq\pbp{f}\s{Y}
			\end{align*}
			and thus we can conclude.
		\end{proof}
		\begin{corollary}[Relative Atiyah Duality]\label{relatiyahdual}
			Let $f:X\rightarrow Y$ be a proper map inducing a locally contractible geometric morphism. Then $\pfs{f}\s{X}\in\Sh{Y}{\spectra}$ is strongly dualizable with dual $\pfp{f}\s{X}$. Moreover, if $X$ and $Y$ are smooth manifolds and $f$ is a proper submersion, then  $\pfs{f}\s{X}$ is strongly dualizable with dual $\pfs{f}\text{Th}(-T_f)$.
		\end{corollary}
		\begin{proof}
			Since $f$ is proper, by \cref{smoothproj} and \cref{proj}, we have, functorially on $F\in\Sh{Y}{\spectra}$
			\begin{align*}
				\sHom{Y}{\pfs{f}\s{X}}{F}&\simeq\pf{f}\sHom{X}{\s{X}}{\pb{f}F} \\
				&\simeq\pf{f}\pb{f}F \\
				&\simeq\pfp{f}\pb{f}F \\
				&\simeq\pfp{f}\s{X}\otimes F.
			\end{align*}
			In particular, when $f$ is a submersion of smooth manifolds, by \Cref{topsubmpb} and the previous theorem, we have $\pfp{f}\s{X}\simeq\pfs{f}\text{Th}(-T_f)$.
		\end{proof}
		\begin{remark}
			Let $X$ be a smooth manifold, $a:X\rightarrow\ast$ the unique map. By specializing the previous corollary to $a$ and \Cref{classicthom}, we see that we recover the classical Atiyah duality (see \cite{atiyah1961thom}).
		\end{remark}

		\newpage
		\nocite{*}
		\bibliographystyle{alpha}
		\bibliography{supports}
	\end{document}